\documentclass[11pt,oneside]{amsart}

\usepackage{amsmath,ifthen, amsfonts, amssymb,srcltx,amsopn, color, enumerate}
\usepackage[cmtip,arrow]{xy}
\usepackage{pb-diagram, pb-xy}

\usepackage{overpic,mathabx}

\usepackage[T1]{fontenc}

\usepackage{todonotes}

\usepackage[breaklinks,hidelinks,unicode]{hyperref}

\newcommand{\showcomments}{yes}
\renewcommand{\showcomments}{no}

\newsavebox{\commentbox}
{\ifthenelse{\equal{\showcomments}{yes}}%
{\footnotemark
        \begin{lrbox}{\commentbox}
        \begin{minipage}[t]{1.25in}\raggedright\sffamily\tiny
        \footnotemark[\arabic{footnote}]}
{\begin{lrbox}{\commentbox}}}%
{\ifthenelse{\equal{\showcomments}{yes}}%
{\end{minipage}\end{lrbox}\marginpar{\usebox{\commentbox}}}
{\end{lrbox}}}

\newcounter{intronum}
\newcounter{ax}
\newtheorem{thm}{Theorem}[section]

\newtheorem{lem}[thm]{Lemma}

\newtheorem{cor}[thm]{Corollary}

\newtheorem{prop}[thm]{Proposition}

\theoremstyle{definition}
\newtheorem{defn}[thm]{Definition}
\newtheorem{rem}[thm]{Remark}
\newtheorem{remi}[intronum]{Remark}
\newtheorem{exmp}[thm]{Example}

\newtheorem{notation}[thm]{Notation}
\newtheorem{conv}[thm]{Convention}

\newtheorem{prob}[thm]{Problem}
\newtheorem{question}[thm]{Question}
\newtheorem{claim}{Claim}

\newtheorem{claim*}{Claim}
\newtheorem*{obs}{Observation}
\newtheorem{cons}[thm]{Construction}

\DeclareMathOperator{\kernel}{ker}
\DeclareMathOperator{\image}{im}
\DeclareMathOperator{\rank}{rk}

\DeclareMathOperator{\Aut}{Aut}
\DeclareMathOperator{\Out}{Out}

\DeclareMathOperator{\stabilizer}{Stab}
\DeclareMathOperator{\diam}{diam}

\newcommand{\homology}{\ensuremath{{\sf{H}}}}

\newcommand{\field}[1]{\mathbb{#1}}
\newcommand{\integers}{\ensuremath{\field{Z}}}

\newcommand{\naturals}{\ensuremath{\field{N}}}
\newcommand{\reals}{\ensuremath{\field{R}}}
\newcommand{\Euclidean}{\ensuremath{\field{E}}}

\newcommand{\closure}[1]{Cl\left({#1}\right)}

\newcommand{\interior} [1] {{\ensuremath \text{\rm Int}(#1) }}

\makeatletter

\newcommand{\Rmnum}[1]{\mathbf{{\expandafter\@slowromancap\romannumeral #1@}}}

\newcommand{\lcm}{\ensuremath{\mathrm{l}\mathrm{c}\mathrm{m}}}

\makeatother

\let\oldmarginpar\marginpar
\renewcommand\marginpar[1]{\-\oldmarginpar[\raggedleft\footnotesize #1]%
{\raggedright\footnotesize #1}}

\newcommand{\nclose}[1]{\langle\langle#1\rangle\rangle}

\newcommand{\actson}{\curvearrowright}

\newcounter{enumitemp}

\newcommand{\dist}{\textup{\textsf{d}}}

\newcommand{\OL}{\overleftarrow}
\newcommand{\OR}{\overrightarrow}

\newcommand{\gate}{\mathfrak g}

\newcommand{\vertices}{\mathrm{Vert}}
\newcommand{\edges}{\mathrm{Edge}}

\newcommand{\axis}{\mathrm{Axis}}

\newcommand*\colvec[2]{\ensuremath{\begin{pmatrix}#1\\#2\\\end{pmatrix}}}

\long\def\Restate#1#2#3#4{
\medskip\par\noindent
{\bf #1 \ref{#2} #3} {\it #4}\par\medskip }

\newcommand{\ttm}{\mathrm{RTT}}
\newcommand{\Fast}{\mathrm{Fast}}

\begin{document}
\title[Cubulating polynomial mapping tori]{Cubulating 
mapping tori of some polynomial growth free group automorphisms}

\author[M.~Hagen]{Mark Hagen}
\address{School of Mathematics, University of Bristol, Bristol, UK}
\email{markfhagen@posteo.net}

\author[D.T.~Wise]{Daniel T. Wise}
\address{Department of Mathematics, Weizmann Institute of Science, Rehovot, Israel}
\email{daniel.wise@weizmann.ac.il}

\date{\today}

\maketitle

\begin{abstract}
 Let $F$ be a finite-rank free group and let $\Phi\in\Out(F)$ have polynomial growth.  Let $G=F\rtimes_\Phi\integers$.  We give sufficient conditions on $\Phi$ that ensure $G$ acts freely on a CAT(0) cube complex.  For $d=1$, the class of $G$ that we cubulate strictly contains tubular free-by-cyclic groups, which were cubulated by Button.  For $d>1$, we cubulate $G$ provided, for instance, the linear-growth mapping tori contained in $G$ are tubular and $G$ satisfies a condition on intersections of certain centralisers.  These conditions are satisfied when the growth rate of $\Phi$ is as large as possible for $F$.  Using this, we show that for any fixed $F$, a random unipotent polynomially growing automorphism $\Phi$ has cubulated mapping torus.

 We do not work directly with relative train tracks, but rely on them via the cyclic hierarchy from \cite{Macura:detour} in the superlinear case and the splitting over $\integers^2$ subgroups from \cite{AndrewMartino:splitting,DahmaniTouikan} in the linear case.  Our proof relies on cubical small-cancellation theory to obtain free actions on CAT(0) cube complexes for groups admitting suitable acylindrical cyclic hierarchies whose bottom-level vertex groups are cubulated; this technical result is of independent interest. 
\end{abstract}

\tableofcontents

\section{Introduction}\label{sec:intro}
In this paper, we study free actions of groups $G=F\rtimes_\Phi\integers$ on CAT(0) cube complexes, where $F$ is a finite-rank free group.  Previously, we showed that $G$ admits a proper cocompact action on a CAT(0) cube complex in the case when $\Phi$ is atoroidal \cite{HagenWise:irred,HagenWise:general}, and this result has been generalised, with a streamlined proof, by Dahmani-Krishna MS \cite{DahmaniKrishna:irred} and Dahmani-Krishna MS-Mutanguha \cite{DahmaniKrishnaMutanguha}.  The natural question is which polynomially-growing $\Phi$ have cubulated mapping tori.

A group is \emph{cubulated} if it acts freely on a CAT(0) cube complex.  We do not require the cube complex to be locally finite or finite-dimensional, or the action to be metrically proper.

Mapping tori of polynomially-growing automorphisms are natural from this viewpoint because, for any $\Phi\in\Out(F)$, the mapping torus $G$ is hyperbolic relative to a collection of sub-mapping tori of polynomially growing automorphisms of free groups \cite{DahmaniLi}.  One might imagine cubulating the peripheral subgroups and then combining this with techniques similar to those from \cite{DahmaniKrishna:irred,DahmaniKrishnaMutanguha} to cubulate the whole group.

However, there are limits on how nice the resulting cube complexes and actions can be: when $\Phi$ has polynomial growth, $G$ need not act freely on a finite-dimensional cube complex, as can be seen from Gersten's example, discussed below.  Although the existence of a free action of $G$ on an arbitrary CAT(0) cube complex does have interesting consequences like a-T-menability, we do not expect such actions to reveal deep properties of $G$. But the question of when free actions exist is intriguing and studying it highlights some mysterious features of these groups.

Gersten's group shows that cubulating mapping tori of linear growth automorphisms is an unforgiving task.  In fact, the difficulties with cubulating always reside in the linear-growth mapping tori inside $G$: the main thrust of this paper is that there is a robust method to cubulate $G$ if one can cubulate the linear sub-mapping tori in a strategic fashion.  But the methods we have found for the linear case are more ad hoc.  We suspect that all linear mapping tori act freely on CAT(0) cube complexes.  However, we foresee difficulties in moving from linear to quadratic mapping tori in general.  Indeed, in Section \ref{subsec:intro-counterexample-question}, we discuss a linear mapping torus $G_0$ whose  cubulations lack the property our methodology requires to cubulate quadratic mapping tori containing $G_0$. A general approach to the superlinear case thus demands new techniques.

\subsection{The superlinear case}\label{subsec:intro-results}
Our main result for $\Phi$ with arbitrary polynomial growth rate is:

\Restate{Theorem}{thm:main}{}{Let $F$ be a finite-rank free group, let $\Phi\in\Out(F)$ be a polynomially-growing outer automorphism, and let $G=F\rtimes_\Phi\integers$.  Suppose that both of the following hold:
\begin{itemize}
    \item $G$ is \emph{unbranched} (Definition \ref{defn:unbranched}), and
    \item bottom-level vertex groups in the \emph{superlinear hierarchy} (Remark \ref{rem:topmost-edge-vertex-groups}) are virtually \emph{tubular}.
\end{itemize}
Then $G$ acts freely on a CAT(0) cube complex.}

The \emph{superlinear (virtual) hierarchy} is described in Section \ref{subsec:topmost-edge-splittings}, and originates in work of Macura \cite{Macura:detour}, using improved relative train track representatives of a power of $\Phi$.  After passing to a finite-index subgroup, $G$ admits a hierarchy in which all of the edge groups are infinite cyclic, all of the splittings are acylindrical, and the bottom-level vertex groups have the form $F'\rtimes_{\Phi'}\integers$, where $\Phi'\in \Out(F')$ has growth rate $0$ or $1$.

A group is \emph{tubular} if it splits as a finite graph of groups with $\integers^2$ vertex groups and $\integers$ edge groups.  Cubulated tubular groups were classified in \cite{Wise:tubular}, and this topic was further explored in \cite{Woodhouse:fin-dim,Woodhouse:special}.  In \cite{Button:tubular-cubes}, Button succinctly characterised tubular free-by-$\integers$ groups and used \cite{Wise:tubular} to cubulate them.  Button's result does not follow from Theorem \ref{thm:main} --- tubular free-by-$\integers$ groups need not be unbranched --- but Proposition \ref{prop:general-linear} does recover Button's work.

A group is \emph{unbranched} if it does not have three incommensurable highest  (free)$\times\integers$ subgroups whose mutual intersection contains $\integers^2$; see Definition \ref{defn:unbranched}.  A motivating example is the fundamental group of a graph manifold; being unbranched corresponds to the fact that each JSJ torus is contained in two Seifert fibered pieces. 

This notion comes from \cite{BGGH}, where various coarse-geometric properties of a free-by-cyclic group $G$ are shown to be equivalent to $G$ being unbranched.  For instance, $G$ is unbranched exactly when it is quasi-isometric to a CAT(0) cube complex; in general $G$ does not act on this cube complex.  If $G$ is not unbranched, then, by work of Munro-Petyt \cite{MunroPetyt}, $G$ is not \emph{cocompactly} cubulated, nor coarse median \cite{BGGH,MunroPetyt}, but this does not obstruct cubulation.

\begin{remi}\label{remi:unbranched}
The unbranched property for $G=F\rtimes_\Phi\integers$ holds if and only if it holds for the linear mapping tori at the bottom of the superlinear hierarchy \cite{BGGH}, as a consequence of acylindricity of the hierarchy.  When $\Phi$ is linearly-growing, $G$ being unbranched is determined by the underlying bipartite graph of the canonical splitting of $G$ over $\integers^2$ subgroups constructed in \cite{AndrewMartino:splitting,DahmaniTouikan} (see Proposition \ref{prop:split-over-tori}), and we work with this characterisation of unbranched.
\end{remi}

Theorem \ref{thm:main} has an attractive special case:

\Restate{Corollary}{cor:fast-cubulated}{(Fast cubulation).}{If $\Phi\in\Out(F)$ has polynomial growth and is \emph{fast}, then $G=F\rtimes_\Phi\integers$ acts freely on a CAT(0) cube complex.}

The automorphism $\Phi$ is \emph{fast} if it has polynomial growth rate $d=\rank(F)-1$, which implies that $G$ is unbranched and has tubular bottom-level vertex groups (Lemma \ref{lem:fast-tubular}).

Using relative train track maps, one can define a random UPG automorphism (see Section \ref{subsection:genericity}).  Specifically, let $F=F_r$ with $r\geq 1$.  We then consider homotopy equivalences $f:\Gamma\to\Gamma$, where $\Gamma$ is a rank-$r$ graph and $f$ is an improved relative train track map as provided by \cite[Thm. 5.1.8]{BFH:tits-I}.  The number of edges in $\Gamma$ is bounded in terms of $r$ \cite{BFH:kolchin}, so there are finitely many possibilities for $\Gamma$.  Roughly, $f$ is determined by some immersed closed paths --- \emph{suffixes} --- in $\Gamma$. For each $L\geq 0$, let $R_r(L)$ denote the set of improved relative train track maps $f$ whose suffixes have length at most $L$. Let $F_{\mathcal P}(L)\subset R_r(L)$ be the subset satisfying some property $\mathcal P$ of an improved relative train track map.  Then $\mathcal P$ holds \emph{asymptotically almost surely for UPG automorphisms of $F$} if $|F_{\mathcal P}(L)|/|R_r(L)|\to 1$ as $L\to\infty$. Proposition \ref{prop:fast-generic} and Corollary \ref{cor:fast-cubulated} yield:

\begin{cor}[\textbf{Cubulation is generic}]\label{cori:generic}
A UPG automorphism $\Phi\in\Out(F_r)$ is asymptotically almost surely fast, and hence $G=F\rtimes_\Phi\integers$ is asymptotically almost surely cubulated.
\end{cor}

We now discuss the proof of Theorem \ref{thm:main}.

\subsection{Cubulation using acylindrical cyclic hierarchies}\label{subsec:acyl-cyclic-intro}
Theorem \ref{thm:main} is proved by induction on the length of the superlinear hierarchy.  This hierarchy terminates in mapping tori of linear-growth automorphisms.  We discuss the linear case later in the introduction. We now address the inductive step in the proof of Theorem \ref{thm:main}, so we assume the growth rate is at least $2$.

The inductive step relies on the following statement, which may be useful for cubulating other groups.  We have rephrased it for readability; refer to Section \ref{subsec:acyl-cyclic} for the precise statement.

\Restate{Corollary}{cor:cyclic-splitting-turns}{(Simplified version)}{Let  $G$ split as a finite graph of groups whose edge groups are maximal cyclic subgroups of $G$ and are separable in $G$. Suppose $G$ acts $2$--acylindrically on the Bass-Serre tree $T$.  Suppose each vertex group $G_v$ acts freely on a CAT(0) cube complex $\widetilde X_v$.  Let $\{t_e\}_{e}\subset G_v$ be generators of the edge groups incident to $v$.  Suppose that:
\begin{itemize}
    \item The set $\{t_e\}_e$ is \emph{wall-independent} for the $G_v$--action on $\widetilde X_v$.

    \item There exists $L$ such that for all vertex groups $G_v$ and all $t_e\in G_v$, the combinatorial translation length of $t_e$ on $\widetilde X_v$ is $L$.
\end{itemize}

Let $\mathcal Q\subset G$ be any finite set of $T$--hyperbolic elements such that each $q\in \mathcal Q$ generates a maximal cyclic subgroup of $G$.  Then $G$ acts freely on a CAT(0) cube complex $\widetilde X$ and, moreover:
\begin{enumerate}[(a)]
    \item There exists $L'$ such that each $q\in\mathcal Q$, and each $t_e$, has translation length $L'$ on $\widetilde X$.
    \item The set of all $t_e$, together with $\mathcal Q$, is wall-independent for the $G$--action on $\widetilde X$.
\end{enumerate}}

For a $G$--action on a CAT(0) cube complex $\widetilde X$, a set $\mathcal S\subset G$ is \emph{wall-independent} (Definition \ref{defn:wall-independent}) if distinct conjugates of elements of $\mathcal S$ are cut by only finitely many common hyperplanes.  

Corollary \ref{cor:cyclic-splitting-turns} hypothesises a graph of cubulated groups satisfying two assumptions on a certain finite set of elements, and yields an action of the total group on a cube complex satisfying the same two properties for a larger set.  So Corollary \ref{cor:cyclic-splitting-turns}  applies inductively in the setting of a group with an acylindrical cyclic hierarchy, as in the proof of Theorem \ref{thm:main}, where $G$ is a superlinear polynomial mapping torus, and the graph of groups is the superlinear hierarchy from \cite{Macura:detour}.

The proof proceeds as follows.  Using the hypothesis on translation lengths, Proposition \ref{prop:joining-walls} provides a $G$--action on a CAT(0) cube complex $\widetilde X'$ that admits a cubical quotient map to $T$ such that the vertex spaces contain translates of the various $\widetilde X_v$ as isometrically embedded subcomplexes.  Inside $\widetilde X'$, each $q\in\mathcal Q$ stabilises a convex subcomplex $\widetilde Y_q$ splitting as a line of spaces whose vertex spaces are vertex spaces of $\widetilde X'$.  Elements of $\mathcal Q\cup\{t_e\}_e$ may have different translation lengths in $\widetilde X'$ and fail to be wall-independent, so we will modify $\widetilde X'$.

A hyperplane of $\widetilde X'$ has  two possible behaviours.  Some hyperplanes lie inside some $\widetilde X_v$, or, more generally, project to bounded sets in $T$.  Such hyperplanes are useful because acylindricity of the $G$--action on $T$ bounds the overlaps in $T$ between the images of distinct $\widetilde Y_q$ and $\widetilde Y_{q'}$.  Hence only finitely many such ``benign'' hyperplanes can cross both $\widetilde Y_q$ and $\widetilde Y_{q'}$.  More problematic are hyperplanes that cross $\widetilde Y_{q'}$ in infinitely many vertex spaces.  When $\widetilde X'$ has no such hyperplanes, we say that the hyperplanes in $\widetilde X'$ \emph{exit} $\mathcal Q$.  Producing exiting hyperplanes is the heart of the difficulty with the base case.

In the setting of Corollary \ref{cor:cyclic-splitting-turns}, we use wall-independence hypothesis and acylindricity to refine the construction of $\widetilde X'$ so that all hyperplanes exit $\mathcal Q$.  This is done by modifying the walls using large powers of Dehn twists in the edge-cylinders, and recubulating.  


We then obtain constant translation length and wall-independence using Theorem \ref{thm:cubulation-with-turns}, which also recubulates a graph of groups. Theorem \ref{thm:cubulation-with-turns} allows more general edge groups and arbitrary acylindricity constants but requires the input cubulation of $G$ to have walls exiting $\mathcal Q$.  Its proof uses a cubical presentation for $G$, whose underlying cube complex is $X'=G\backslash\widetilde X'$ and whose relators are the local isometries $\widetilde Y_q\hookrightarrow\widetilde X'\to X'$.  Actually, $\pi_1(\langle X'\mid \{\widetilde Y_q\}\rangle)=G$ since the $\widetilde Y_q$ are simply connected, and simple connectivity also makes this a cubical small-cancellation presentation in the sense of \cite{Wise:QCH}.  This trick allows us to create new walls in $\widetilde X'$, using the template from \cite[Sec. 5]{Wise:QCH}, which reduces the problem to declaring new walls in $\widetilde X'$ by choosing walls in each $\widetilde Y_q$.  The new walls in $\widetilde Y_q$ are declared to be certain pairs of exiting hyperplanes that are far apart in $T$.  In the new wallspace on $\widetilde X'$, these two hyperplanes cutting $q$ are absorbed into a single wall that fails to cut $q$.  This allows us to do two things: ensure that walls cutting various $t_e$ or $q'\in\mathcal Q$ do not cut $q$, and modify the translation length of $q$ without affecting the translation lengths of any $t_e$ or $q'$.  This produces  $\widetilde X$ with the required properties.

In Theorem \ref{thm:main}, the unbranching and tubularity hypotheses are dictated by the base case.  The base case is where $G$ is a linear mapping torus, which has a different acylindrical splitting as a graph of groups with product vertex groups and $\integers^2$ edge groups.  We need to cubulate $G$ so that the hyperplanes  exit the generators of edge groups at the next stage of the hierarchy.  This enables Theorem \ref{thm:cubulation-with-turns}, which equalises translation lengths and provides wall-independence, so that we can use Corollary \ref{cor:cyclic-splitting-turns} at the next stage.  The linear case requires a different approach, described in the next section.

We conclude our discussion of the superlinear case with some questions.

\begin{question}\label{question:general-cocompactness}
In Corollary \ref{cor:cyclic-splitting-turns}, suppose in addition that each of the actions $G_v\actson\widetilde X_v$ is cocompact and the edge groups are cubically convex-cocompact in the vertex groups.  Is there a free, cocompact action of (a finite index subgroup of) $G$ on a CAT(0) cube complex $\widetilde X$ satisfying the conclusion of the statement?  
\end{question}

One can imagine variants of Question \ref{question:general-cocompactness} where cocompactness is replaced by, for instance, finite-dimensionality and/or local finiteness of $\widetilde X_v$ and $\widetilde X$.  This leads to:

\begin{question}\label{question:proper}
For which $\Phi\in\Out(F)$ does $G=F\rtimes_\Phi\integers$ act metrically properly on a CAT(0) cube complex? For which $\Phi$ is there a free action on a proper CAT(0) cube complex?  For which polynomially growing $\Phi$ does $G$ act freely on a finite dimensional CAT(0) cube complex?  When is the action cocompact?  Is $G$ (virtually) cocompactly cubulated whenever $\Phi$ is fast?
\end{question}

One might hope for results analogous to the characterisation of cocompactly cubulated tubular groups in \cite{Wise:tubular} and the characterisations of tubular groups admitting finite-dimensional and/or locally finite cubulations in \cite{Woodhouse:fin-dim,Woodhouse:special}.

The cubulation arguments in the linear case rely on Lemma \ref{lem:small-cancellation-improper} and therefore produce infinite-dimensional cube complexes.  However, in the fast case, it is easy to make a cocompact action (with exits) using Remark \ref{rem:cocompactness-in-the-fast-case}.  We believe the answer to the fast part of Question \ref{question:general-cocompactness}, and Remark \ref{rem:cocompactness} sketches a strategy for modifying the proof of Theorem \ref{thm:main} to yield a cocompact action when $\Phi$ is fast.

\subsection{The linear case}\label{subsec:linear-intro}
Gersten's group,
$$G=\langle a,b,c,t\mid a^t=a,\ b^t=ba,\ c^t=ca^2 \rangle,$$
was shown to act freely on a CAT(0) cube complex in \cite{Wise:tubular}, using the tubular structure of $G$.  This example famously admits no proper semisimple action on a CAT(0) space \cite{Gersten}, and therefore shows that one cannot expect a finite-dimensional cubulation.  In fact, recent work of Munro-Petyt \cite{MunroPetyt} shows that $G$ is not quasi-isometric to a CAT(0) cube complex and actually has no coarse median structure.  This is an artifact of $G$ failing to be unbranched \cite{BGGH}.  So, even in the linear case, the best we can hope for is a free action.

On the positive side, Button showed that all tubular free-by-cyclic groups act freely on CAT(0) cube complexes \cite{Button:tubular-cubes}; Button uses the free-by-cyclic structure to find \emph{equitable sets}, yielding a free cubical action by \cite{Wise:tubular}.  There is thus a large class cubulated mapping tori in the linear case.  Indeed, one-ended tubular groups are \emph{thick of order $\le 1$} in the sense of \cite{BDM}, while mapping tori of superlinear-growth automorphisms are not \cite{Hagen:thickness,Macura:detour}.

After passing to a finite-index subgroup, any linear mapping torus $G$ arises from an automorphism preserving a splitting of $F$ over infinite cyclic subgroups.  So, up to finite index, $G$ has an acylindrical splitting as a finite graph of groups with $\integers^2$ edge groups and $F_i\times \integers$ vertex groups \cite{AndrewMartino:splitting,DahmaniTouikan}.  We call this the \emph{torus splitting} of $G$.

The torus splitting generalises the JSJ decomposition of a $3$--dimensional graph manifold arising as the mapping torus of a multitwist on a punctured surface, and our goal is to build walls similar to those employed in \cite{PrzytyckiWise} for graph manifolds. However, there are are significant differences.  For instance, graph manifold groups are unbranched, while in $G$, there may be a single $\integers^2$ contained in more than two vertex groups. When $G$ fails to be unbranched, we do not know how to cubulate $G$ so that the hyperplanes exit a given finite collection of elements.  This is why Theorem \ref{thm:main} assumes that $G$ is unbranched.  In Section \ref{subsec:intro-counterexample-question} we speculate about whether this phenomenon means there are non-cubulated quadratic  mapping tori.

Before explaining the conditions under which we are able to cubulate $G$ with exits, we discuss the question of when $G$ is cubulated at all.  As mentioned earlier, Button has cubulated all tubular free-by-$\integers$ groups, and our first statement about the linear case generalises his result. 

\Restate{Proposition}{prop:general-linear}{(Linear cubulation).}{Let $\Phi\in\Out(F)$ be linearly-growing and suppose the torus splitting of $G=F\rtimes_\Phi\integers$ has \emph{good edge measurements}.  Then $G$ is cubulated.}

To explain what \emph{good edge-measurements} are, it is best to outline the proof of the proposition.  First, since one can always use induced actions on finite products to pass from a virtual cubulation to a cubulation, we use standard tricks to pass to a finite index subgroup of $G$ and assume that the torus splitting of $G$ has  additional properties. Most importantly, $G$ is $\pi_1$ of a graph of spaces $\Delta$ with bipartition $\vertices(\Delta)=\vertices_\circ(\Delta)\sqcup\vertices_\bullet(\Delta)$.  Each $b\in \vertices_\bullet(\Delta)$ has vertex space is $r_b\times t_b$, where $r_b,t_b$ are circles.  Each $w\in\vertices_\circ(\Delta)$ has vertex space is $X_w\times t_w$, where $X_w$ is a graph and $t_w$ is a circle.  Edges have the form $(b,w)\in\vertices_\bullet(\Delta)\times\vertices_\circ(\Delta)$ and the  $(b,w)$ edge space is a copy of $r_b\times t_b$ attached to the vertex space $r_b\times t_b$ using the identity map, and to $X_w\times t_w$ using a map that sends $r_b$ to an \emph{embedded} cycle $p_b\subset X_w$ and that sends $t_b$ to the path $p_bt_w^k$ for some $k\in\integers$.  Moreover, the cycles $p_b$ in $X_w$ are distinct for distinct $b$ incident to $w$, which ensures the splitting is acylindrical.  The group $G$ is unbranched if and only if all $b$ have valence $2$.

By Proposition \ref{prop:joining-walls}, to cubulate $G$, it suffices to find immersed walls in each $X_w\times t_w$ where:
\begin{itemize}
    \item the resulting actions of the $\pi_1(X_w\times t_w)$ on cube complexes are free, and
    \item if $w,w'$ are adjacent to a common $b$, then the cubulations of $\pi_1(X_w\times t_w)$ and $\pi_1(X_{w'}\times t_{w'})$ induce the same cubulation on the common $\integers^2$ subgroup $\pi_1(r_b\times t_b)$.  
\end{itemize}
It actually suffices to produce \textbf{non-free} actions of the vertex groups on cube complexes that restrict to \textbf{free} actions of the edge groups satisfying the second bullet point, since Lemma \ref{lem:small-cancellation-improper} promotes these actions to free actions of the vertex groups without affecting the cubulations of the edge groups.  Our cube complexes are infinite-dimensional because of Lemma \ref{lem:small-cancellation-improper}.

Immersed walls in $X_w\times t_w$ are of two types: \emph{horizontal} walls correspond to kernels of virtual homomorphisms to $\integers$, and \emph{vertical} walls that do not cut the element represented by $t_w$.

The idea of the proof of Proposition \ref{prop:general-linear} is to construct a family of horizontal walls in $X_w\times t_w$.  These intersect each embedded torus $p_b\times t_w$ in a collection of circles.  The goal is to choose the walls according to prescribed circles in the various tori $X_w\times t_w$, to ensure that the induced cubulations of the edge-tori match those induced by the other vertex group cubulations.  We have been unable to do this for all linear mapping tori, but expect that extending Proposition \ref{prop:general-linear} will involve this strategy, and ask:

\begin{question}\label{question:linear}
Is the mapping torus of every linearly-growing element of $\Out(F)$ cubulated?
\end{question}

The \emph{good measurements} assumption enables the above strategy.  First, building on a canonical completion argument from \cite{Wise:omnipotence}, we pass to a further finite cover and assume that the set $\mathbf E(w)$ of cycles $p_b$ in $X_w$ is partitioned into subsets $\mathbf E_i(w)$ where, for each $i$, there is an \emph{edge measurement} homomorphism $\pi_i:\homology_1(X_w)\to \integers$ that takes nonzero values on elements of $\mathbf E_i(w)$ and vanishes on elements of $\mathbf E(w)-\mathbf E_i(w)$. \emph{Good edge measurements} asserts that $\{\pi_i\}_i$ can be chosen so that $\pi_i(p_b)=\pi_i(p_{b'})$ for all $p_b,p_{b'}\in\mathbf E_i(w)$.  This is satisfied, for instance, when $\pi_1X_w$ has a free product decomposition with all of the $p_b$ belonging to distinct free factors. Every tubular $G$ has good edge measurements, by Proposition \ref{prop:tubular-cubulation}, but Example \ref{exmp:good-measurements-non-tubular} is non-tubular.  Good edge measurements provide  homomorphisms $\pi_1(X_w\times t_w)\to \integers$ whose kernels give the vertex group cubulations that are then spliced together, using Proposition \ref{prop:joining-walls}, to cubulate $G$.  See Remark \ref{rem:edge-measure} for more on good edge-measurements.  A variant of this strategy gives the following, which serves as the base case in the inductive proof of Theorem \ref{thm:main}. 

\Restate{Proposition}{prop:exit-arrangement}{(Linear cubulation with exits).}{Let $\Phi$ be linearly growing and suppose that the torus splitting of $G=F\rtimes_\Phi\integers$ has good edge measurements, and $G$ is unbranched.  Then $G$ acts freely on a CAT(0) cube complex.  Moreover, if $\mathcal Q\subset G$ is a finite set of elements that are hyperbolic in the torus splitting, then there is a cubulation of $G$ where the hyperplanes exit $\mathcal Q$.}

We hope that others will find our analysis of the linear case to be a useful starting point for resolving Question \ref{question:linear}.  We wonder if ideas in \cite{MerladetMinasyan,Nguyen:separability,WuYe} can be applied.  We discuss other approaches in Section \ref{subsec:intro-counterexample-question}, where we present an example of a tubular linear mapping torus $G_0$ showing that Proposition \ref{prop:exit-arrangement} fails without unbranching, and an important test case is identified in Problem \ref{prob:super-dehn}.

\subsection{Outline}\label{subsec:outline}
Section \ref{sec:cubulation-lemmas} contains background on cube complexes, cubical small-cancellation theory, and other notions.  In Section \ref{sec:joining-walls}, we cubulate certain graphs of groups.  Section \ref{sec:mapping-tori-background} is background on splittings of mapping tori of UPG automorphisms.  Section \ref{sec:sufficient-conditions-linear} deals with linear mapping tori.  Section \ref{subsec:acyl-cyclic} contains the small-cancellation argument for turns, yielding Corollary \ref{cor:cyclic-splitting-turns}.  The ingredients are assembled in Section \ref{sec:conclusion}, which proves Theorem \ref{thm:main} and Corollary \ref{cori:generic}.

\medskip
\paragraph{\textbf{Acknowledgments}}\label{subsec:acknowl}
We thank Monika Kudlinska for telling us about \cite{AndrewMartino:splitting} and   Harry Petyt for discussions of \cite{MunroPetyt}.  We are very grateful to Naomi Andrew for explaining the $\integers^2$ splitting of linear mapping tori and for suggesting the appropriate notion of genericity among UPGs.  We thank Armando Martino for many helpful comments. We also thank referees of previous versions for useful feedback.

\section{Background}\label{sec:cubulation-lemmas}
We refer the reader to \cite[Sec. 2]{Wise:QCH} and \cite{Haglund:semisimple} for background on CAT(0) cube complexes.  

\subsection{Translation length in cube complexes}\label{subsec:cubical-translation}
Let $C$ be a CAT(0) cube complex.  Then $g\in\Aut(C)$ is \emph{(combinatorially) elliptic} if $g$ fixes a $0$--cube and \emph{(combinatorially) hyperbolic} if there is a bi-infinite combinatorial geodesic $\widetilde A$ in $C$, and 
$\tau\in\integers$, such that, regarding $\widetilde A$ as an isometric embedding $\widetilde A:\reals\to C^1$, we have $g\widetilde A(t)=\widetilde A(t+\tau)$ for all $t\in\reals$.  By Corollary~5.2 of~\cite{Haglund:semisimple}, the \emph{translation length} $\tau$ is independent of  $\widetilde A$; $\tau$ is the minimal distance by which $g$ moves a $0$--cube. By replacing $C$ by a \emph{cubical subdivision} -- i.e. by subdividing $C$ so that the hyperplanes become subcomplexes and each hyperplane is replaced by $2$ parallel hyperplanes -- we can assume that any group $G\leq\Aut(C)$ has the property that each $g\in G$ is either elliptic or hyperbolic~\cite{Haglund:semisimple}.  Indeed, it is shown in~\cite{Haglund:semisimple} that the action of $G$ on $C$ has this property provided $G$ acts \emph{without inversions in hyperplanes} in the sense that whenever $g\in G$ stabilizes a hyperplane, it stabilizes each of the associated halfspaces.  Passing to the cubical subdivision of $C$ ensures that the action is without inversions.

\begin{notation}[Translation length]\label{notation:trans_length}
Let $C$ be a CAT(0) cube complex.  If $g\in\Aut(C)$ is elliptic, $\|g\|_C=0$.  If $g\in\Aut(C)$ has axis $\tilde A$ in $C$, then $\|g\|_C$ denotes the number of $\langle g\rangle$--orbits of $1$--cubes in $\tilde A$.  If $G$ acts on $C$ without inversions, then $\|g\|_C$ is well-defined for all $g\in G$. We sometimes use the following convention on subscripts: if $g\in\Aut(C_\heartsuit)$, then $\|g\|_\heartsuit$ means $\|g\|_{C_\heartsuit}$, etc.
\end{notation}

\begin{obs}\label{obs:additive}
Suppose that $C_\heartsuit$ and $C_\diamondsuit$ are CAT(0) cube complexes, and $G$ acts on $C_\heartsuit$ and $C_\diamondsuit$.  Let $C=C_\heartsuit\times C_\diamondsuit$. Then, with respect to the diagonal action $G\to\Aut(C)$, $$\|g\|_C=\|g\|_\heartsuit+\|g\|_\diamondsuit$$ for all $g\in G$.  Moreover, $\|g^d\|_\heartsuit=|d|\|g\|_\heartsuit$ for all $g\in G$ and $d\in\integers$.
\end{obs}

\subsection{Cutting}\label{subsec:cut_wall}
We now recall a special case of \emph{wallspaces}  in~\cite{HaglundPaulin:wallspace}; see~\cite{HruskaWise:finiteness} for more details. A \emph{wallspace} $(\widetilde X,\mathcal W)$ is a metric space $\widetilde X$ with a set $\mathcal W$ of connected subspaces $\widetilde W$, called \emph{walls}, such that $\widetilde X-\widetilde W$ has exactly two components, called \emph{halfspaces}, and such that for all $x,y\in\widetilde X$, there are finitely many $\widetilde W\in\mathcal W$ for which $x,y$ lie in different components of $\widetilde X-\widetilde W$ (in which case $\widetilde W$ \emph{separates} $x,y$).  An \emph{automorphism} of $(\widetilde X,\mathcal W)$ is an isometry $g:\widetilde X\to\widetilde X$ so that $g\widetilde W\in\mathcal W$ for $\widetilde W\in\mathcal W$; the group of automorphisms of $(\widetilde X,\mathcal W)$ is denoted $\Aut(\widetilde X,\mathcal W)$.

\begin{defn}[Cutting, parallel]\label{defn:cutting}
Let $(\widetilde X,\mathcal W)$ be a wallspace and let $g\in\Aut(\widetilde X,\mathcal W)$.  Let $A:\reals\to\widetilde X$ be an embedding with $g$--invariant image so that $g\circ A:\reals\to\widetilde X$ is increasing.  Then $\widetilde W\in\mathcal W$ \emph{cuts} $g$ if there exists $x\in\reals$ so that $A((x,\infty))$ and $A((-\infty,x))$ lie in different halfspaces associated to $\widetilde W$. If no such $A$ exists, then the set of walls cutting $g$ is empty. We call $g,h\in\Aut(\widetilde X,\mathcal W)$ \emph{parallel} if: each $W\in\mathcal W$ cuts $g$ if and only if $W$ cuts $h$.
\end{defn}

The CAT(0) cube complex $C=C(\widetilde X,\mathcal W)$ \emph{dual to} the wallspace $(\widetilde X,\mathcal W)$ was defined in~\cite{Sageev:cubes}.  Denoting by $\widehat{\mathcal W}$ the set of halfspaces, the $0$--cubes of $C$ are maps $c:\mathcal W\to\widehat{\mathcal W}$ sending each wall to one of the two associated halfspaces, subject to: 
\begin{itemize}
    \item $c(\widetilde W)\cap c(\widetilde W')\neq\emptyset$ for $\widetilde W,\widetilde W'\in\mathcal W$, and
    \item $|\{\widetilde W:x\not\in c(\widetilde W)\}|<\infty$  for all $x\in\widetilde X$.
\end{itemize}
The $0$--cubes $c,c'$ are joined by a $1$--cube if and only if $|\{\widetilde 
 W:c(\widetilde W)\neq c'(\widetilde W)\}|=1$, and higher-dimensional cubes are added when their $1$--skeleta appear.

If $G$ acts on $(\widetilde X,\mathcal W)$, there is an induced $G$--action on $C$ defined by: $(gc)(\widetilde W)=g\cdot c(g^{-1}\widetilde W)$ for $g\in G$, $c\in C^0$, and $\widetilde W\in\mathcal W$.  This action takes hyperplanes to hyperplanes, and there is a $G$--equivariant bijection from $\mathcal W$ [resp. $\widehat{\mathcal W}$] to the set of hyperplanes [resp. halfspaces] in $C$.

The first of the following lemmas is~\cite[Lemma~2.1]{Wise:tubular}. The second follows by considering hyperplanes intersecting axes in the dual cube complex.

\begin{lem}[Cut-wall criterion]\label{lem:free}
Let $G$ act on a wallspace $(\widetilde X,\mathcal W)$ and suppose that each $g\in G-\{1_G\}$ is cut by some wall in $\mathcal W$.  Then $G$ acts freely on $C(\widetilde X,\mathcal W)$.
\end{lem}

\begin{lem}[Cut-walls and translation length in $C(\widetilde X,\mathcal W)$]\label{lem:cut_wall_trans_length}
Let $G$ act on the wallspace $(\widetilde X,\widetilde W)$ and let $g\in G$.  Let $\|g\|_{\mathcal W}$ be the number of $\langle g\rangle$--orbits of walls that cut $g$.  Then the translation length $\|g\|$ of $g$ on $C$ is defined and $\|g\|=\|g\|_{\mathcal W}$.
\end{lem}

If $g\in G$ has finite order, then $g$ cannot be cut, so $\|g\|_{\mathcal 
W}=0$, and $g$ must fix a point in $C$, so $\|g\|_C=0$.  This paper is about torsion-free groups, so this situation only arises for $g=1$.

\subsection{Virtual cubulation}\label{subsec:virtual-cubulation}
We now elaborate on the standard induced action construction in the cubical setting (see also \cite[Lem. 7.13]{Wise:QCH}).    

\begin{defn}[Virtual translation length]\label{defn:vtl}
Let $G$ be a group with a finite-index subgroup $G'$ acting on a CAT(0) cube complex $C'$.  Given $g\in G$, let $r>0$ be such that $g^r\in G'$ (for instance, $r=[G:G']!$).  The \emph{virtual translation length} is $\|g\|^{virt}_{C'}:=\frac1r\|g^r\|_{C'}$, which is independent of $r$.  
\end{defn}

\begin{defn}[Wall-independent]\label{defn:wall-independent}
Let $S$ be a group acting on a CAT(0) cube complex $C$ and let $\{s_i\}_{i\in I}\subset S$ be a set of combinatorially hyperbolic elements, so for each $i$, there is a combinatorial axis $\axis(s_i)$ for $s_i$ in $C$.  Then $\{s_i\}_{i\in I}$ is \emph{wall-independent} if for all $g\in S$ and $i,j\in I$, either $gs_ig^{-1}=s_j^{\pm1}$, or the set of hyperplanes $h$ of $C$ such that both $h\cap g\axis(s_i)$ and $h\cap \axis(s_j)$ are nonempty is finite. Any two axes of $s_i$ are cut by the same hyperplanes, so this definition is independent of the choice of axes.
\end{defn}

\begin{defn}[Algebraic elevation]\label{defn:finite-set-elevation}
Let $G$ be a group and let $G'\leq G$ have finite index.  Let $\mathcal Q\subset G$.  Let $\mathcal Q_0\subset \mathcal Q$ contain exactly one representative of each $G$--conjugacy class in $\mathcal Q$.  For each $g\in G$ and each $q\in\mathcal Q_0$, let $q'$ be a generator of $g\langle q\rangle g^{-1}\cap G'$.  Let $\mathcal Q''$ be the set of all such $q'$ (as $g$ and $q$ vary) and let $\mathcal Q'\subset G'$ contain exactly one representative of each $G'$--conjugacy class in $\mathcal Q''$.  Then $\mathcal Q'$ is an \emph{elevation} of $\mathcal Q$.
\end{defn}

\begin{lem}[Induced cubulation]\label{lem:star}
Let $G$ be a group and $G'\leq G$ a finite index subgroup.  If $G'$ acts on a CAT(0) cube complex $C'$, then $G$ acts on a CAT(0) cube complex $C$ such that:
\begin{itemize}
 \item $C$ is isomorphic to $(C')^{[G:G']}$, with the product cubical structure.
 
 \item For $a,b\in G$ with $\|h_1ah_1^{-1}\|_{C'}^{virt}=\|h_2bh_2^{-1}\|_{C'}^{virt}$ for all $h_1,h_2\in G$, we have $\|a\|_C=\|b\|_C$.
\end{itemize}

Suppose, moreover, that $\mathcal Q\subset G$ and $\mathcal Q'\subset G'$ is an elevation (see Definition \ref{defn:finite-set-elevation}).  Suppose that $\mathcal Q'$ is wall-independent for the $G'$--action on $C'$, and that for all $q,p\in\mathcal Q$ and $g\in G$, $g\langle p\rangle g^{-1}\cap \langle q\rangle\neq\{1\}$ implies $p=q^{\pm1}$. Then $\mathcal Q$ is wall-independent for the $G$--action on $C$.
\end{lem}

\begin{proof}
Let $\rho:G'\to\Aut(C')$ be the given action.  Equipping $G$ with the left multiplication action and $C'$ with the action $\rho$, let $C$ be the set of maps $f:G\to C'$ such that $f(hg)=\rho(h)(f(g))$ for all $h\in G'$ and $g\in G$.  Define a $G$--action on $C$ by: given $f\in C$ and $g\in G$, let $g\cdot f:G\to C'$ be given by $g\cdot f(k)=f(kg)$ for all $k\in G$.

Let $g_1,\ldots,g_{[G:G']}$ be a right transversal for $G'$ in $G$.  Define a map $I:C\to (C')^{[G:G']}$ by $I(f)=(f(g_1),\ldots,f(g_{[G:G']}))$.  This is bijective since $G'$--equivariant maps from $G$ are determined uniquely by where they send the $g_i$, and hence, by pulling back the cubical structure/metric, we can regard $C$ as a CAT(0) cube complex with the product structure as in the statement.

The inverse $I:(C')^{[G:G']}\to C$ is given by $I((c_i)_i)=f$, where $f(hg_i)=\rho(h)c_i$ for $h\in G'$ and $i\leq [G:G']$.  The $G$--action on $(C')^{[G:G']}$ is thus given by $g\cdot(c_i)_i=I(g\cdot f)=(f(g_ig))_i$.  So if $g=g_j^{-1}ag_j$ for some $a\in G'$, then for any $(c_i)_i$, the $j$--coordinate of $g\cdot(c_i)_i$ is $$f(ag_j)=\rho(a)f(g_j)=\rho(a)c_j.$$

Let $r=[G:G']!$ and let $a\in G$.  Then $a^r\in \bigcap_ig_i^{-1}G'g_i$, so for each $i\leq [G:G']$ there exists $a_i\in G'$ such that $a^{nr}=g_i^{-1}a_i^ng_i$ for all $n>0$.  Hence, if $f\in C$, then the $i^{th}$--coordinate of $I(a^{nr}\cdot f)$ is $a^{nr}\cdot f(g_i)=f(a_i^ng_i)=\rho(a_i)^nf(g_i)$.  Hence
$$I(a^{nr}\cdot f)=(\rho(a_1)^n(f(g_1)),\ldots,\rho(a_{[G:G']})^n(f(g_{[G:G']}))).$$
Hence the $\ell_1$--distance from $I(a^{nr}\cdot f)$ to $I(f)$ is 
$\sum_i\dist_{C'}(\rho(g_ia^{r}g_i^{-1})^n(f(g_i)),f(g_i)),$
where we have used $g_ia^rg_i^{-1}=a_i$. Dividing by $n$ and passing to the limit as $n\to\infty$ gives $\|a^r\|_C=\sum_i\|g_ia^rg_i^{-1}\|_{C'}=r\sum_i\|g_iag_i^{-1}\|^{virt}_C.$  This implies that $\|a\|_C=\|b\|_C$ provided all conjugates of $a,b$ have equal virtual translation length.

\textbf{Description of walls:}  Each hyperplane of $C$ is associated to a pair $(W,g_i)$, where $W\subset C'$ is a hyperplane: $(W,g_i)$ is the set of points in $\prod_{j=1}^{[G:G']}C'$ with $i$--coordinate in $W$.  

\textbf{Wall-independence:}  Let $\mathcal Q$ and $\mathcal Q'$ be as in the statement, and suppose that $\mathcal Q'$ is wall-independent for the $G'$--action on $C'$.  Let $\mathcal Q_0$ be a set of conjugacy class representatives for $\mathcal Q$ in $G$.  Then by re-choosing the coset representatives $g_1,\ldots,g_{[G:G']}$, we can assume that the elements of $\mathcal Q'$ have the form $g_iq^{n(q,i)}g_i^{-1}$ for $q\in\mathcal Q_0$ and $1\leq i\leq [G:G']$.  

Suppose that $\mathcal Q$ is not wall-independent for the $G$--action on $C$.  Let $p,q\in\mathcal Q$ and $g\in G$ be such that  $gpg^{-1}$ and $q$ are cut by infinitely many common walls in $C$.  So by the pigeonhole principle there exists $j$ such that infinitely many walls of the form $(W,g_j)$ cut both $gpg^{-1}$ and $q$.  These same walls cut $gp^rg^{-1}$ and $q^r$, both of which belong to $g_j^{-1}G'g_j$.  Hence $g_jg p^rg^{-1}g_j^{-1}$ and $g_j^{-1}q^rg_j$ are elements of $G'$ that are both cut by infinitely many walls for the action of $G'$ on $C'$.  Now, $g_j^{-1}q^rg_j \in a \langle q_0\rangle a^{-1}$ for some $q_0\in\mathcal Q'$ and $a\in G'$ and $(g_jg)p^r(g_jg)^{-1}\in b\langle q_1\rangle b^{-1}$ for some $b\in G'$ and $q_1\in \mathcal Q'$.  So, infinitely many hyperplanes in $C'$ cut both $aq_0a^{-1}$ and $bq_1b^{-1}$, so by the wall-independence assumption, we have $aq_0a^{-1}=bq_1^{\pm1}b^{-1}$.  Hence $g\langle p^r\rangle g^{-1}$intersect $\langle q^r\rangle$ nontrivially, so by our hypothesis about $\mathcal Q$, $p=q^{\pm1}$, so $\mathcal Q$ is wall-independent.   
\end{proof}

\subsection{Finitely many orbits of hyperplanes}\label{subsec:essential-core}
We recall a well-known fact (see e.g. \cite{CapraceSageev:rank}):

\begin{lem}\label{lem:finitely-many-orbits}
 If $G$ is finitely-generated group acting on a CAT(0) cube complex $X$,  then there is a convex $G$--invariant subcomplex $Y\subseteq X$ with finitely many orbits of hyperplanes.
\end{lem}

\subsection{Auxiliary $F\times\integers$ cube complexes with prescribed elliptics}\label{subsec:small-cancellation-free-times-Z}
When cubulating graphs of groups with $F\times\integers$ vertex groups and $\integers\times\integers$ edge groups, our strategy will be to construct cubical actions for the vertex groups where the edge groups are cubulated in a specific way, so we can extend walls over the whole graph.  We need to add additional walls in the vertex groups in order to obtain a free action on a cube complex, but we would like to do this without modifying the cubulations of the edge groups.  The following lemma enables this, at the expense of guaranteeing that the dimension of the cube complex will be infinite.

\begin{lem}[Improper cubical action with small-cancellation]\label{lem:small-cancellation-improper}
Let $F$ be a finite-rank free group and let $\mathcal P\subset F$ be a finite set of nontrivial elements such that $\langle p\rangle$ is a maximal cyclic subgroup of $F$ for each $p\in\mathcal P$.  Then $F$ acts on a CAT(0) cube complex $Z_{sc}$ in such a way that for all $a\in F$, we have $\|a\|_{Z_{sc}}=0$ if and only if $a$ is conjugate into $\langle p\rangle$ for some $p\in\mathcal P$.
\end{lem}

\begin{proof}
Without loss of generality, no two of the maximal cyclic subgroups $\langle p\rangle,\ p\in \mathcal P$ of $F$ are conjugate, and by conjugating, we can assume that all elements of $\mathcal P$ are cyclically reduced.  Let $\mathcal S\subset F$ be a set of cyclically reduced elements containing $\mathcal P$ and containing exactly one element of each conjugacy class of nontrivial $f\in F$ such that $\langle f\rangle$ is a maximal cyclic subgroup.

Let $f:\naturals\to \mathcal S$ be an enumeration of $\mathcal S$, with $f(1),\ldots,f(|\mathcal P|)$ the elements of $\mathcal P$.  For $g\in F-\{1\}$, let $\#(g)$ be the unique $i$ so that $g$ is conjugate into $\langle f(i)\rangle$.  

For each $n\in\naturals$, let $q_n>\max_{1\leq i\leq n}\{12, 6|f(i)|\}$.  For $n\geq 1$, let $$F_n=F/\langle\langle f(1)^{q_n},\ldots,f(n)^{q_n}\rangle\rangle.$$
Our choice of $q_n$ ensures that the above presentations are $C'(\frac16)$, so $F_n$ acts properly and cocompactly on a CAT(0) cube complex $C_n$ by \cite[Thm. 1.2]{Wise:small-cancellation}.  

The torsion elements of $F_n$ are powers of $f(1),\ldots,f(n)$ and their conjugates.  So, for each $g\in F-\{1\}$, if $n<\#(g)$, then $g$ has infinite-order image in $F_n$, and $n\geq \#(g)$ implies $g$ has finite-order image.  By cubically subdividing, we can assume that every isometry of each $C_n$ has a combinatorial axis or fixes a $0$--cube.

Fix a base $0$--cube $c_n\in C_n$.  Let $Z_{sc}^{(0)}\subset \prod_{n> |\mathcal P|}C_n^{(0)}$ be the set of $(v_n)_{n\geq 1}$ such that $\sum_{n> |\mathcal P|}\dist_{C_n}(c_n,v_n)<\infty$, where $\dist_{C_n}$ is the $\ell_1$ metric on $C_n$.  Since this is a sum of nonnegative integers, the condition is equivalent to asking that $v_n\neq c_n$ for only finitely many values of $n$.  Now join $(v_n)_n,(w_n)_n$ by an edge if there is a unique $m$ with $v_m\neq w_m$ and $v_m,w_m$ are adjacent in $C_m$.  Add higher-dimensional cubes where their $1$--skeleta appear.  We now have a CAT(0) cube complex $Z_{sc}$ where $F$ acts by isometries, acting on the $n^{th}$ factor via $F\mapsto F_n$.  (The product action preserves the subcomplex $Z_{sc}$ of the infinite product since $\|g\|_n=0$ for $n\geq \#(g)$ and all $g\in F$.)  On the other hand, $\|g\|_n>0$ for $n<\#(g)$, so $g$ is hyperbolic on $Z_{sc}$ if and only if $g$ is not conjugate into $\langle f(i)\rangle$ for some $i\leq |\mathcal P|$.  This completes the proof.
\end{proof}

\subsection{Matching shapes for subgroups}\label{subsec:shape}
The notion of the \emph{shape} given to a subgroup by an action of the ambient group on a CAT(0) cube complex is useful for cubulating graphs of groups.

\begin{defn}[Shape]\label{defn:shape}
Let $H,K$ be isomorphic groups acting on CAT(0) cube complexes $Y,Z$ respectively.  Then $H\curvearrowright Y$ and $K\curvearrowright Z$ \emph{have the same shape for the isomorphism $\phi:H\to K$} if there is a cubical isometry $\psi:Y\to Z$ where $\phi(h)(\psi(y))=\psi(hy)$ for all $h\in H$ and $y\in Y$.  
\end{defn}

\begin{defn}[Isometric subcomplex]\label{defn:isometric-subcomplex}
Let $\widetilde X$ be a CAT(0) cube complex.  A CAT(0) subcomplex $\widetilde Y\subset\widetilde X$ is \emph{isometric} if any $x,y\in\widetilde Y^{(0)}$ are joined by a $\widetilde X^{(1)}$--geodesic lying in $\widetilde Y$.
\end{defn}

\begin{rem}\label{rem:isometric-subcomplex}
Convex subcomplexes are isometric, but the converse does not hold.  Also, if $\widetilde Y\subset\widetilde X$ is isometric, then each hyperplane of $\widetilde X$ intersects $\widetilde Y$, if at all, in a single hyperplane of $\widetilde Y$, and all hyperplanes of $\widetilde Y$ have the form $h\cap \widetilde Y$ for a unique hyperplane $h$ of $\widetilde X$.    \end{rem}

\begin{defn}[Subgroup shape]\label{defn:subgroup-shape}
Let $G$ act on the CAT(0) cube complex $\widetilde X$.  Let $H,K\leq G$, and fix an isomorphism $\phi:H\to K$.  Then the $G$--action on $\widetilde X$ \emph{gives $H$ and $K$ the same shape for $\phi$} if there are isometric subcomplexes $\widetilde Y,\widetilde Z$ of $\widetilde X$ that are, respectively, $H$ and $K$ invariant and the restrictions $H\curvearrowright \widetilde Y$ and $K\curvearrowright \widetilde Z$ of the $G$--action on $\widetilde X$ have the same shape for $\phi$.
\end{defn}

We are interested in the shapes of free abelian subgroups of rank $1$ or $2$.  The following appears in various forms in the literature; here is a version with no dimension restriction:

\begin{lem}[Axis theorem]\label{lem:generalised-axis}
Let $G$ be a group acting without inversions on a CAT(0) cube complex $\widetilde X$, and suppose that $A\leq G$ is a finite-rank free abelian subgroup such that the action of $A$ on $\widetilde X$ is free.  Then there exists $N\in[\rank A,\infty)$ such that there is an isometric subcomplex $\widetilde Y\subset\widetilde X$ that is $G$--invariant and isomorphic to the standard tiling of $\reals^N$ by unit $N$--cubes.
\end{lem}

\begin{proof}
By \cite[Lem. 5.6]{Genevois}, and our assumption that isometries of $\widetilde X$ are either combinatorially elliptic or hyperbolic (the action is without inversions), $A$ stabilises a \emph{median flat} $\widetilde Y^{(0)}$ in $\widetilde X$.  This is a finite-dimensional median subalgebra of $\widetilde X^{(0)}$ consisting of the union of $0$--skeletons of axis of elements of $A$.  Hence the CAT(0) cube complex $\widetilde Y$ with $0$--skeleton $\widetilde Y^{(0)}$ naturally includes into $\widetilde X$ and by, for instance, \cite[Lem. 2.11]{HagenPetyt}, $\widetilde Y$ is isometric.  Now, since $\widetilde Y$ is finite-dimensional, we can apply \cite[Thm. 4.3]{Woodhouse:axis} to the $A$--action on $\widetilde Y$ to replace $\widetilde Y$ by an isometric subcomplex decomposing as the product of $N$ cubical quasilines, for some $N\geq \rank A$.  Now choose invariant combinatorial geodesics in the factors, and replace $\widetilde Y$ by their product.
\end{proof}

\begin{defn}\label{defn:generalised-axis}
Given $G,A,$ and $\widetilde X$ as in Lemma \ref{lem:generalised-axis}, a subcomplex $\widetilde Y\subset\widetilde X$ satisfying the conclusion of the lemma is a \emph{generalised axis} for $A$.
\end{defn}

The next lemma is immediate from the definitions:

\begin{lem}\label{lem:shapes-and-products}
Let $G$ act on the CAT(0) cube complexes $\widetilde X_1$ and $\widetilde X_2$.  Suppose that $H,K\leq G$ and that the $G$--actions on $\widetilde X_1$ and $\widetilde X_2$ both give the same shape to $H$ and $K$, for some fixed isomorphism $\phi:H\to K$.  Then the diagonal action of $G$ on $\widetilde X_1\times\widetilde X_2$ gives the same shape to $H$ and $K$, for the same isomorphism $\phi$.
\end{lem}


\subsection{Cubical small-cancellation theory}\label{subsec:small-cancellation}
We review background from \cite{Wise:QCH}.  Throughout,  $\widetilde X$ denotes a CAT(0) cube complex and $X$ a nonpositively-curved cube complex.  For a hyperplane $\widetilde{U}$ of $\widetilde{X}$, we denote by $N(\widetilde{U})$ its \emph{carrier}, i.e. the union of all closed cubes intersecting $\widetilde{U}$.  We do the same for immersed hyperplanes in $X$. More background on cubical presentations appears in \cite[Sec. 3.b]{Wise:QCH}, but the main definition is:

\begin{defn}[Cubical presentation]\label{defn:cubical_presentation}
A \emph{cubical presentation} $\langle X | \{Y_i:i\in I\}\rangle$ consists of connected non-positively curved cube complexes $X$ and $\{Y_i\}_{i\in I}$, and local isometries $\{Y_i \to X\}_{i\in I}$. We use the notation $X^*$ for the cubical presentation above. As a topological space, $X^*$ is $X$ with a cone on each $Y_i$.
\end{defn}

  Each $Y_i\to X$ is $\pi_1$--injective, and $\pi_1X^*\cong\pi_1X/\nclose{\{\pi_1Y_i:i\in I\}}$.  For each $i\in I$, the cone over $Y_i$ in $X^*$ lifts to the universal cover $\widetilde{X^*}$, and accordingly $Y_i\to X\to X^*$ lifts to $Y_i\to\widetilde{X^*}$.

\begin{defn}[Piece]\label{defn:piece}
A \emph{cone-piece} in $Y_i$ is a component of $\widetilde{Y}_i \cap g \widetilde{Y}_j$, where $g \in \pi_1(X)$, excluding the case where $i = j$ and $g \in \stabilizer( \widetilde{Y_i} )$. A \emph{wall-piece} in $Y_i$ is a component of $\widetilde{Y}_i \cap N(\widetilde{U})$, where $\widetilde{U}$ is a hyperplane that is disjoint from  $\widetilde{Y}_i$.  A \emph{piece-path} in $Y$ is a path in a piece of $Y$.  We sometimes abuse language and refer to piece-paths as \emph{pieces}.
\end{defn}

\begin{rem}\label{rem:piece-stuff}
The general definition of a piece is more complicated; the given definition is that of a \emph{contiguous} piece.  As shown in \cite{Wise:QCH}, it suffices to work with contiguous pieces.  Another subtlety in the definition (see \cite[Conv. 3.3]{Wise:QCH} and \cite[Rem. 3.10]{ArzhantsevaHagen}), about the difference between $\pi_1Y_i\leq \pi_1X$ and $\stabilizer_{\pi_1X}(\widetilde Y_i)$, is irrelevant, because in our applications, $\stabilizer_{\pi_1X}(\widetilde Y_i)$ is cyclic, eliminating the subtlety (whose nature we therefore won't mention).     
\end{rem}

\begin{defn}[$C'(\alpha)$ cubical presentation]\label{defn:c12}
The cubical presentation $X^*=\langle X\mid\{Y_i:i\in I\}\rangle$ satisfies the  \emph{$C'(\alpha)$ condition} if $|P|\leq\alpha|S|$ for each geodesic piece-path $P$ and each essential closed path $S\to Y_i$ with $P$ a subpath of $S$.
\end{defn}

Next is a simplified form of the $B(6)$ condition described in~\cite[Def. 5.1, Rem. 5.4]{Wise:QCH}.

\begin{defn}[Simplified $B(6)$] \label{defn:B6}
Let $X^*=\langle X\mid\{Y_i:i\in I\}\rangle$ be a cubical presentation satisfying the $C'(1/14)$ condition.  Then $X^*$ satisfies the \emph{$B(6)$ condition} if the following hold:
\begin{enumerate}
 \item \label{B6:equivalence} for each $i\in I$, there is an equivalence relation $\sim_i$ on the hyperplanes of $Y_i$.
 \item \label{B6:wallspace} for each $i\in I$ and each pair of (possibly equal) hyperplanes $U,V$ of $Y_i$ with $U\sim_i V$, the hyperplanes $U,V$ do not cross or osculate (recall that hyperplanes $U,V$ \emph{osculate} if they are disjoint but have intersecting carriers);
 \item \label{B6:wallspace_2} for each $i\in I$, and each $\sim_i$--class $[U]$ of hyperplanes, $W=\cup_{V\in[U]}V$ is a \emph{wall}: the space $Y_i-W$ consists of two subspaces $\OL Y_i,\OR Y_i$ such that $\closure{\OL Y_i}\cap\closure{\OR Y_i}=W$;
 \item \label{B6:homotopy} if $P\to Y_i$ is a path that is the concatenation of at most $7$ piece-paths and $P$ starts and ends on the carrier $N(U)$ of a wall then $P$ is path-homotopic into $N(U)$;

\item \label{B6:aut} $\Aut(Y_i\rightarrow X)$ preserves the wallspace structure.
\end{enumerate}
\end{defn}

The walls in $Y_i$ Definition~\ref{defn:B6} need not be connected, unlike the walls in Section~\ref{subsec:cut_wall}.  

\subsection{Free cubical actions of generalised $B(6)$ groups}\label{subsec:cubulating-small-cancelllation}
The $B(6)$ condition provides a wallspace structure on $\widetilde{X^*}$ as follows: two hyperplanes of $\widetilde{X^*}$ are \emph{equivalent} if they intersect a lift $Y_i\hookrightarrow \widetilde{X^*}$ of some $Y_i$ in equivalent hyperplanes of that $Y_i$.  The transitive closure of this relation is an equivalence relation on the hyperplanes of $\widetilde{X^*}$ and the union of the hyperplanes in each class is a wall.  The CAT(0) cube complex dual to this wallspace is the cube complex \emph{associated} to $X^*$; observe that $\pi_1X^*$ acts on this cube complex.  This is explained in \cite[Sec. 5]{Wise:QCH}.  We now state a result from \cite{Wise:QCH} ensuring that this construction gives a free action.

\begin{defn}[Proximate]\label{defn:proximate}
Let $\langle X\mid\{Y_i:i\in I\}\rangle$ be a cubical presentation satisfying 
Definition~\ref{defn:B6}.\eqref{B6:equivalence},\eqref{B6:wallspace},\eqref{B6:wallspace_2}.  A hyperplane $U$ in 
$Y_i$ is \emph{proximate} to a $0$--cube $v$ of $Y_i$ if 
there is a 1-cube $u$ dual to $U$ and a path $AB$ that starts with $u$ and ends with $v$,
and where each of $A$ and $B$ is either a 1-cube or lies in a piece.  A wall $W$ in $Y_i$ is \emph{proximate} 
to $v$ if some hyperplane $U$ of $W$ is proximate to $v$.
\end{defn}

Given $\langle X\mid\{Y_i:i\in I\}\rangle$ as in Definition~\ref{defn:proximate}, a set $\mathcal R$ of hyperplanes 
of $X$ is \emph{preferred} if each $Y_i$ has the following property: let $W$ be a wall in $Y_i$.  Then either all 
hyperplanes of $W$ map to a hyperplane of $X$ belonging to $\mathcal R$, or no hyperplane of $W$ maps to a 
hyperplane in $\mathcal R$.  If $\mathcal R$ is a preferred set of hyperplanes in $X$, then a wall $\widetilde W$ of 
$\widetilde{X^*}$ is \emph{preferred} (with respect to $\mathcal R$) if some (and hence any) hyperplane in 
$\widetilde W$ maps to a preferred hyperplane in $X$.  If $\widetilde W$ is not preferred, then none of its 
constituent hyperplanes maps to a preferred hyperplane in $X$.

The following is Theorem~5.40 in~\cite{Wise:QCH}:

\begin{thm}\label{thm:cubulating X star}
Let $\langle X \mid \{Y_i\}_i\rangle$ be a cubical presentation.  Suppose that:
\begin{enumerate}
\item $X^*$ satisfies the $B(6)$ condition and has \emph{short innerpaths} (see \cite[Definition~3.69]{Wise:QCH}).

\item $X$ contains a preferred set $\mathcal R$ of hyperplanes.

\item \label{LS:strong separation}
The following holds for each $Y\in\{Y_i\}$.  Let $\kappa\to Y$ be a geodesic with endpoints $p,q$.  Let $w_1,w'_1$ 
be hyperplanes of $Y$ lying in the same wall and mapping under $Y\to X$ to hyperplanes in $\mathcal R$.  Suppose 
that $\kappa$ traverses a $1$--cube dual to $w_1$ and either $w'_1$ is proximate to $q$ or $w_1'$ is dual to a 
$1$--cube traversed by $\kappa$.  Then there is a preferred wall $w_2$ in $Y$ that separates $p,q$ but is not proximate to $p$ 
or $q$.
\end{enumerate}
Let $g\in\pi_1X^*$.  Then one of the following holds:
\begin{enumerate}
     \item there exists $Y\in\{Y_i\}$ so that $g\in\Aut(Y)$ for some lift of $Y$ to $\widetilde 
X^*$;\label{item:in_aut}
\item $g$ is cut by a preferred wall of $\widetilde{X^*}$;\label{item:preferred_cut}
\item $g$ is the image of $\tilde g\in\pi_1X$ not cut by a hyperplane of $\widetilde X$ mapping to a 
hyperplane in $\mathcal R$;\label{item:not_preferred_cut}
\item $g\in\pi_1X^*$ has finite order.\label{item:torsion}
\end{enumerate}

\end{thm}

In our applications, $\langle X \mid \{Y_i\}_i\rangle$ will satisfy the short innerpaths condition 
by~\cite[Lemma 3.70]{Wise:QCH} because of the metric small-cancellation condition.  Therefore, it is not 
necessary to define short innerpaths.  We will apply Theorem~\ref{thm:cubulating X star} in the case where each $Y_i$ is a specific CAT(0) cube complex mapping by a local isometry to $X$.  Finally we need  \cite[Thm. 5.18]{Wise:QCH}:

\begin{thm}\label{thm:520}
Let $X^*$ be a $B(6)$ presentation and let $W$ be a wall in $\widetilde{X^*}$.  If $H_1,H_2$ are hyperplanes in $W$, and $Y$ is a cone, then $H_1\cap Y$ and $H_2\cap Y$ lie in the same wall of $Y$.
\end{thm}

\section{Wall-extension for graphs of groups}\label{sec:joining-walls}
Fix a group $G$ acting on a tree $T$. We always assume such actions are without inversions, so that the quotient $\bar T=G\backslash T$ is the underlying graph of a graph of groups. For each $v\in \vertices(\bar T)$, let $G_v$ denote the associated vertex group.  For each oriented edge $e\in\edges(\bar T)$, let $G_e$ be the associated edge-group, and let $G_e^\pm$ be the images, under the edge-monomorphisms, of $G_e$ in $G_{e^\pm}$, where $e^-$ and $e^+$ are the initial and terminal vertices of $e$.  Let $\phi_e:G_e^-\to G_e^+$ be the composition of the edge-monomorphism $G_e\to G_{e^+}$ with the isomorphism $G_e^-\to G_e$ obtained by inverting the  edge-monomorphism.

\subsection{Graphs of groups with matching walls}\label{subsec:matching-walls}
The following definition abstracts a situation that arises when cubulating polynomial mapping tori.

\begin{defn}\label{defn:graph-walls}
Let $G$ act cocompactly without inversions on the tree $T$. Assume:
\begin{enumerate}[(A)]
 \item $\mathcal Q\subset G$ is a finite (possibly empty) set of elements that are hyperbolic on $T$.  For each $q\in \mathcal Q$, let $\axis(q)$ be the axis of $q$ in $T$.  Assume $\stabilizer_G(\axis(q))=\langle q\rangle$ and there exists $\chi<\infty$ such that if $q,q'\in\mathcal Q$ and $g\in G$ and $\axis(q)\neq g\axis(q')$, then $\diam_T(\axis(q)\cap g\axis(q'))\leq \chi$.\label{item:stass-Q}
  
 \item For all $v\in\vertices(\bar T)$, the vertex group $G_v$ acts freely on a CAT(0) cube complex $\widetilde X_v$, with finitely many orbits of hyperplanes.\label{item:stass-cubulated}
 
 \item \label{item:stass-shape} For all $e\in\edges(\bar T)$, there is a $G_e^\pm$--invariant CAT(0) subcomplex $\widetilde X_e^\pm\subset \widetilde X_{e^\pm}$ that is isometric in $\widetilde X_{e^\pm}$, and $G_e^-\curvearrowright \widetilde X_e^-$ and $G_e^+\curvearrowright \widetilde X_e^+$ have the same shape with respect to $\phi_e$.  Let $\psi_e:\widetilde X_e^-\to\widetilde X_e^+$ be the cubical isomorphism from Definition \ref{defn:shape}.
\end{enumerate}
The data $(G,T,\{\widetilde X_v\}_v,\{\widetilde X_e^{\pm}\}_e,\{\psi_e\}_e)$ is called a \emph{graph of cubulated groups with compatibly cubulated edge-groups} if it satisfies \eqref{item:stass-cubulated} and \eqref{item:stass-shape}.  
\end{defn}

The main goal in this section is to prove:

\begin{prop}[Basic wall extension]\label{prop:joining-walls}
Let $(G,T,\{\widetilde X_v\}_v,\{\widetilde X_e^{\pm}\}_e,\{\psi_e\}_e)$ be a graph of cubulated groups with compatibly cubulated edge groups.  Let $\mathcal Q\subset G$ satisfy Definition \ref{defn:graph-walls}.\eqref{item:stass-Q}.  Then $G$ acts freely on a CAT(0) cube complex $\widetilde X$ with the following properties.
\begin{enumerate}[(I)] 
 \item \label{item:join-walls} $\widetilde X$ has finitely many $G$--orbits of hyperplanes.   

 \item \label{item:join-tree} There is a $G$--equivariant cubical map $\pi:\widetilde X\to T$.
 
 \item Let $v\in\vertices(\bar T)$ and let $\tilde v\in T$ be a vertex projecting to $v$.  Let $g\in G$ be such that $\stabilizer_G(\tilde v)=gG_vg^{-1}$.  Then $\pi^{-1}(\tilde v)$ is a $\stabilizer_G(\tilde v)$--invariant convex subcomplex of $\widetilde X$  and there is a cubical isometric embedding $g:\widetilde X_v\to \pi^{-1}(\tilde v)$ that is equivariant with respect to the $G_v$--action on $\widetilde X_v$, the $\stabilizer_G(\tilde v)$--action on $\pi^{-1}(\tilde v)$, and the isomorphism $G_v\to \stabilizer_G(\tilde v)$ given by conjugation by $g$.  Moreover, $\pi^{-1}(\tilde v)$ is the convex hull in $\widetilde X$ of the image of this embedding.\label{item:join-vertex-space}
 
 \item For each $\tilde e\in\edges(T)$, the subcomplex $\pi^{-1}(\tilde e)$ is a hyperplane carrier in $\widetilde X$.\label{item:join-hyperplane}
 
 \item \label{item:join-subtree} For any subtree $T'$ of $T$, the subcomplex $\pi^{-1}(T')$ is convex in $\widetilde X$.
 
 \item \label{item:join-edge-space} Let $e\in\edges(\bar T)$ and let $\tilde e\in\edges(T)$ be an edge mapping to $e$ with $\stabilizer_G(\tilde e)=G_e$ and let $m_{\tilde e}$ be its midpoint.  Then $\pi^{-1}(m_{\tilde e})$ is the convex hull in $\widetilde X$ of a $G_e$--invariant isometrically embedded CAT(0) subspace that is $G_e$--equivariantly isometric to $\widetilde X_e^-$.
 
\end{enumerate}
Hence $\widetilde Y_q=\pi^{-1}(\axis(q))$ is a convex subcomplex of $\widetilde X$ and $\stabilizer_G(\widetilde Y_q)=\langle q\rangle$, for each $q\in\mathcal Q$.  
\end{prop}

\begin{proof}
For each $v\in\vertices(\bar T)$, let $X_v=G_v\backslash \widetilde X_v$, which is a nonpositively-curved cube complex since $G_v$ acts freely without inversions.  For each $e\in\edges(\bar T)$, let $X_e^\pm=G_e^{\pm}\backslash \widetilde X_e$, which are also nonpositively-curved cube complexes with universal covers $\widetilde X_e^\pm\to X_e^\pm$.  The inclusion $\widetilde X_e^\pm\hookrightarrow\widetilde X_{e^\pm}$ descends to a cubical map $X_e^\pm\to X_{e^\pm}$ which is $\pi_1$--injective and induces the inclusion $G_e^\pm\to G_{e^\pm}$.  Moreover, $\psi_e$ descends to a cubical isomorphism $\bar\psi_e:X_e^-\to X_e^+$ inducing the isomorphism $\phi_e$ of fundamental groups. Define a graph of spaces as follows.
\begin{itemize}
 \item The underlying graph is $\bar T$.
 \item For each $v\in\vertices(\bar T)$, the vertex space is $X_v$.
 \item For each $e\in\edges(\bar T)$, the edge space is $X_e^-$.
 \item Equip $X_e^-\times[-1,1]$ with the product (nonpositively-curved) cubical structure, where $[-1,1]$ is regarded as a $1$--cube.
 \item Attach each $X_e^-\times[-1,1]$ as follows.  Identify $X_e^{-}\times\{-1\}$ with a subspace of $X_{e^-}$ using the map $X_e^-\times\{-1\}\stackrel{Id}{\longrightarrow}X_e^-\to X_{e^-}$, where the second arrow is induced by the inclusion $\widetilde X_e^-\hookrightarrow \widetilde X_{e^-}$.  Identify $X_e^-\times\{1\}$ with a subspace of $X_{e^+}$ using the map $X_e^-\times\{1\}\stackrel{\bar\psi_e}{\longrightarrow} X_e^+\to X_{e^+}$, where the second arrow is induced by the inclusion $\widetilde X_e^+\hookrightarrow\widetilde X_{e^+}$.
\end{itemize}
The total space of the resulting graph of spaces is denoted $Z$, which is a cube complex (containing each $X_v$ as a subcomplex) that need not be nonpositively curved since the attaching maps of the edge spaces need not be local isometries (recall $\widetilde X_e^\pm$ is isometric, but not necessarily convex, in $\widetilde X_{e^\pm}$).  Note that $\pi_1Z\cong G$ since $G$ is the fundamental group of the graph of groups with underlying graph $\bar T$, vertex groups $G_v,\ v\in\vertices(\bar T)$, etc.  Let $\widetilde Z\to Z$ be the universal cover.

\textbf{Walls in $\widetilde Z$:}  There is a cellular map $\bar\rho:Z\to\bar T$ sending $X_v\to v$ for each $v\in\vertices(\bar T)$ and sending the image of $X_e^-\times[-1,1]$ to $e\cong [-1,1]$ for each $e\in\edges(\bar T)$.  This gives the usual $G$--equivariant surjection $\rho:\widetilde Z\to T$ sending each lift of $\widetilde X_v$ to the corresponding lift of $v$, and each lift of $\widetilde X_e\times[-1,1]$ to the corresponding lift of the edge $e\cong [-1,1]$.

Pulling back the cubical structure from $Z$ makes $\widetilde Z$ a cube complex.  The map $\rho$ decomposes $\widetilde Z$ as a tree of cube complexes whose vertex spaces are the $G$--translates of $\widetilde X_v,\ v\in\vertices(\bar T)$ and whose edge spaces are the $G$--translates of the CAT(0) cube complexes $\widetilde X_e^-\times[-1,1],\ e\in\edges(\bar T)$.  We now analyse hyperplanes.

For each edge $\tilde e$ of $T$, letting $m_{\tilde e}$ be the midpoint, $\rho^{-1}(m_{\tilde e})$ is a separating hyperplane in $\widetilde Z$; the associated halfspaces are the $\rho$--preimages of the components of $T-\{m_{\tilde e}\}$.  We record:    

\begin{defn}\label{defn:vertical-Z-wall}
Such hyperplanes are \emph{vertical} and $\mathcal V(\widetilde Z)$ is the set of vertical hyperplanes in $\widetilde Z$.    
\end{defn}

Next are \emph{horizontal} hyperplanes.  Let $\tilde v$ be a vertex of $T$ and consider the CAT(0) subcomplex $\rho^{-1}(\tilde v)$.  Let $h_{\tilde v}=h\subset \rho^{-1}(\tilde v)$ be a hyperplane.  For each edge $\tilde e$ of $T$ incident to $\tilde v$, there is a subcomplex $\widetilde X_{\tilde e}^\pm$ of $\rho^{-1}(\tilde v)$ (the $\pm$ corresponds to whether $\tilde v$ is terminal or initial in $\tilde e$) which is the image of $\rho^{-1}(m_{\tilde e})\times \{\pm 1\}$ under an edge map. By definition $\widetilde X_{\tilde e}^\pm\hookrightarrow\rho^{-1}(\tilde v)$ is a cubical isometric embedding.  Hence either $h\cap\widetilde X_{\tilde e}^\pm=\emptyset$, or $h\cap\widetilde X_{\tilde e}^\pm$ is a single hyperplane $h_{\tilde e}$ of the CAT(0) cube complex $\widetilde X_{\tilde e}^\pm$, which extends to a hyperplane $h_{\tilde e}\times[-1,1]$ of $\rho^{-1}(\tilde e)$.  Now, letting $\tilde w$ be the other vertex of $\tilde e$, the image of the $\tilde e$ edge map in $\rho^{-1}(\tilde w)$ is the isometrically embedded CAT(0) subcomplex $\widetilde X_{\tilde e}^\mp$, which contains a hyperplane $h_{\tilde e}\times\{\mp1\}$.  Hyperplanes of any $\ell_1$--embedded CAT(0) subcomplex of $\rho^{-1}(\tilde w)$ are intersections of the subcomplex with hyperplanes of $\rho^{-1}(\tilde w)$; specifically, there exists a unique hyperplane $h_{\tilde w}$ of $\rho^{-1}(\tilde w)$ such that $h_{\tilde e}\times\{\mp 1\}=h_{\tilde w}\cap \widetilde X_{\tilde e}^\mp$.

Define $h_{\tilde v}$ to be equivalent to each such $h_{\tilde w}$, as $\tilde w$ varies over the vertices incident to $\tilde v$ in $T$ with the property that the edge $\tilde e$ from $\tilde v$ to $\tilde w$ corresponds to an edge space in $\rho^{-1}(\tilde v)$ intersecting $h_{\tilde v}$ nontrivially.  In this case, we also declare $h_{\tilde v}$ to be equivalent to the hyperplane $h_{\tilde e}\times[-1,1]$ of $\widetilde X_{\tilde e}^-\times[-1,1]$, which has the property that $h_{\tilde v}$ and $h_{\tilde w}$ intersect $\widetilde X_{\tilde e}^-\times\{\pm 1\}$ in $h_{\tilde e}\times \{\pm1\}$ and $h_{\tilde e}\times\{\mp1\}$ respectively (up to possibly changing signs).  Taking the transitive closure gives a tree of spaces whose vertex spaces are hyperplanes $h_{\tilde v}\subset \rho^{-1}(\tilde v)$ for various vertices $\tilde v$ of $T$ and whose edge spaces are hyperplanes of the form $h_{\tilde e}\times [-1,1]$.  The total space $H$ is equipped with a map $H\to\widetilde Z$ which is the inclusion on each vertex space and edge space of $H$.  

Observe that $H$ intersects each $\rho^{-1}(\tilde v)$ in at most one hyperplane, i.e. two distinct hyperplanes $h_{\tilde v},h'_{\tilde v}$ in the same vertex space $\rho^{-1}(\tilde v)$ of $\widetilde Z$ cannot be equivalent; in other words, the underlying graph of $H$ injects in $T$.  (The key point here is that each edge space of $\widetilde X$ has isometrically embedded image in $\widetilde X_{\tilde v}$ and therefore intersects each hyperplane of $\widetilde X_{\tilde v}$ in a single component.)

Hence $H$ is a tree of CAT(0) cube complexes whose edge spaces are CAT(0) subcomplexes that are convex in their incident vertex subcomplexes (since hyperplanes are convex in CAT(0) cube complexes).  Thus $H$ is a CAT(0) cube complex.  Moreover, $H\to \widetilde Z$ is an embedding, and we regard $H$ as a subspace of $\widetilde Z$ intersecting each vertex/edge space in at most one hyperplane.  From the construction, $H$ locally separates $\widetilde Z$ into two components, so $\widetilde Z-H$ has two components, by the usual Mayer-Vietoris argument using simple connectivity of $\widetilde Z$.  Hence $H$ is a wall in $\widetilde Z$.  

\begin{defn}\label{defn:wandering-in-Z}
$H$ is a \emph{wandering} hyperplane and $\mathcal W(\widetilde Z)$ is the set of wandering hyperplanes.  
\end{defn}

Note that $\mathcal W(\widetilde Z)\cup\mathcal V(\widetilde Z)$ is a $G$--invariant set of walls in $\widetilde Z$.  Let $\widetilde X$ be the CAT(0) cube complex dual to the $G$--wallspace $(\widetilde Z, \mathcal W(\widetilde Z)\cup\mathcal V(\widetilde Z))$.  The hyperplanes of $\widetilde X$ are in a $G$--equivariant bijection with $\mathcal W(\widetilde Z)\cup\mathcal V(\widetilde Z)$.  Accordingly, the hyperplanes of $\widetilde X$ are \emph{vertical} or \emph{wandering}, and: 

\begin{defn}\label{defn:vertical-wandering-in-X}
We let $\mathcal W(\widetilde X)$ and $\mathcal V(\widetilde X)$ respectively denote the sets of hyperplanes in $\widetilde X$ corresponding to wandering and vertical walls in $\widetilde Z$.
\end{defn}

\textbf{The $G$--action on $\widetilde X$:} Each of $\mathcal W(\widetilde Z),\mathcal V(\widetilde Z)$ is $G$--invariant.  The $G$--action on the latter is cofinite since there are finitely many orbits of edges in $T$.  The $G$--action on the former is cofinite since there are finitely many $G$--orbits of vertices, each $\widetilde X_v$ has finitely many $G_v$--orbits of hyperplanes, and each wandering hyperplane is determined by a vertex space and a hyperplane of that vertex space.  Hence $G$ acts on $\widetilde X$ with finitely many orbits of hyperplanes, verifying \eqref{item:join-walls}.

Next, we check that the $G$--action on $\widetilde X$ is free.  Let $g\in G-\{1\}$.  If $g$ is hyperbolic on $T$, then $g$ is cut by a vertical wall, so $\|g\|_{\widetilde X}>0$.  Otherwise, up to conjugacy, $g\in G_v$ for some $v\in\vertices(\bar T)$, i.e. the action of $g$ on $\widetilde Z$ stabilises $\widetilde X_{\tilde v}$.  Since  $G_v$ acts on $\widetilde X_{\tilde v}$ freely (and we can assume that the action is without inversions), $g$ is cut by some hyperplane $h_{\tilde v}$ of $\widetilde X_{\tilde v}$.  Since the wandering hyperplane $H$ determined by $h_{\tilde v}$ intersects $\widetilde X_{\tilde v}$ in $h_{\tilde v}$ only, and $g$ has an axis in $\widetilde X_{\tilde v}$, we see that $g$ is cut by $H$.  Hence $\|g\|_{\widetilde X}>0$.  (Both translation length computations use Lemma \ref{lem:cut_wall_trans_length}.)  Hence the $G$--action on $\widetilde X$ is free, by Lemma \ref{lem:free}.

\textbf{The map $\pi$:}  Consider the $G$--wallspace with underlying set $\widetilde X^{(0)}$ and walls $\mathcal V(\widetilde X)$.  Note that the CAT(0) cube complex dual to this wallspace is $G$--equivariantly isomorphic to $T$, and forgetting the hyperplanes in $\mathcal W(\widetilde X)$ therefore yields a $G$--equivariant restriction quotient $\pi:\widetilde X\to T$.  (See \cite[Sec. 2]{CapraceSageev:rank} or \cite{HuangKleiner} for a description of restriction quotients.)  A general property of restriction quotients is that preimages of convex subcomplexes are convex subcomplexes, and preimages of hyperplanes are hyperplanes \cite[Sec. 4]{HuangKleiner}.  So, we have verified \eqref{item:join-tree}, \eqref{item:join-hyperplane}, \eqref{item:join-subtree}.

We next verify \eqref{item:join-vertex-space}.  Let $v\in\vertices(\bar T)$ and let $\tilde v\in T$ be a vertex mapping to $v$, so that $\stabilizer_G(\tilde v)=gG_vg^{-1}$ for some $g\in G$.  By the construction of $\widetilde Z$, we have that $\rho^{-1}(\tilde v)$ is an isometric copy of $\widetilde X_v$, and there is a cubical isomorphism $g:\widetilde X_v\to\rho^{-1}(\tilde v)$ that conjugates the given action $G_v\to\Aut(\widetilde X_v)$ to the action $gG_vg^{-1}\to \Aut(\rho^{-1}(\tilde v))$.

Note that there is a $G$--equivariant cubical embedding $\iota:\widetilde Z\to\widetilde X$.  Indeed, each $z\in\widetilde Z^{(0)}$ chooses a consistent orientation of the walls in $\mathcal W(\widetilde Z)\cup\mathcal V(\widetilde Z)$ and hence defines a unique $0$--cube $\iota(z)\in\widetilde X$.  Moreover, if $y,z$ are adjacent in $\widetilde Z$, then they are separated by a unique wall (because the walls in $\widetilde Z$ are just hyperplanes, which we showed above separate), so $\iota(y),\iota(z)$ are distinct, adjacent $0$--cubes in $\widetilde X$.  This defines $\iota$ on the $1$--skeleton, and one extends over higher-dimensional cubes using that $\widetilde X$ is CAT(0).  There is a commutative diagram
\begin{center}
 $
 \begin{diagram}
  \node{\widetilde Z}\arrow{s,l}{\rho}\arrow{e,t}{\iota}\node{\widetilde X}\arrow{s,r}{\pi}\\
  \node{T}\arrow{e,b}{Id}\node{T}
 \end{diagram}
 $
\end{center}
of $G$--equivariant cubical maps.  In particular, $\iota$ restricts to a $\stabilizer_G(\tilde v)$--equivariant embedding $\rho^{-1}(\tilde v)\to\pi^{-1}(\tilde v)$.  Recall that the hyperplanes of $\rho^{-1}(\tilde v)$ are exactly the nonempty intersections $H\cap \rho^{-1}(\tilde v)$, where $H$ varies over the wandering hyperplanes in $\widetilde Z$ that intersect $\rho^{-1}(\tilde v)$; this is because vertex spaces in $\widetilde Z$ are CAT(0) cube complexes and the wandering hyperplanes intersect vertex spaces in single hyperplanes.  On the other hand, $\pi^{-1}(\tilde v)$ is the convex subcomplex of $\widetilde X$ defined as follows: choose an arbitrary $z\in\rho^{-1}(\tilde v)$ and consider all $0$--cubes $y\in\widetilde X$ such that every hyperplane separating $\iota(z)$ from $y$ corresponds to a wandering hyperplane.  Each such $y$ is in $\pi^{-1}(\tilde v)$.  If some $0$--cube $y\in\widetilde X$ is separated from $\iota(z)$ by a vertical hyperplane, then $\pi(y)\neq \tilde v$.  So $\pi^{-1}(\tilde v)$ is spanned by the $0$--cubes $y$ differing from $\iota(z)$ only on wandering hyperplanes.  This shows that $\pi^{-1}(\tilde v)$ is the convex hull of $\iota(\rho^{-1}(\tilde v))$, and that the latter is a connected median subalgebra of the former, in particular it is isometrically embedded (see e.g. \cite[Sec. 2]{Bowditch} or \cite[Lem. 2.11]{HagenPetyt}).  (Note that $\iota(\rho^{-1}(\tilde v))$ might not be convex, because wandering hyperplanes that are disjoint in $\widetilde X_{e^-}$ must extend disjointly through $\widetilde X_e^-\times[-1,1]$ but could then map to crossing hyperplanes in $\widetilde X_{e^+}$, since $\widetilde X_{e}^+$ need not be convex in $\widetilde X_{e^+}$.)  We have shown \eqref{item:join-vertex-space}.

Now consider $\tilde e\in\edges(T)$ with stabiliser $G_e$, where $e$ is the edge of $\bar T$ to which $\tilde e$ maps.  Recall that $\rho^{-1}(m_{\tilde e})$ is a vertical hyperplane in $\widetilde Z$, and in fact coincides with $\widetilde X_e^-\times\{0\}$.  So $\pi^{-1}(m_{\tilde e})$ is the $\ell_1$ convex hull (it is a subcomplex of the cubical subdivision, and in fact a hyperplane) in $\widetilde X$ of $\iota(\rho^{-1}(m_{\tilde e}))$, proving \eqref{item:join-edge-space}.

\textbf{Axes for $\mathcal Q$:}  For each $q\in\mathcal Q$, $\axis(q)$ is a geodesic in $T$, so $\widetilde Y_q=\pi^{-1}(\axis(q))$ is a convex subcomplex of $\widetilde X$ by \eqref{item:join-subtree}.  The stabiliser of $\widetilde Y_q$ coincides with the stabiliser of $\axis(q)$ by \eqref{item:join-tree}, which is $\langle q\rangle$ by Definition \ref{defn:graph-walls}.\eqref{item:stass-Q}. 
\end{proof}

\begin{rem}[Dehn twisting walls]\label{rem:arranging-standing-assumption-with-isomorphisms}
Suppose that $G$ and $T$ satisfy Definition \ref{defn:graph-walls}.  Part \eqref{item:stass-shape} of the definition says that for each edge $e$ of $\bar T$, we have a cubical isomorphism $\psi_e:\widetilde X_e^-\to\widetilde X^+_e$ such that the diagram
\begin{center}
 $
 \begin{diagram}
 \node{\widetilde X_e^-}\arrow{s,l}{g}\arrow{e,t}{\psi_e}\node{\widetilde X_e^+}\arrow{s,r}{\phi_e(g)}\\
 \node{\widetilde X_e^-}\arrow{e,b}{\psi_e}\node{\widetilde X_e^+}
 \end{diagram}
 $
\end{center}
commutes for all $g\in G_e^-$.  Now, fix $h\in G_e^+$.  Then $h\circ\psi_e:\widetilde X_e^-\to\widetilde X_e^+$ is again a cubical isomorphism, and the diagram above continues to commute when $\psi_e$ is replaced by $h\circ\psi_e$ and $\phi_e$ is replaced by $Ad_h\circ \phi_e:G_e^-\to G_e^+$, where $Ad_h\in\Aut(G_e^+)$ is conjugation by $h$.  In particular, if $h$ is central in $G_e^+$, then we can modify the construction of $\widetilde X$ in Proposition \ref{prop:joining-walls} as follows: build $Z$ exactly as in the proof of the proposition, except glue $X_e^-\times\{1\}$ to $X_{e^+}$ using the map $X_e^-\times \{1\}\stackrel{\bar \psi_e}{\longrightarrow} X_e^+\to X_{e^+}$, where the second arrow is induced by $\widetilde X_e^+\stackrel{h}{\longrightarrow}\widetilde X_e^+\hookrightarrow\widetilde X_{e^+}$.

We will only use this in the proof of Corollary \ref{cor:cyclic-splitting-turns}, where the edge groups are cyclic.  In this case, $G_e$ is generated by an element $t_e$ and $\widetilde X_e$ is a line.  We will choose $h=t_e^n$ for some $n\in\integers$.  The above construction has the following simple interpretation in this case: $Z$ is constructed exactly as in the proof (using $\psi_e$, but we do not impose a cubical structure on $X_e\times [-1,1]$.  Instead, it is just a cylinder whose attaching circles $X_e\times \{\pm1\}$ are given the same orientation, and $X_e\times \{\pm1\}$ contains a finite set $\{p_i^\pm\}_i$ of points that are preimages of hyperplanes of $X_{e^\pm}$ under the attaching maps.  The map $\psi_e$ induces a bijection $\beta:\{p_i^-\}_i\to\{p_i^+\}_i$.  We extend the walls from the vertex spaces across the cylinder by joining each $p_i^-$ to $\beta(p_i^-)$ by a properly embedded arc in $X_e\times[-1,1]$ that winds around the cylinder $|n|$ times (in the same direction as the orientations of the attaching circles if $n\geq 0$ and in the opposite direction if $n<0$).  In other words, the new immersed walls in $Z$ are just the images of the ones from the proof of the proposition under a collection of Dehn twists in the edge spaces.  The resulting  $\widetilde X$ satisfies the conclusions of the proposition; the proof works verbatim for the Dehn-twisted walls.
\end{rem}

\begin{rem}[Cocompactness in Proposition \ref{prop:joining-walls}]\label{rem:splice-cocompact}
Let $(G,T,\{\widetilde X_v\}_v,\{\widetilde X_e^{\pm}\}_e,\{\psi_e\}_e)$ be a graph of cubulated groups with compatibly cubulated edge groups.  Let $\mathcal Q\subset G$ satisfy Definition \ref{defn:graph-walls}.\eqref{item:stass-Q}.  Suppose, in addition, that:
\begin{itemize}
    \item $G_v$ acts on $\widetilde X_v$ cocompactly for each $v\in\vertices(\bar T)$.
    \item For each $e\in\edges(\bar T)$, the $G_e^\pm$--invariant subcomplex $\widetilde X_e^\pm\subset \widetilde X_{e^\pm}$ is convex in $\widetilde X_e^{\pm}$.
\end{itemize}
Then $G$ acts freely \emph{and cocompactly} on the CAT(0) cube complex $\widetilde X$ satisfying the conclusions of Proposition \ref{prop:joining-walls}.  Indeed, in this situation, the cube complex $Z$ in the proof of the proposition is a finite graph of spaces where the vertex spaces and edge spaces are compact nonpositively-curved cube complexes and the attaching maps are local isometries, so by a standard argument, $Z$ is already nonpositively curved, and it suffices to take $\widetilde X=\widetilde Z$.

Moreover, if $G$ is torsion-free and the $G$--action on the Bass-Serre tree $T$ is acylindrical (so any finite collection $\mathcal Q$ of $T$--hyperbolic elements that are not proper powers automatically satisfies Definition \ref{defn:graph-walls}.\eqref{item:stass-Q}), then each $q\in\mathcal Q$ is a Morse element \cite[Thm. 1]{Sisto:morse}, and hence $\langle q\rangle$ is \emph{convex-cocompact}, i.e. it acts cocompactly on a convex subcomplex of $\widetilde X$ (the convex hull of any $q$--axis), by, for instance, \cite[Lem. 4.8]{Hagen:tuples}.  Moreover, for any $v$ and any $t\in G_v$ such that $\langle t\rangle$ is convex-cocompact on $\widetilde X_v$, Proposition \ref{prop:joining-walls}.\eqref{item:join-vertex-space} implies that $\langle t\rangle$ is still convex-cocompact on $\widetilde X$.  
\end{rem}

\subsection{Exits and projection bounds}\label{subsec:exits}
Let $(G,T,\{\widetilde X_v\}_v,\{\widetilde X_e^{\pm}\}_e,\{\psi_e\}_e)$ be a graph of cubulated groups with compatibly cubulated edge groups.  Let $\mathcal Q\subset G$ be a finite subset satisfying Definition \ref{defn:graph-walls}.\eqref{item:stass-Q}. Let $\widetilde X$ be the cube complex, and $\pi:\widetilde X\to T$ the map, from Proposition \ref{prop:joining-walls}.

\begin{defn}[Walls exit]\label{defn:walls-exit}
  The $G$--cube complex $\widetilde X$ has \emph{walls exiting $\mathcal Q$} if there exists $\chi_0<\infty$ such that $\diam_T(\pi(H)\cap \axis(q))\leq \chi_0$ for all $q\in\mathcal Q$ and all hyperplanes $H$ of $\widetilde X$.
\end{defn}

\begin{rem}\label{rem:better-exits}
In statements later where the hypothesis is that $\widetilde X$ has walls exiting $\mathcal Q$, we could use a weaker notion of walls exiting, namely that the diameter bound in Definition \ref{defn:walls-exit} applies to all $q\in\mathcal Q$ and all $H$ \emph{that cut $q$}.  However, in the settings where we are able to arrange any form of exiting at all, we get the stronger form.  So we chose the simpler definition.
\end{rem}

\begin{defn}[Projection to $\widetilde Y_q$]\label{defn:axis-projection}
 For $q\in\mathcal Q$, recall from Proposition \ref{prop:joining-walls} that $\widetilde Y_q=\pi^{-1}(\axis(q))$ is a convex subcomplex of $\widetilde X$ satisfying $\stabilizer_G(\widetilde Y_q)=\langle q\rangle$.  By convexity, for each $q\in\mathcal Q$ and $g\in G$, there is a combinatorial closest-point projection $\gate_{g,q}:\widetilde X\to g\widetilde Y_q$ such that $\gate_{gq}(gx)=g\gate_{1,q}(x)$ for all $x\in\widetilde X$.  (This is known as the \emph{gate map} in the median literature and \cite[Rem. 3.5, Rem. 3.6]{ArzhantsevaHagen} explains that this coincides with the notion of \emph{wall-projection} from \cite{HW:comb}.)
\end{defn}

The following lemma will be used later to support cubical small-cancellation arguments.

\begin{lem}[Axis projection bound]\label{lem:axis-subcomplex-overlap}
Let $(G,T,\{\widetilde X_v\}_v,\{\widetilde X_e^{\pm}\}_e,\{\psi_e\}_e)$ be a graph of cubulated groups with compatibly cubulated edge groups, with $\mathcal Q$ as in Definition \ref{defn:graph-walls}.\eqref{item:stass-Q}.  Let $\widetilde X$  and $\pi:\widetilde X\to T$ satisfy the conclusion of Proposition \ref{prop:joining-walls}.  Then there exists $\chi_1<\infty$ such that for all $q_0,q_1\in\mathcal Q$ and $g_0,g_1\in G$, either
$\diam(\pi(\gate_{g_0,q_0}(g_1\widetilde Y_{q_1})))\leq \chi_1$
or $q_0=q_1$ and $g_0g_1^{-1}\in\langle q_0\rangle$.
\end{lem}

\begin{proof}
Let $\bar\gate:T\to \pi(g_0\widetilde Y_{q_0})$ be the closest-point projection map, which exists since $\pi(g_0\widetilde Y_{q_0})$ is a subtree of $T$ and hence convex.  Since $\pi$--preimages of convex sets are convex, by Proposition \ref{prop:joining-walls}, $\pi\circ \gate_{g_0,q_0}=\bar\gate\circ\pi$.  Hence $\pi(\gate_{g_0,q_0}(g_1\widetilde Y_{q_1}))=\bar\gate(g_1\widetilde Y_{q_1})$ is either a single point, or coincides with $\pi(g_1\widetilde Y_{q_1})\cap \pi(g_0\widetilde Y_{q_0})$, in which case Definition \ref{defn:graph-walls}.\eqref{item:stass-Q} provides the desired $\chi_1$.
\end{proof}

\section{Polynomial mapping tori and their hierarchies}\label{sec:mapping-tori-background}
Fix a finite-rank free group $F$ and an element $\Phi\in\Out(F)$.  The \emph{mapping torus} $G$ of $\Phi$ is the preimage of $\langle \Phi\rangle$ in 
$\Aut(F)$, so $G\cong F\rtimes_\phi\integers$ for any representative automorphism $\phi\in \Phi$, i.e. $G$ 
has the presentation $\langle F,t\mid tft^{-1}=\phi(f),\ f\in F\rangle$. 

Fixing a free basis $S$ of $F$, let $|\bullet|_S$ denote word length in $F$.  
For $g\in F$, let $\ell_S(g) = \min\{|hgh^{-1}|_S:h\in F\}$ be the conjugacy 
length of $g$.  For $n\geq 0$, and $\phi\in \Phi$, note that 
$\ell_S(\phi^n(g))$ is independent of the choice of $\phi$, and we denote this 
quantity $\ell_S(\Phi^n(g))$.  

\begin{defn}\label{defn:polynomial-growth-rate}
The \emph{growth rate} of $\Phi$ is 
\emph{polynomial of degree $d$} if for all nontrivial $g\in F$ there exists 
$d'\leq d$ and $A,B>0$ such that
$$An^{d'}\leq \ell_S(\Phi^n(g))\leq Bn^{d'}$$
and $d$ is infimal such that the above holds.  If $d=1$, we say 
that $\Phi$ has \emph{linear growth} and if $d>1$, we say $\Phi$ has 
\emph{superlinear growth}.  
\end{defn}

\begin{defn}\label{defn:UPG}
Let $f$ be the rank of $F$.  If the image of $\Phi$ in 
$GL_r(\integers)$ is unipotent, then $\Phi$ is \emph{UPG} (short for \emph{unipotent polynomially growing}).
\end{defn}

\begin{rem}[UPG case]\label{rem:make-upg}
For any $k>0$, the mapping torus of $\Phi^k$ has finite index in $G$, so if $\Phi$ has polynomial growth, then $\Phi^k$ 
is polynomial of the same growth rate \cite{Macura:detour}.  
Moreover, by \cite[Cor. 5.7.6]{BFH:tits-I}, there exists 
$k=k(\rank(F))>0$ such that $\Phi^k$ is UPG. By Lemma \ref{lem:star}, for cubulating $G$, it suffices to consider finite-index subgroups of $G$, so we often assume $\Phi$ is UPG.  
\end{rem}

\subsection{Relative train track representatives}\label{subsubsec:RTTM}
If $P\to\Gamma$ is a combinatorial path in a graph, then $[P]$ denotes the \emph{tightening} of $P$, i.e. the immersed 
path homotopic rel endpoints to $P$. 

\begin{defn}[Nielsen path]\label{defn:Nielsen}
Let $f:\Gamma\to\Gamma$ be a continuous map taking vertices to vertices.  A \emph{Nielsen path} is a path $P$ in $\Gamma$ such that 
$[f(P)]=P$, and an \emph{indivisible Nielsen path} is a Nielsen path $P$ such that either $P$ is an edge fixed pointwise by $f$, or $P$ cannot be expressed as the concatenation of two nontrivial Nielsen paths.  A path $P$ such that $[f^k(P)]=P$ for some $k>0$ is a \emph{periodic Nielsen path} and $k$ is its \emph{period}.
\end{defn}

Combining \cite[Thm. 5.1.5]{BFH:tits-I} and \cite[Prop. 3.12]{BFH:kolchin} yields well-behaved representatives of UPG automorphisms as graph maps:

\begin{prop}[Improved relative train tracks for UPGs]\label{prop:basic-rttm}
Let $F$ be a free group with $\rank(F)<\infty$ and let $\Phi\in \Out(F)$ be a 
UPG automorphism.  Then there exists a finite connected graph 
$\Gamma$ and a homotopy equivalence $f:\Gamma\to\Gamma$ such that all of the 
following hold:
\begin{enumerate}
 \item There is a sequence of $f$--invariant subgraphs 
$\emptyset=\Gamma_0\subset 
\Gamma_1\subset\cdots\subset\Gamma_r=\Gamma$.\label{item:filtration}

\item For $1\leq i\leq r$, there is an edge $E_i$ of $\Gamma$ such that 
$E_i=\Gamma_i-\interior{\Gamma_{i-1}}$.

\item For $1\leq i\leq r$, there is a (possibly trivial) immersed closed 
combinatorial path $P_i$ in $\Gamma_{i-1}$ such that $f(E_i)=E_iP_i$. 
\label{item:edge-image}

\item \label{item:vertex-preserved}  In particular, $f(v)=v$ for all $v\in 
\vertices(\Gamma)$.

\item For each $v\in\vertices(\Gamma)$, there is an isomorphism 
$\tau_v:\pi_1(\Gamma,v)\to F$ such that the induced isomorphism 
$f_*:\pi_1(\Gamma,v)\to\pi_1(\Gamma,v)$ has the property that 
$\phi_v=\tau_v\circ f_*\circ\tau_v^{-1}$ is a representative of the outer 
automorphism $\Phi$.

\item All periodic Nielsen paths have period $1$.

\item $f$ is a \emph{UR map} in the sense of \cite[Def. 3.13]{BFH:kolchin}.
\end{enumerate}
We call $f:\Gamma\to\Gamma$ an \emph{improved relative train track 
representative} of $\Phi$.
\end{prop}

We do not need the exact definition of a UR map; we include it above to emphasise that a map satisfying the conclusion of the proposition satisfies the hypotheses of various statements from \cite{BFH:kolchin} that we will apply later.

We allow the $\Gamma_i$ to be disconnected (and view the restrictions of $f$ to the components as relative train track representatives of UPG automorphisms of the fundamental groups of the components).  If $P_i$ is nontrivial, then $E_i$ necessarily ends on $\Gamma_{i-1}$, so adding $E_i$ does not increase the number of components.

Fix a UPG $\Phi$ and let $f:\Gamma\to\Gamma$ be an 
improved relative train track representative.  Define a relation $\preceq$ on 
$\edges(\Gamma)$ by $E\preceq E'$ if the combinatorial path $f(E')$ traverses 
the edge $E$ (in either direction).  Observe that $\preceq$ is a partial 
order.  Note that if $1\leq i<j\leq r$, then $E_j$ is not an edge of $\Gamma_{i-1}$, 
so we cannot have $E_j\preceq E_i$.  So, the total order on the edges given by 
$E_i<E_j$ if $i<j$ is obtained from $\preceq$ by breaking ties.  So, after relabelling, there exists $\ell\leq r$ such that the 
edge $E_i$ is $f$--\emph{invariant}, i.e. $P_i$ is the trivial path, if and 
only if $1\leq i\leq \ell$.  The edge $E$ has \emph{growth rate $d(E)$} if $\frac{|[f^n(E)]|}{n^{d(E)}}\in[a,b]$ for all $n\geq 1$, where 
$0<a\leq b$. 

\begin{lem}\label{lem:stratify-linear-superlinear}
Up to relabelling, there exist $1\leq \ell\leq m\leq r$ such that
\begin{itemize}
 \item $E_1,\ldots,E_\ell$ are 
invariant, i.e. $d(E_i)=0$ for $i\leq\ell$;
\item $E_{\ell+1},\ldots,E_m$ are \emph{linear edges}, i.e. 
they satisfy $d(E_i)=1$ for $\ell+1\leq i\leq m$, and
\item $E_{m+1},\ldots,E_r$ are \emph{superlinear 
edges}, i.e. they satisfy $d(E_i)>1$ for $m+1\leq i\leq r$.
\end{itemize}
Moreover, $\Phi$ has at most linear growth if and only if $r=m$.
\end{lem}

\begin{proof}
If $E_i$ is invariant and $E_j$ is not, then 
$E_j\not\preceq E_i$, so after relabelling, $i<j$.  We now show that if 
$d(E)=1$ and $d(E')>1$, then $E'\not\preceq E$, so we can interchange $E,E'$ in the filtration of $\Gamma$ without affecting the conclusion of Proposition \ref{prop:basic-rttm}.  

Suppose to the contrary that $E$ maps over $E'$, so $f(E)=E\cdot P$, where the 
closed immersed path $P$ does not traverse $E$, but $P=A\cdot E'\cdot B$, for 
some paths $A,B$. 
By \cite[Prop. 4.10]{BFH:kolchin}, for any $n>0$, there exists 
$k>0$ such that for all $L\geq k$, the immersed path $[f^L(P)]$ (or its 
inverse) contains the subpath $E'\cdot P'\cdot[f(P')]\cdots[f^n(P')]$, where 
$P'$ is the path provided by Proposition \ref{prop:basic-rttm} such that 
$f(E')=E'\cdot P'$.  Here we used that $P'$ cannot be a Nielsen path since otherwise $[f^k(E')]=E'\cdot (P')^k$, 
contradicting that $d(E')>1$.  So $[f^L(E)]$ has length at least $n(L-k)$ for 
all $L>k$.  Hence $\lim_L|[f^L(E)]|/L\geq n$, and thus $d(E)>1$, a contradiction.

Thus we can relabel so that $E_1,\ldots,E_\ell$ are invariant, 
$E_{\ell+1},\ldots,E_m$ are linear, and $E_{m+1},\ldots,E_r$ are superlinear.  If $r=m$, then all conjugacy classes grow 
at most linearly, i.e. $\Phi$ has linear growth.  The converse is 
another application of \cite[Prop. 4.10]{BFH:kolchin}.
\end{proof}

\begin{conv}\label{conv:edge-labels}
We always assume that 
improved relative train track representatives have their edges numbered as in Lemma \ref{lem:stratify-linear-superlinear}.
\end{conv}

\subsection{Cyclic hierarchy in the superlinear case}\label{subsec:topmost-edge-splittings}
We now recall the cyclic hierarchy for a superlinear UPG mapping torus $G$, following the discussions in \cite{Macura:detour,Hagen:thickness,AndrewHughesKudlinska}. Let $\Phi$ be a UPG element of $\Out(F)$ and let 
$f:\Gamma\to\Gamma$ be an improved relative train track map representing $\Phi$.  Suppose $\Phi$ has superlinear growth, and recall that 
$E_r$ is the last edge in the filtration of $\Gamma$ from Proposition 
\ref{prop:basic-rttm}.  The superlinear edges are $E_{m+1},\ldots,E_r$.  

Let $E_{m+1},\ldots,E_{m_1}$ be the $\preceq$--minimal superlinear edges; let 
$E_{m_1+1},\ldots,E_{m_2}$ be the $\preceq$--minimal edges among $E_j$ with 
$j>m_1$.  Continuing in this way gives $s:=m_p\leq r$ such that 
$E_s,\ldots,E_r$ are pairwise $\preceq$--incomparable (and are all 
$\preceq$--maximal) superlinear edges, and for $s\leq j\leq k$, the following 
holds.  Recall that $f(E_j)=E_jP_j$.  Then $P_j$ must traverse at least $p-1$ 
superlinear edges (namely all those in a largest $\preceq$--chain not involving $\preceq$--maximal edges).  We say 
that $s$ was chosen using a \emph{greedy clique sequence} for the partial order 
$\preceq$.  If $p=0$, then the superlinear edges are all incomparable, but each one must map over some linear edge.

As described in \cite{Macura:detour,Hagen:thickness,AndrewHughesKudlinska},  $G=F\rtimes_\Phi\integers$ decomposes as a finite 
graph of groups:
\begin{enumerate}
 \item The underlying graph $\Delta_r$ is obtained from $\Gamma_r=\Gamma$ by 
collapsing each component of 
$\Gamma-\bigcup_{i=s}^r\interior{E_i}$.  Identify $\edges(\Delta_r)$ with 
$\{E_s,\ldots,E_r\}$ in the obvious way.

\item For $\bar v\in\vertices(\Delta_r)$, let $v\in\vertices(\Gamma)$ be an 
arbitrary vertex in the component $\Gamma^v$ of 
$\Gamma-\bigcup_{i=s}^r\interior{E_i}$ that collapses to $\bar v$.  Since $f$ 
fixes every vertex, by Proposition \ref{prop:basic-rttm}, and since no edge in 
$\Gamma^v$ can map over any of $E_s,\ldots,E_r$, the map $f$ restricts to a map 
$f:\Gamma^v\to\Gamma^v$ that is an improved relative train track map 
representing the automorphism $\phi^v$ obtained by restricting $\phi_v$ to the 
free factor  $F^v=\pi_1(\Gamma^v,v)$ of $\pi_1(\Gamma,v)$.  We let 
$G_{\bar v}=F^v\rtimes_{\phi^v}\langle t_v\rangle$.  This describes the vertex 
groups in the splitting of $G$.

\item Orient the edges of $\Delta_r$.  For $s\leq i\leq r$, the edge group 
$G_{E_i}$ is an infinite cyclic group generated by an element that we denote 
$T_{E_i}$.  Let $\bar v_i^+,\bar v_i^-$ denote the terminal and initial vertices 
of $E_i$ regarded as an edge of $\Delta_r$, and let $v_i^\pm\in\Gamma^{v_i^\pm}$ 
be the lifts to $\Gamma$ described above.  Let $w_i^\pm\in \Gamma^{v_i^\pm}$ be 
the endpoints of the oriented edge $E_i$ of $\Gamma$.  Let $\alpha_i^\pm$ be a combinatorial path in 
$\Gamma^{v_i^\pm}$ from $v_i^\pm$ to $w_i^\pm$.  Then edge maps 
$I_i^\pm:\langle T_{E_i}\rangle\to G_{\bar v^\pm_i}$ are:
$$I_i^-(T_{E_i})=t_{v_i^-}\text{ \ and \  }I_i^+(T_{E_i})=[\alpha^+_iP_if(\alpha_i^+)^{-1} t_{v_i^+}].$$
\end{enumerate}

The fundamental group $\pi_1(\Delta_r,\{G_v\},\{G_{E_i}\},\{I_i^\pm\})$ of this 
graph of groups is isomorphic to $G$, and this decomposition is called the 
\emph{top strata decomposition} of $G$.

\begin{rem}[Edge groups and the 
free-by-$\integers$ structure]\label{rem:edge-groups-relate-to-fibres}
Let $\pi:G\to\integers$ be the quotient with kernel $F$.  Then $\pi(I_i^\pm(T_{E_i}))=1$ for each $E_i$.  
\end{rem}

\begin{rem}[Superlinear hierarchy]\label{rem:topmost-edge-vertex-groups}
Each $G_v=F^v\rtimes_{\phi_v}\langle t_v\rangle$ is a UPG polynomially growing 
automorphism of the free group $F^v$ with $f:\Gamma^v\to\Gamma^v$ an improved 
relative train track representative, and one of the following holds:
\begin{itemize}
 \item the image of $\phi_v$ in $\Out(F^v)$ is linearly growing or trivial, or
 \item there is a top strata decomposition of $G_v$ using a nonempty 
$\prec$--maximal set of superlinear edges of $\Gamma^v$, exactly as above.
\end{itemize}
Since each decomposition reduces the number of edges, $G_v$ has 
a hierarchy where all graph of groups decompositions at each level of the 
hierarchy are as described above.  The hierarchy terminates in groups 
$F^v\rtimes_{\Phi'}\integers$, where either $\Phi'\in\Out(F^v)$ is trivial --- 
i.e. the terminal vertex group is $F^v\times\integers$ --- or $\Phi$ is a 
linearly growing UPG.  In summary:
\begin{enumerate}[({Hier. }1)]
 \item There is a sequence $\Delta_0,\Delta_1,\ldots,\Delta_n$ of finite graphs, 
with $\Delta_n$ connected.
 \item $\Delta_n$ is the underlying graph of the top strata decomposition of $G$.
 
 \item For $0\leq i<n$, there is a bijection from  components of 
$\Delta_i$ to vertices of $\Delta_{i+1}$ such that each 
component of $\Delta_i$ is the underlying graph of the top strata decomposition 
of the corresponding vertex group of $\Delta_{i+1}$, if $i\geq 1$, or a single 
vertex, if $i=0$.

\item $\edges(\Delta_0)=\emptyset$, and $G_v=F_v\rtimes_{\Phi_v}\integers$ with $d(\Phi_v)\leq 1$, for $v\in\vertices(\Delta_0)$.
\end{enumerate}
For each $e\in\edges(\Delta_i)$, with vertices $v$ and $w$, the image of the edge group in $F^v\rtimes\integers$ is a cyclic subgroup that is elliptic in the topmost edge splitting of $F^v\rtimes\integers$, but the image of the edge group in $F^w\rtimes\integers$ is hyperbolic in the topmost edge splitting of $F^w\rtimes\integers$ (up to swapping $v,w$).
\end{rem}

\begin{lem}[Superlinear acylindricity]\label{lem:topmost-edges-acylindrical}
Let $\Delta_n$ be the top strata decomposition of $G$, as in Remark \ref{rem:topmost-edge-vertex-groups}.  Suppose that $n>0$, so that the edges $E_s,\ldots,E_r$ in $\Delta_n$ are superlinear.  Let $T$ be the Bass-Serre tree.  Then the $G$--action on $T$ is $2$--acylindrical (i.e. vertices at distance at least $3$ have trivially-intersecting stabilisers).
\end{lem}

\begin{proof}
Observe that the splitting of $G$ in the statement is exactly the output of \cite[Prop. 2.5]{AndrewHughesKudlinska}, so by \cite[Lem. 5.2]{KudlinskaValiunas}, the action of $G$ on $T$ is $2$--acylindrical.
\end{proof}

\subsection{Acylindrical splitting over $\integers^2$ in the linear case}\label{subsec:acylindrical-splitting-linear}
Let  $\Phi\in \Out(F)$ be a UPG that is at most linearly-growing.  We begin with a well-known splitting of $F\rtimes_\Phi\integers$ over $\integers^2$ edge groups; compare \cite[Prop. 5.2.2]{AndrewMartino:splitting}, \cite[Thm. 2.7]{DahmaniTouikan}, and \cite[Lem. 4.4]{AndrewHughesKudlinska}.

\begin{prop}[Linear mapping torus decomposition]\label{prop:split-over-tori}
Let $\Phi$ be a linearly growing UPG outer automorphism of the finite rank free group $F$ and let $G=F\rtimes_\Phi\langle t\rangle$ be its mapping torus.  Then there is a finite connected graph $\Delta$ such that $F$ and $G$ split as graphs of groups with underlying graph $\Delta$, and the following all hold:
\begin{enumerate}
\item For each $v\in\vertices(\Delta)$, the vertex group $F_v$ has finite rank at least $1$, and $G_v=F_v\times \langle t_v\rangle$, where $t_v\in Ft$.  If $\rank(F_v)>1$, then $F_v\times\langle t_v\rangle$ is a maximal free-times-$\integers$ subgroup of $G$.  If $\rank(F_v)=1$, then $F_v\times\langle t_v\rangle$ is a maximal $\integers^2$ subgroup of $G$.\label{item:linear-vertices}

\item If $v_0,v_1$ are distinct vertices and $\rank(F_{v_i})>1$ for $i\in\{0,1\}$, then $t_{v_0}\neq t_{v_1}$.\label{item:different-vertex-groups}

\item For each $e\in\edges(\Delta)$, the group $F_e=\langle r_e\rangle$ is a maximal infinite cyclic subgroup of $F$ and $G_e=\langle r_e\rangle \times\langle t_e\rangle$.  Let $r_e$ denote the image of $r_e$ in $F_{e^-}$ and let $p_e$ denote its image in $F_{e^+}$.  Then the image of $t_e$ in $G_{e^-}$ is $t_{e^-}$ and the image of $t_e$ in $G_{e^+}$ is $p_e^{k(e)}t_{e^+}$ for some $k(e)\in\integers$.\label{item:linear-edges}

\item If $e\in\edges(\Delta)$ and $v\in\vertices(\Delta)$ is incident to $e$, then the centre of $G_v$ is contained in $\image(G_e\to G_v)$.\label{item:linear-centre}  

\item The outer automorphism $\Phi$ has a representative preserving the graph of groups decomposition $\Delta$ of $F$, acting as a multitwist automorphism in the sense of \cite[Sec. 6]{CohenLustig}.\label{item:linear-dehn}

 \item $\Delta$ is bipartite, and $v\in\vertices(\Delta)$ is \textbf{black} if $F_v$ has rank $1$ and \textbf{white} otherwise, all edges are oriented from their black to their white vertex, and edge monomorphisms to black vertices are surjective; if $v$ is black, then $F_v=\langle r_v\rangle$ for some infinite-order $r_v$ and $r_e$ is sent by the edge map to $r_v$ for each $e$ with $e^-=v$.\label{item:linear-bipartite}

\item Letting $T$ be the Bass-Serre tree for the above decomposition of $G$ as a graph of groups, the action of $G$ on $T$ is $4$--acylindrical.  \label{item:linear-acylindrical}

\item Let $b\in\vertices(\Delta)$ be black and let $v,v'\in\vertices(\Delta)$ be white vertices joined to $b$ by distinct edges $e,e'$.  Then $\langle t_{b}r_e^{-k(e)}, t_{b}r_{e'}^{-k(e')}\rangle$ has finite index in $G_b$, so $k(e)\neq k(e')$.\label{item:linear-finite-index}

\item If $\tilde v\in\vertices(T)$ is white and $\tilde b,\tilde b'\in\vertices(T)$ are distinct black vertices adjacent to $\tilde v$, then $\stabilizer_G(\tilde b)\cap\stabilizer_G(\tilde b')$ is cyclic.\label{item:linear-cylic-intersection}
 
\end{enumerate}
\end{prop}

\begin{proof}
The existence of the graph of groups decomposition satisfying items \eqref{item:linear-vertices},\eqref{item:different-vertex-groups},  \eqref{item:linear-edges}, \eqref{item:linear-centre} is given by \cite[Thm. 2.7, Cor. 2.9, Prop. 3.4]{DahmaniTouikan} or by \cite[Prop. 5.2.2]{AndrewMartino:splitting} (with some of the listed properties also made explicit in \cite[Lem. 4.4]{AndrewHughesKudlinska}).  The fact that $\Phi$ can be realised as a Dehn twist on this graph of groups (item \eqref{item:linear-dehn}) is given by \cite[Thm. 2.7]{DahmaniTouikan} and also follows from the proof of \cite[Prop. 5.2.2]{AndrewMartino:splitting}.  Acylindricity (item \eqref{item:linear-acylindrical}) for the splitting from \cite{AndrewMartino:splitting} is proved in \cite[Lem. 4.4]{AndrewHughesKudlinska}, and the alternate proof from \cite{DahmaniTouikan} gives an acylindrical action by \cite[Prop. 3.2, Prop. 3.4]{DahmaniTouikan}.  The $2$--colouring and orientation (item \eqref{item:linear-bipartite}) follows the colour convention from \cite{DahmaniTouikan}.  Items \eqref{item:linear-finite-index}, \eqref{item:linear-cylic-intersection} follow from \cite[Prop. 3.2, Prop. 3.4]{DahmaniTouikan}.
\end{proof}

\begin{rem}[Explicit multitwist]\label{rem:multi-twist}
Let $\hat F'=\left(\Asterisk_{v\in\vertices(\Delta)} F_v\right)*\left(\Asterisk_{e\in\edges(\Delta)}\langle e\rangle\right),$
and let $\hat F$ be the quotient obtained by imposing the relations $ep_ee^{-1}=r_e$.  Let $\hat\phi:\hat F'\to\hat F'$ be the automorphism which is the identity on each $F_v$, and for each $e$, satisfies $\hat\phi(e)=ep_e^{k(e)}$.  Since $\hat \phi(r_e)=r_e$ and $\hat\phi(ep_ee^{-1})=ep_ee^{-1}$ for each $e$, this descends to an automorphism $\hat\phi$ of $\hat F.$  By Proposition \ref{prop:split-over-tori}, there is a choice of vertex $v\in \Delta$ such that $F$ is the fundamental group of the graph $\Delta$ of groups with the vertex/edge groups given in the proposition; $F$ naturally embeds as a $\hat\phi$--invariant subgroup of $\hat F$, by identifying vertices in $\Delta$ to a single vertex.  The restriction of $\hat\phi$ to $F\leq \hat F$ is an automorphism representing the outer autormorphism $\Phi$.  

Now, \cite[Thm. 2.7]{DahmaniTouikan} makes $\hat\phi$ a \emph{full one-sided substitution}: if $v\in \vertices(\Delta)$ is black and $e_1,e_2$ are distinct edges joining $v$ to   vertices $w_1,w_2$, respectively, then  $$e_1^{-1}e_2\neq p_{e_1}^{-k(e_1)}e_1^{-1}e_2p_{e_2}^{k(e_2)}$$
(as elements of $\hat F$).  This makes clear why $k(e_1)$ and $k(e_2)$ cannot both be $0$.

Consider the action of $G$ on $T$, and pull back the vertex colours from $\Delta$ to $T$.  If $w\in \vertices(T)$ is white, with stabiliser $G_w=F_w\times \langle t_w\rangle$, then for all $n\in\integers-\{0\}$, the element $t_w^n$ fixes exactly the vertices in $T$ at distance at most $2$ from $w$, by \cite[Prop. 3.4]{DahmaniTouikan} and, if $w\in\vertices(\Delta)$ is white and $e_1,e_2$ are distinct edges incident to $w$, then $\langle p_{e_1}\rangle \cap\langle p_{e_2}\rangle=\{1\}$.
\end{rem}

\subsubsection{Linear-growth topological representatives}\label{subsec:linear-top-rep}
It will be useful to work with an explicit graph of spaces corresponding to the splittings of $F$ and $G$ from Proposition \ref{prop:split-over-tori}.  Keeping the notation of that proposition, we define a graph of spaces in the usual way, with  additional properties extracted from the proof of \cite[Thm. 2.7]{DahmaniTouikan}:
\begin{enumerate}[(A)]
    \item The underlying graph is $\Delta$.  Let $\vertices_\bullet(\Delta)$ be the set of black vertices and $\vertices_\circ(\Delta)$ the white vertices.  No element of $\vertices_\bullet(\Delta)$ is a leaf of $\Delta$.\label{item:multitwist-verts}

    \item For each $v\in \vertices_\bullet(\Delta)$, the vertex space $X_v$ is a graph homeomorphic to $S^1$, with one vertex $x_v$.  We identify $F_v$ with $\pi_1(X_v,x_v)$. \label{item:multitwist-vertspace-black}

    \item For each $w\in\vertices_\circ(\Delta)$, the vertex space $X_w$ is a finite graph with a base vertex $x_w$.  We identify $F_w$ with $\pi_1(X_w,x_w)$. \label{item:multitwist-vertspace-white}

    \item For each $e\in\edges(\Delta)$ --- recall that $e^-\in\vertices_\bullet(\Delta)$ and $e^+\in\vertices_\circ(\Delta)$ --- let $\iota_e^-:X_{e^-}\times\{-1\}\to X_{e^-}$ be the identity and let $\iota_{e}^+:X_{e^-}\times\{+1\}\to X_{e^+}$ be obtained by subdividing $X_{e^-}\times\{+1\}$ so it is a combinatorial circle, and then mapping to the graph $X_{e^+}$ by a combinatorial immersion.  We let $p_e\to X_{e^+}$ be the immersed, closed, nontrivial combinatorial path that represents a conjugacy class of maximal cyclic subgroups in $F_{e^+}$ and has the property that $\iota_{e}^+$ corresponds to $p_e$.  (The paths $p_e$ come from improved relative train track maps, as explained in \cite{DahmaniTouikan}, and the key point is that they are immersed \emph{circuits}, i.e. immersions of $S^1$, not just immersed closed paths.)  These maps induce the edge monomorphisms from Proposition \ref{prop:split-over-tori}.  We let $r_e\to X_{e^-}$ be the path corresponding to $\iota_e^-$, i.e. the identity map on an oriented $1$--cycle. \label{item:multitwist-edges}

    \item If $w\in\vertices_\circ(\Delta)$ and $e_1,e_2$ are distinct edges terminating at $w$, then the immersed circuits $p_{e_1}\to X_w$ and $p_{e_2}\to X_w$ have the property that each component of the fibre product $p_{e_1}\otimes_{X_w}p_{e_2}$ is contractible. \label{item:multitwist-no-copy-p}

    \item The total space $Y$ is formed from $X:=\bigsqcup_{v\in\vertices(\Delta)}X_v$ by attaching, for each edge $e$, the cylinder $X_{e^-}\times[-1,1]$ using the given attaching maps $\iota_e^\pm$.  This makes $Y$ a $2$--complex whose $1$--skeleton consists of $X$, together with a copy of $\edges(\Delta)$.  Moreover, $\pi_1Y\cong F$. \label{item:multitwist-total}
\end{enumerate}

Define a continuous map $\varphi:Y\to Y$ to be the identity on $X$, and on each cylinder $X_{e^-}\times [-1,1]$, let $\varphi$ act as a $k(e)$--fold Dehn twist (in the direction of the orientation of $p_e$).  Then:
\begin{enumerate}[(A)]
\setcounter{enumi}{6}
    \item $\varphi$ induces the automorphism of $F\cong\pi_1Y$ from Remark \ref{rem:multi-twist}, i.e. a Dehn twist representing the outer automorphism $\Phi$. \label{item:multitwist-dehn}

    \item Let $M_\varphi$ be the mapping torus of $\varphi$.  Then $\pi_1M_\varphi\cong G=F\rtimes_\Phi\langle t\rangle$. \label{item:multitwist-mapping-torus}

    \item For each vertex $v\in\vertices_\bullet(\Delta)$ and distinct edges $e_1,e_2$ originating at $v$, the path $e_1^{-1}e_2$ is not path-homotopic in $Y$ to the path $\varphi(e_1^{-1}e_2)$.  (These paths have the same endpoints since $\varphi$ fixes $X\subset Y$ pointwise.)  This follows from the full one-sidedness condition in Remark \ref{rem:multi-twist}, and the same holds when $\varphi$ is replaced by any nonzero power. \label{item:multitwist-full}

    \item For each $v\in\vertices(\Delta)$, let $S_v$ be a $1$--cycle and let $M^v=X_v\times S_v$.  (So, $M^v$ is a torus when $v\in\vertices_\bullet(\Delta)$.)  Then $M_\varphi$ (is homotopy equivalent to a space that) splits as a graph of spaces with underlying graph $\Delta$, and:
    \begin{itemize}
        \item For each $v\in\vertices(\Delta)$, the vertex space is $M^v$.
        \item For each $e\in \edges(\Delta)$, the edge space is $X_{e^-}\times S_{e^-}=M^e$.
        \item We attach $M^e\times[-1,1]$ by identifying $M^e\times \{-1\}$ with $X_{e^-}\times S_{e^-}$ using the identity map, and gluing $M^e\times \{+1\}\cong X_{e^-}\times S_{e^-}$ to $X_{e^+}\times S_{e^+}$ using the map $I_e$ defined as follows.  The torus $X_{e^-}\times S_{e^-}$ is spanned by the loops $r_e$ (traverse $X_{e^-}$ once in the positive direction) and $t_e$ (traverse $S_{e^-}$ once in the positive direction).  Then there is a homeomorphism $X_{e^-}\times S_{e^-}\to S^1\times S^1$ given by sending $r_e$ to the vector $\colvec{0}{1}$ and $t_e$ to $\colvec{1}{0}$.  Now map $S^1\times S^1$ to $X_{e^+}\times S^{e^+}$ by first applying the homeomorphism $S^1\times S^1\to S^1\times S^1$ given by $\begin{pmatrix}
            1&0\\
            k(e)&1
        \end{pmatrix}$
        and then the map $p_e\times \mathrm{id}$.\label{item:multitwist-decomp}
    \end{itemize}
When we construct $G$--actions on cube complexes, we will do so by building immersed walls in $M_\varphi$ using this graph of spaces.
\end{enumerate}

\begin{notation}\label{notation:P-R}
Let $\mathcal R$ be the multiset of combinatorial cycles $r_e\to X$ for $e\in\edges(\Delta)$.  We say ``multiset'' to emphasise that $r_e$ and $r_{e'}$ are the same cycle whenever $e,e'$ are edges with a common (necessarily black) initial vertex.  Since each black vertex has at least two outgoing edges (recall that $\Delta$ has no black leaf), each $r_e$ \emph{always} coincides with $r_{e'}$ for some $e\neq e'$. If $e,e'$ are edges with distinct initial vertices, then $r_e$ and $r_e'$ lie in distinct vertex spaces.

Let $\mathcal P$ be the set of immersed circuits $p_e\to X$ for $e\in\edges(\Delta)$.  Note that $p_e\to X$ factors through the inclusion $X_{e^+}\to X$, so if $e,e'$ are edges with distinct terminal (necessarily white) vertices, then $p_e,p_{e'}$ have disjoint images.  If $e^+=(e')^+=v$ for some (white) vertex $v$, then $p_e\neq p_{e^+}$ (since $p_{e}\otimes_{X_v}p_{e'}$ has contractible components).
\end{notation}

\subsubsection{Future attaching elements}\label{subsec:future-attaching}
For later use (when the linear-growth case forms the base of an induction on the superlinear hierarchy), we record the following definition and lemma:

\begin{defn}\label{defn:linear-hyperbolics}
Let $\mathcal Q$ be a finite set of elements of $G$ of the form $qt$, where $q\in F$, such that each $qt$ acts hyperbolically on the Bass-Serre tree $T$.  We call $\mathcal Q$ a set of \emph{future attaching elements} provided no two distinct elements of $\mathcal Q$ are conjugate in $G$.
\end{defn}

As usual, a group $E$ is \emph{elementary} if it is virtually cyclic (including the case where $|E|<\infty$).

\begin{lem}\label{lem:future-attaching-acyl}
Let $\mathcal Q$ be a set of future attaching elements in $G$.  Then there exists $\chi<\infty$ such that for all $qt,q't\in\mathcal Q$ and $g\in G$, the following hold:
\begin{itemize}
    \item Let $A_{qt}$ be the axis of $qt$ in $T$.  Then $\langle qt\rangle$ is a maximal elementary subgroup of $G$, and $\stabilizer_G(A_q)=\langle qt\rangle$.
    \item   Either $gA_{qt}\cap A_{q't}$ has diameter bounded by $\chi$, or $gqtg^{-1}=q't$.  In the latter case, $q=q'$ and $g\in\langle qt\rangle$.
\end{itemize}
\end{lem}

\begin{proof}
Recall that $G$ acts on $T$ acylindrically, by Proposition \ref{prop:split-over-tori}, so $\stabilizer_G(A_q)$ is a maximal elementary subgroup, and it is cyclic since $G$ is torsion-free.  Since $q\in F$, the element $qt$ cannot be a proper power in $G$, and therefore $\langle qt\rangle$ is a maximal cyclic subgroup, and hence coincides with $\stabilizer_G(A_q)$.  Acylindricity and finiteness of $\mathcal Q$ also provide $\chi$ such that $\diam_T(gA_{qt}\cap A_{q't})>\chi$ implies $gA_{qt}=A_{q't}$, so $\langle gqtg^{-1}\rangle=\langle q't\rangle$, and by considering the retraction $G\to\langle t\rangle$ with kernel $F$ we find $gqtg^{-1}=q't$.  Since distinct elements of $\mathcal Q$ are not conjugate in $G$, we have $q'=q$ and, since the maximal elementary subgroup $\langle qt\rangle$ is self-normalising (by acylindricity), $g\in\langle qt\rangle$.
\end{proof}

\begin{lem}\label{lem:cubulate-linear-finite-index}
Let $\mathcal Q\subset G$ be a (possibly empty) set of future attaching elements.  Let $\mathcal Q_1$ be an elevation (see Definition \ref{defn:finite-set-elevation}) of $\mathcal Q$ to some finite-index subgroup $G_1\leq G$.

Then for each $g\in \mathcal Q_1$, $\langle g\rangle$ is a maximal elementary subgroup of $G_1$ and $\langle g\rangle=\stabilizer_{G_1}(A_g)$, where $A_g$ is the axis of $g$ on $T$.  Moreover, if $g,g'\in \mathcal Q$, then any two $G_1$--translates of $A_g,A_{g'}$ either coincide, or their intersection has diameter at most $\chi$.

Suppose, moreover, that $G_1$ acts freely on a CAT(0) cube complex $C_1$ in such a way that $\mathcal Q_1\cup \mathcal S$ is wall-independent, where $\mathcal S$ is an elevation of $\{t\}$ to $G_1$, and $\|g\|_{C_1}=\|s\|_{C_1}$ for all $g\in\mathcal Q_1$ and $s\in\mathcal S$.  Then $G$ acts freely on a CAT(0) cube complex $C$ such that $\mathcal Q\cup\{t\}$ is wall-independent and $\|t\|_C=\|qt\|_C$ for all $qt\in\mathcal Q$.
\end{lem}

\begin{proof}
By Definition \ref{defn:finite-set-elevation}, each $g\in\mathcal Q_1$ has the property that $\langle g\rangle = \langle aqta^{-1}\rangle \cap G_1$ for some $a\in G,qt\in\mathcal Q$, so $\langle g\rangle$ is a maximal cyclic subgroup of $G_1$ by Lemma \ref{lem:future-attaching-acyl}.  By the same lemma, $\langle g\rangle=\stabilizer_{G_1}(A_g)$, and the bound on axis intersections follows.  

Now let $G_1$ act freely on the CAT(0) cube complex $C_1$ with the properties hypothesised in the statement of the lemma.  Note that $\mathcal Q_1\cup\mathcal S$ is an elevation of $\mathcal Q\cup\{t\}$.  So any $G$--conjugates of elements in $\mathcal Q\cup \{t\}$ have powers that are conjugate in $G_1$ into $\mathcal Q_1\cup\mathcal S$, so our hypotheses imply that these $G$--conjugates have the same virtual translation length on $C_1$.  Hence Lemma \ref{lem:star} provides the $G$--action on $C$ with the desired properties.
\end{proof}

\subsubsection{Big finite covers of $Y$ and $M_\varphi$ where edge spaces embed}\label{subsubsec:linear-finite-covers}

\begin{defn}\label{defn:P-embedded-starred}
Let $c:\widehat Y\to Y$ be a finite-degree cover.  Let $\widehat{\Delta}$ be the underlying graph of the decomposition of $\widehat Y$ induced by the above decomposition of $Y$ with underlying graph $\Delta$.  Then:
\begin{itemize}
    \item $c$ is \emph{$\varphi$--invariant} if the homeomorphism $\varphi:Y\to Y$ lifts to $\hat \varphi:\widehat Y\to\widehat Y$.

    \item $c$ is \emph{$\mathcal P$--clean } if for all $p_e\in\mathcal P$, every elevation $\hat p_e$ of $p_e$ to $\widehat Y$ is an embedded circuit.  (Note that this is automatically true for each elevation of each $r_e\in\mathcal R$.)

    \item $c$ is \emph{starred} if the bipartite graph $\widehat{\Delta}$ contains no $2$--cycle.
\end{itemize}
We say $Y$ is $\mathcal P$--clean  [resp. starred] if the identity  $Y\to Y$ is $\mathcal P$--clean  [resp. starred].
\end{defn}

\begin{rem}
The terminology \emph{clean} is consistent with that in \cite[Def. 4.5]{Wise:omnipotence}; we highlight $\mathcal P$ to emphasise which edge spaces  can fail to embed (each $r_e\in\mathcal R$ is always an embedding).
\end{rem}

\begin{rem}\label{rem:invariant-cover}
Let $c_1:Y_1\to Y$ be $\varphi$--invariant and $\varphi:Y_1\to Y_1$ a lift.  Then since the graph $X_1=c_1^{-1}(X)$ is finite and $\varphi$ fixes $X$ pointwise, there is $N>0$ such that $\varphi^N:Y_1\to Y_1$ fixes $X_1$ pointwise.  Hence the induced decomposition of $Y_1$ as a graph of spaces (the vertex spaces are the components of $X_1$ and the edge spaces are circles) corresponds to a graph $\Delta_1$ of groups decomposition of a $\varphi^N$--invariant finite-index subgroup $F_1$ of $F$, and $Y_1$ has all the properties \eqref{item:multitwist-verts}--\eqref{item:multitwist-total}.  The map $\varphi^N$ and its mapping torus $M_{\varphi^N}$ have all the properties \eqref{item:multitwist-dehn}--\eqref{item:multitwist-decomp}.  The fundamental group of this mapping torus is $G_1=F_1\rtimes_{\Phi^N}\langle s\rangle$, and the inclusion $F_1\to F$ and the map $s\to t^N$ induce an embedding of $G_1$ as a finite-index subgroup of $G$ (there is a natural covering map $M_{\varphi^N}\to M_\varphi$ extending the cover $c_1:Y_1\to Y$ and inducing $G_1\hookrightarrow G$).  
\end{rem}

Since $\pi_1Y$ has finite rank, there are finitely many subgroups of any given index, so:

\begin{lem}\label{lem:invariant-covers}
For any finite cover $c:\widehat Y\to Y$, there exists a finite regular cover $c_1:Y_1\to Y$ that is $\varphi$--invariant and factors through $c$.
\end{lem}

\begin{lem}\label{lem:finite-P-embedded-starred-cover}
There exists a finite regular cover $c:\widehat Y\to Y$ such that any finite cover $c_1:Y_1\to Y$ factoring through $c$ is $\mathcal P$--clean  and starred.
\end{lem}

\begin{proof}
Note that if $c$ is a $\mathcal P$--clean  cover, then so is any further finite cover, and the same is true for the property of being starred.  So it suffices to produce a $\mathcal P$--clean  starred finite cover.  Each of the finitely many $p_e\in\mathcal P$ is an immersion $p_e:\mathbb S^1\to X_{e^+}$, so $p_e\to Y$ is $\pi_1$--injective.  So by separability of $\langle p_e\rangle$ in the free group $\pi_1Y$, $p_e$ lifts (using \cite{Scott}) to an embedding in a finite cover of $Y$, and by passing to a further finite cover which is a regular cover of $Y$, the same is true for all elevations of $p_e$.  Repeating for each $p_e\in\mathcal P$ and passing to the common cover yields a finite regular cover $c_0:Y_0\to Y$ that is $\mathcal P$--clean.  (This part also follows from \cite{Wise:omnipotence}.) 

Now, for each $v\in\vertices(\Delta)$, define a complex $Z_v$ as follows.  Begin with $X_v$ and attach the edge space $X_{e}\times [-1,1]$ using the appropriate attaching map $r_e$ or $p_e$ from the construction of $Y$.  If $e$ is such that $e^\pm=v$, then attach a copy of $X_{e^\mp}$ to $X_e\times\{\mp1\}$ using the appropriate attaching map, $r_e$ or $p_e$.  The resulting complex $Z_v$ admits a natural combinatorial map $Z_v\to Y$ induced by the inclusions $X_v$, the various $X_e$, and the various $X_{e^{\mp}}$ into $Y$.  This map fails to be injective exactly when there are two distinct edges incident to $v$ that form a bigon; in this case we have edges $e,f$ such that, say, $X_{e^-}=X_{f^-}$, and the subspaces $X_{e^-},X_{f^-}$ of $Z_v$ are mapped to the same vertex space of $Y$.  Note that $Z_v\to Y$ is $\pi_1$--injective and $Z_v$ is compact.  So, again using subgroup separability of $\pi_1Y$ and \cite{Scott}, we obtain a finite cover of $Y$ where $Z_v$ embeds.  Repeating for each vertex and taking the common cover, we obtain a regular cover $c_0':Y_0'\to Y$ where all $Z_v$ embed.  Taking the cover corresponding to the normal core of the intersection of the subgroups of $\pi_1Y$ corresponding to $c_0$ and $c_0'$ gives the desired $c$.
\end{proof}

\begin{rem}\label{rem:assume-P-embedded}
By Lemma \ref{lem:finite-P-embedded-starred-cover}, Lemma \ref{lem:invariant-covers}, and Remark \ref{rem:invariant-cover}, in any situation where it is sufficient to prove that some property of the linear mapping torus $G$ holds virtually (for instance, by Lemma \ref{lem:star}, to cubulate), we can assume that $G=\pi_1M_\varphi$, with $Y$ being $\mathcal P$--clean  and starred. 
\end{rem}

\subsection{Fast and unbranched mapping tori}\label{subection:unbranched}

\begin{defn}[Fast]\label{defn:fast}
A polynomially-growing $\Phi\in\Out(F)$  is \emph{fast} if the polynomial growth rate $d$ of $\Phi$ satisfies $d=\rank(F)-1$.
\end{defn}

Using \cite[Lem. 2.16]{Macura:detour} shows $d\leq \rank(F)-1$, so the content of Definition \ref{defn:fast} is the reverse inequality. The next definition is from \cite{BGGH}:

\begin{defn}\label{defn:unbranched}
A \emph{block} in a group $G$ is a subgroup $B$ containing a finite index subgroup isomorphic to $F\times \integers$ where $F$ is a nonabelian free group.  A \emph{highest block} is a block $B$ such that for all blocks $B'$ with $[B:B\cap B']<\infty$, we have $[B':B\cap B']<\infty$. 

The group $G$ is \emph{unbranched} if for all triples $B_1,B_2,B_3$ of highest blocks such that $[B_i:B_i\cap B_j]=\infty$ for all $i\neq j$, the intersection $B_i\cap B_j$ is elementary.
\end{defn}

We use the unbranched property via the following lemma, from \cite{BGGH}:

\begin{lem}\label{lem:unbranched-char}
Let $\Phi\in\Out(F)$ be a polynomially growing automorphism with growth rate $d\geq 1$ and let $G$ be the mapping torus of $\Phi$.  Let $k>0$ be such that $\Phi^k$ is UPG, and let $G'$ be the mapping torus of $\Phi^k$.  Then $G$ is unbranched if and only if $G'$ is, and:
\begin{itemize}
    \item if $d=1$, then $G'$ is unbranched if and only if the graph $\Delta$ from Proposition \ref{prop:split-over-tori} has the property that every black vertex has valence exactly $2$;

    \item  if $d\geq 2$, then $G'$ is unbranched if and only if each terminal subgroup in the superlinear hierarchy from Remark \ref{rem:topmost-edge-vertex-groups} is unbranched.
\end{itemize}
\end{lem}

\begin{proof}
Equivalence of the unbranched property for $G,G'$ is immediate from the definition.  The assertion about linear UPG mapping tori is \cite[Lem. 4.8]{BGGH}.  The assertion about the superlinear hierarchy follows from \cite[Prop. 3.3]{BGGH} and the fact that the splittings in the superlinear hierarchy are acylindrical (Lemma \ref{lem:topmost-edges-acylindrical}).
\end{proof}

\begin{lem}\label{lem:fast-implies-unbranched}
If $\Phi$ is fast, then $G$ is unbranched.
\end{lem}

\begin{proof}
By Lemma \ref{lem:unbranched-char}, we can assume that $\Phi$ is a UPG element.  The growth rate $d$ satisfies $d\leq \rank(F)-1$.  

Suppose $d>1$ and consider the top strata decomposition.  Each vertex group $G_v$ has the form $G_v=F_v\rtimes_{\phi_v}\langle t_v\rangle$ where $\phi_v$ has polynomial growth rate $d_v$.  Moreover, $F_v$ is a free factor of $F$ and $F_v\cap F_w=\{1\}$ for $v\neq w$.   Note that $d=\max_v d_v+1.$ Let $r_v=\rank(F_v)$ and observe that
$$d\leq \sum_v(d_v+1)\leq \sum_vr_v\leq \rank(F)\leq d+1,$$
since $\Phi$ is fast.  If $\rank(F)=1+\sum_v r_v$, then there is a unique vertex $v$, and $d_v=d-1$, so $\phi_v$ is fast.  Hence, by induction, $G_v$ is unbranched, so $G$ is unbranched by Lemma \ref{lem:unbranched-char}.  If $\rank(F)=\sum_v r_v$, then there are two vertices $v,w$ joined by a unique edge that grows with polynomial growth rate $d$.  We have, say, $d_v+1=d$ and $d_w+1\leq d$, and $r_v+r_w=d+1$, and $r_v\geq d$.  If $r_v=d$, then $\phi_v$ is fast and $r_w=1$, so $G_v$ is again unbranched by induction and $G_w\cong\integers^2$ is unbranched by definition, so again $G$ is unbranched by Lemma \ref{lem:unbranched-char}.  If $r_v>d$, then $r_w=0$, contradicting that the top strata decomposition is nontrivial.  

It remains to consider the base case, $d=1$.  In this case, $\rank(F)=2$, so $\Phi$ is represented by a geometric automorphism $\phi$ which is induced by a homeomorphism $h:S\to S$ of a once-punctured torus $S$.  Passing to a power, we can assume that $h$ is a power of a Dehn twist about a curve $\alpha$ on $S$.  Letting $M$ be the mapping torus of $h$, we see that cutting $M$ along the torus $\alpha\times \mathbb S^1$ produces the splitting from Proposition \ref{prop:split-over-tori}.  Since $M$ is a manifold, the black vertex (which is unique) has valence $2$, so $G$ is unbranched by Lemma \ref{lem:unbranched-char}.
\end{proof}

\begin{cor}\label{cor:unbranched-separable}
Suppose that $d=1$ and $G$ is unbranched and $\Phi$ is UPG.  Let $\Delta$ be the graph of groups decomposition from Proposition \ref{prop:split-over-tori}.  Let $v\in\vertices(\Delta)$, let $e\in \edges(\Delta)$, and let $H\in \{G_v,G_e\}$.  Let $H'$ be a finite-index subgroup of $H$.  Then there exists a finite-index subgroup $G'\leq G$ such that $G'\cap H\leq H'$.  
\end{cor}

\begin{proof}
From Lemma \ref{lem:unbranched-char}, in the case where $d=1$, we have that $G'$ is unbranched exactly when the graph of groups from Proposition \ref{prop:split-over-tori} is a \emph{simple admissible graph of groups} in the sense of \cite[Defn. 2.2]{Nguyen:separability}.  Now apply \cite[Thm. 1.2]{Nguyen:separability}.
\end{proof}

\section{Sufficient conditions for cubulation in the linear case}\label{sec:sufficient-conditions-linear}
Let $\Phi\in\Out(F)$ be a linearly-growing element.  As usual, let $G=F\rtimes_\Phi\langle t\rangle$.

\subsection{Finite-index subgroup with $\Delta$-edge measurements}\label{subsec:linear-finite-index}
By Remark \ref{rem:make-upg}, $\Phi$ has a positive power that is UPG, so by Lemma \ref{lem:star}, we can assume that $\Phi$ is itself UPG.  Let $\Delta$ be the underlying graph of the splitting from Proposition \ref{prop:split-over-tori}, let $Y$ be the graph of spaces from Subsection \ref{subsec:linear-top-rep} (satisfying all of the listed conditions in that subsection) with $\pi_1Y=F$, let $\varphi:Y\to Y$ be the multitwist, and let $M_\varphi$ be the mapping torus; recall from the same section how $M_\varphi$ splits as a graph of spaces with underlying graph $\Delta$ (see item \eqref{item:multitwist-decomp}).  Let $\mathcal P,\mathcal R$ be as in Notation \ref{notation:P-R}.  By Lemma \ref{lem:star} and Remark \ref{rem:assume-P-embedded}, we can assume that $Y$ is $\mathcal P$--clean and starred.

\begin{defn}[Edges at a vertex]\label{defn:E-edge-sets}
For each $w\in\vertices_\circ(\Delta)$ and each $b\in\vertices_\bullet(\Delta)$, let
$$\mathbf E(w)=\{e\in\edges(\Delta):e^+=w\}\ \ \text{and }\ \mathbf E(b)=\{e\in\edges(\Delta):e^-=b\}.$$
\end{defn}

For each $w\in\vertices(\Delta)$, fix a base vertex $x_w\in X_w$ and identify $F_w$ with $\pi_1(X_w,x_w)$.  

\begin{lem}\label{lem:partition-edges}
Up to replacing $G$ by a further finite-index subgroup, we may assume the following.  For all $w\in\vertices_\circ(\Delta)$, there exists a partition
$$\mathbf E(w)=\bigsqcup_{i\in J_w}\mathbf E_i(w)$$
with all of the following properties:
\begin{enumerate}
    \item \label{item:edge-partition-coeffs}  $k(e)=k(e')$ for all $i\in J_w$ and all $e,e'\in \mathbf E_i(w)$.  Moreover, let $b=e^-,b'=(e')^-$.  Then there is a bijection $f\mapsto f'$ from $\mathbf E(b)\to \mathbf E(b')$ such that $k(f)=k(f')$.
    \item \label{item:edge-partition-maps} For all $i\in J_w$, there is a homomorphism $\pi^i_w:\pi_1(X_w,x_w)\to\integers$ such that for all $e\in \mathbf E_i(w)$, we have $\pi^i_w(g)\neq 0$ for $g\in\pi_1(X_w,x_w)$ in the conjugacy class represented by $p_e$.  
    \item \label{item:edge-partition-kernel} For $f\in\mathbf E(w)-\mathbf E_i(w)$, and $g$ in the conjugacy class represented by $p_f$, we have $\pi^i_w(g)=0$.
\end{enumerate}
\end{lem}

\begin{proof}
Fix $w$.  By \cite[Cor. 3.18]{Wise:omnipotence}, there is a finite regular cover $c:\widehat X_w\to X_w$ such that: for each $e\in\mathbf E(w)$, there is a homomorphism $\homology_1(\widehat X_w)\to \integers$ that is nonzero on each $1$--cycle represented by an elevation of $p_e$ but which vanishes on each $1$--cycle which is an elevation of $p_f$ when $f\in\mathbf E(w)-\{e\}$.  The existence of such homomorphisms persists in further finite covers. 

For each $w$, the map $\widehat X_w\to X_w\to Y$ is $\pi_1$--injective.  Since $\pi_1Y$ is free and $\widehat X_w$ is compact, subgroup separability provides a finite cover $\widehat Y_w\to Y$ such that $\widehat X_w\to X_w\to Y$ lifts to an embedding $\widehat X_w\to \widehat Y_w$.  Let $\widehat Y\to Y$ be the finite characteristic cover corresponding to the intersection over all $w\in\vertices_\circ(\Delta)$ of the characteristic cores of the subgroups $\pi_1\widehat Y_w$ of $\pi_1Y$.  Then for all $w$, every lift $X'_w\to Y'$ of $X'_w\to X_w\to Y$ has $X'_w\to X_w$ factoring through $\widehat X_w$. 

Now, $\widehat Y\to Y$ is a characteristic cover, so $\phi$ has a positive power lifting to a map decomposing as the product of Dehn twists in the edge-cylinders of $\widehat Y$.  Let $\widehat M_\phi$ be the mapping torus of this lift, so that $\widehat M_\phi$ is a finite regular cover of $M_\phi$; the corresponding subgroup $\widehat G\leq G$ is the finite-index subgroup from the statement.

Let $\widehat\Delta$ be the underlying graph of the graph of spaces decomposition of $\widehat Y$ induced by our decomposition of $Y$.  Each vertex space $\widehat X_w$ in $\widehat Y$ has its incident edges partitioned according to their images in $Y$.  This is the partition $J_w$ from the statement.  By construction of $\widehat X_w$, these partitions satisfy conditions \eqref{item:edge-partition-maps} and \eqref{item:edge-partition-kernel} from the statement.  Now, if $e,e'\in\edges(\widehat\Delta)$ are incident to a common white vertex $w$, and $e,e'\in \mathbf E_i(e)$ for some $i\in J_w$, then by the definition of the partition, there is an edge $\bar e$ of $\Delta$ such that the attaching circles in $\widehat X_w$ of the cylinders in $\widehat Y$ corresponding to the edges $e,e'$ are elevations $p_e,p_{e'}$ of the same cycle $p_{\bar e}$ in $X_w$.  Since $\widehat M_\phi\to M_\phi$ is a regular cover, we also get \eqref{item:edge-partition-coeffs}.  (The quotient $\widehat\Delta\to\Delta$ restricts to a graph isomorphism on the star of each black vertex, using Proposition \ref{prop:split-over-tori}.\eqref{item:linear-bipartite} and the fact that $Y$ is starred.)  
\end{proof}

\begin{defn}
We call the homomorphisms $\pi^i_w$ from Lemma \ref{lem:partition-edges} \emph{$\Delta$-edge measurements}.
\end{defn}

\begin{rem}\label{rem:disjoint-cycles-in-partition}
The proof of Lemma \ref{lem:partition-edges} reveals an additional fact about the $\mathbf E_i(w)$, needed in Proposition \ref{prop:exit-arrangement} but not in Proposition \ref{prop:general-linear}, namely that for all $i\in J_w$ and distinct $f,f'\in\mathbf E_i(w)$, the cycles $p_f,p_{f'}$ in $X_w$ are disjoint.  Indeed, recall that $p_f,p_{f'}$ are distinct elevations of a single embedded cycle $p_e$ in the original $X_w$ (before we passed to the cover $\widehat X_w$ and renamed it $X_w$).  Indeed, the original $p_e$ was embedded because the original $X_w$ belonged to the $Y$ that we had already arranged to be $\mathcal P$--clean.
\end{rem}

\subsection{Cubulation using good edge measurements}\label{subsec:good-edge-measurements}
Here is a sufficient condition for cubulating, phrased in terms of edge measurements:

\begin{prop}\label{prop:general-linear}
Suppose that (after passing to a suitable finite index subgroup) the $\Delta$--edge measurements $\pi^i_w$ can be chosen so that for all $w\in\vertices_\circ(\Delta)$ and $i\in J_w$, there exists $K_i\in\integers-\{0\}$ such that $\pi^i_w(g)=K_i$ for all $e\in\mathbf E_i(w)$ and all $g$ in the conjugacy class represented by $p_e$.  Then $G$ acts freely on a CAT(0) cube complex.
\end{prop}

\begin{rem}[Excuses]\label{rem:edge-measure}
The hypothesis about $\Delta$-edge measurements in Proposition \ref{prop:general-linear} is a byproduct of the proof, but we will see that it is a generalisation of the property of $G$ being tubular, so Proposition \ref{prop:general-linear} is a generalisation of the fact that tubular free-by-cyclic groups are cubulated \cite[Thm. 2.1, Thm. 2.4]{Button:tubular-cubes}.  Proposition \ref{prop:general-linear} also motivates Proposition \ref{prop:exit-arrangement}, which we will use to cubulate mapping tori of superlinear automorphisms provided the bottom-level vertex groups in the superlinear hierarchy are unbranched, with suitable edge-measurements.  We suspect the proof of Proposition \ref{prop:general-linear} can be made to work under weaker hypotheses.
\end{rem}

\begin{proof}[Proof of Proposition \ref{prop:general-linear}]
 We may assume that each $K_i>0$.  Let $K=\lcm\{K_i:i\in\bigsqcup_{w\in\vertices_\circ(\Delta)}J_w\}$.  By replacing each $\pi^i_w$ with $\frac{K\pi^i_w}{K_i}$, we can assume that $K_i=K$ for all $i$.  Choose $L\in\naturals$ such that $L>|m_i|$ for all $i$, where $m_i$ is the constant from Lemma \ref{lem:partition-edges} with $k(e)=m_i$ for all $e\in\mathbf E_i(w)$.

 \textbf{$\vertices_\circ(\Delta)$ cubulations.}  Fix $w\in\vertices_\circ(\Delta)$.  Define $\alpha_i:F_w\times \langle t_w\rangle\to\integers$ by
 $$\alpha_i(ft_w^x)=\pi^i_w(f)-K(L+m_i)x$$
 for $x\in\integers$ and $f\in F_w$.  Define $\beta_i$ in exactly the same way, except with $L$ replaced by $-L$.

 \begin{claim}\label{claim:kernel-alpha-linear}
 Let $j\in J_w$ and let $e\in \mathbf E_j(w)$.  Then $\kernel(\alpha_i)\cap \langle p_e,t_w\rangle =\langle p_e^{L+m_i}t_w\rangle$ if $i=j$ and $\kernel(\alpha_i)\cap \langle p_e,t_w\rangle=\langle p_e\rangle$ if $i\neq j$.  The same holds replacing $\alpha_i$ by $\beta_i$ and $L$ by $-L$.
 \end{claim}
\renewcommand{\qedsymbol}{$\blacksquare$}
 \begin{proof}[Proof of Claim \ref{claim:kernel-alpha-linear}]
 If $j\neq i$ and $e\in\mathbf E_j(w)$, then $\pi^i_w(p_e)=0$, so $\alpha_i(p_e)=0$.  If $i=j$, then $\alpha_i(p_e)=\pi^i_w(p_e)=K$.  Moreover, $\alpha_i(t_w)=-K(L+m_i)$, and similarly for $\beta_i$ and $-L$.
 \end{proof}

Let $\Lambda^w_i,\Omega^w_i$ be finite graphs equipped with continuous maps $\Lambda^w_i,\Omega^w_i\to X_w\times S_w$ whose compositions with the natural projection to $X_w$ are covers, and whose compositions with the natural projection to $S_w$ factor through covers, and which respectively induce the inclusions of $\kernel\alpha_i$ and $\kernel\beta_i$ in $F_w\times\langle t_w\rangle$.  Then $\Lambda^w_i$ and $\Omega^w_i$ are immersed walls in $X_w\times S_w$ and the immersions lift to embeddings $\widetilde \Lambda^w_i,\widetilde \Omega^w_i\to\widetilde X_w\times \widetilde S_w$ of universal covers whose images are walls.  Let $\widetilde C^w_i$ be the CAT(0) cube complex dual to the wallspace with underlying space $\widetilde X_w\times \widetilde S_w$ and walls the various $F_w\times \langle t_w\rangle$--translates of $\widetilde \Lambda^w_i,\widetilde \Omega^w_i$.

For each $e\in \mathbf E_i(w)$, there is an embedded torus $p_e\times S_w\to X_w\times S_w$, and the preimage of $\Lambda^w_i$ in this torus is a circle representing $p_e^{L+k(e)}t_w$, and for $\Omega^w_i$, the preimage circle represents $p_e^{-L+k(e)}t_w$.  Hence $\langle p_e,t_w\rangle$ acts on $\widetilde C^w_i$ freely, by Lemma \ref{lem:free}.

\begin{claim}\label{claim:translation-length-linear}
For all $i\in J_w$, we have $\|t_w\|_{\widetilde C^w_i}=2L$.
\end{claim}

\begin{proof}[Proof of Claim \ref{claim:kernel-alpha-linear}]
The element $t_w$ is represented by an embedded closed path $\sigma$ in $p_e\times S_w$ such that the translation length of $t_w$ is the sum of the intersection numbers, in the torus $p_e\times S_w$, of $\sigma$ with the cycles $p_e^{\pm L+k(e)}t_w$.  View the torus as the quotient of $\reals\times \reals$ by the group $\langle p_e,t_w\rangle$ of affine isometries where $p_e$ is identified with $(0,1)^T$ and $t_w$ with $(1,0)^T$.  Then
$$\|t_w\|_{\widetilde C^w_i}=\left|\det\begin{pmatrix}
    1&1\\
    L+k(e)&0
\end{pmatrix}\right|+\left|\det\begin{pmatrix}
    1&1\\
    -L+k(e)&0
\end{pmatrix}\right|=2L,$$
where the last equality holds since $L>|k(e)|$.
\end{proof}

For $j\neq i$, the homomorphisms $\alpha_j,\beta_j$ send $p_e$ to $0$, for $e\in\mathbf E_i(w)$.  Hence, letting $A_j$ be a combinatorial axis for $t_w$ in $\widetilde C^w_j$, we have that $A_j$ is a $\langle p_e,t_w\rangle$--invariant combinatorial line on which $t_w$ acts as a translation with displacement $2L$.

Let $M>\max_w(|J_w|-1)$, let $M_w=M-(|J_w|-1)$, and let $\widetilde C^\dagger_w$ be the standard tiling of $\Euclidean^{M_w}$ by unit cubes (so it is the product of $M_w$ combinatorial lines).  Let $F_w\times \langle t_w\rangle$ act on $\widetilde C^\dagger_w$ by declaring $F_w$ to act trivially and $t_w$ to act by a combinatorial translation of displacement $2L$ on each factor.  Define
$$\widetilde C_w=\widetilde C_w^\dagger\times\prod_{j\in J_w}\widetilde C^w_j,$$
equipped with the diagonal action of $F_w\times \langle t_w\rangle$.

For each $i\in J_w$ and $e\in\mathbf E_i(w)$, consider the action of $\langle p_e,t_w\rangle$ on $\widetilde C_w$.  By Lemma \ref{lem:generalised-axis}, there is a $\langle p_e,t_w\rangle$--invariant isometric subcomplex $\widetilde Y'_e\subset \widetilde C_i^w$ which is isomorphic to the cube complex dual to the pair of parallelism classes of walls in $\langle p_e,t_w\rangle$ corresponding to the subgroups $\langle p_e^{\pm L+k(e)}t_w\rangle$.  Let 
$$\widetilde Y_e=\widetilde C^\dagger_w\times \widetilde Y_e'\times \prod_{j\neq i}A_j.$$
Then $\widetilde Y_e$ is a $\langle p_e,t_w\rangle$--invariant isometric subcomplex of $\widetilde C_w$.

\textbf{$\vertices_\bullet(\Delta)$ cubulations.}  Let $b\in\Delta$ be a black vertex, so we have $G_b=\langle r_b\rangle\times \langle t_b\rangle$, which acts on $\reals^2$ via $r_b\mapsto (0,1)^T$ and $t_b\mapsto (1,0)^T$.  Let $\widetilde C_b'$ be the CAT(0) cube complex dual to the wallspace whose underlying set is $\reals^2$ and whose walls are the $G_b$--translates of the lines $\{z(1, L)^T:z\in\reals\}$ and $\{z(1, -L)^T:z\in\reals\}$.

Let $C''_b$ be the product of $M$ copies of the standard tiling of $\reals$ by $1$--cubes, with $G_b$ acting on $\widetilde C_b''$ diagonally, where the action on each factor has $r_b$ acting trivially and $t_b$ acting as a translation with displacement $2L$.  Let $\widetilde C_b=\widetilde C_b'\times\widetilde C_b''$, with the diagonal $G_b$--action.

\textbf{Shape comparison.}  Let $e\in\edges(\Delta)$ and let $b=e^-$ and $w=e^+$.  Then the edge map $I_e:G_b\to G_w$ is $I_e(r_b)=p_e$ and $I_e(t_b)=p_e^{k(e)}t_w$.  Since $I_e(r_bt^{\pm L})=p_e^{\pm L+k(e)}t_w$, there is an $I_e$--equivariant cubical isomorphism $\psi'_e:\widetilde C_b'\to \widetilde Y'_e\subset \widetilde C_i^w$, where $i\in J_w$ is such that $e\in\mathbf E_i(w)$.

Furthermore, $\widetilde C_b''$ and $\widetilde D_w=\widetilde C_w^\dagger\times \prod_{j\neq i}A_j$ are both isomorphic to the product of $M$ copies of the standard tiling of $\reals$ by $1$--cubes; $t_b$ acts as a length--$2L$ translation in each factor of the former and $t_w$ acts as a length--$2L$ translation in each factor of the latter.  Since $p_e$ acts on $\widetilde D_w$ trivially, it follows that there is a cubical isomorphism $\psi_e'':\widetilde C_b''\to\widetilde D_w$ that is $I_e$--equivariant.  So, by Lemma \ref{lem:shapes-and-products}, the $G_b$--action on $\widetilde C_b$ and the $G_w$--action on $\widetilde C_w$ have the same shape for $I_e$. 

\textbf{Conclusion.}  Finally, for each $w\in\vertices_\circ(\Delta)$, let $\widetilde C^\times_w$ be a CAT(0) cube complex with a $G_w$--action such that $\|a\|_{\widetilde C^\times_w}=0$ for exactly those $a\in G_w$ that are conjugate into $\langle p_e,t_w\rangle$ for some $e\in\mathbf E(w)$.  This exists by Lemma \ref{lem:small-cancellation-improper}.  Let $\widetilde Z_w=\widetilde C_w\times \widetilde C^\times_w$ be given the diagonal $G_w$--action, which is free.  For each edge $e$ incident to $w$, the subgroup $\langle p_e,t_w\rangle$ stabilises $\widetilde Y_e\times\{*\}$ for some point $*\in\widetilde C_w^\times$, so $\widetilde Z_w$ still gives the same shape to $\langle p_e,t_w\rangle$ as $\widetilde C_{e^-}$ gives to $\langle r_{e^-},t_{e^-}\rangle$.  Proposition \ref{prop:joining-walls} (applied with $\mathcal Q=\emptyset$) now provides a free $G$--action on a CAT(0) cube complex.
\renewcommand{\qedsymbol}{$\Box$}
\end{proof}

\subsection{Recovering the tubular case}\label{subsec:tubular-recovery}
Recall that a group $H$ is \emph{tubular} if it decomposes as a finite graph of groups with vertex groups isomorphic to $\integers^2$ and edge groups isomorphic to $\integers$. In \cite{Button:tubular-cubes}, Button characterised the tubular groups that are free-by-$\integers$.  If $G=F\rtimes_\Phi\integers$ is tubular, then its divergence function is at most quadratic \cite{BehrstockDrutu:short}, so by \cite{Macura:detour}, the automorphism $\Phi$ has linear growth.  In \cite{Button:tubular-cubes}, Button also showed that tubular free-by-$\integers$ groups admit \emph{equitable sets} and are hence cubulated by a result in \cite{Wise:tubular}.  We now  recover this result.

\begin{prop}\label{prop:tubular-cubulation}
Let $G=F\rtimes_\Phi\integers$, where $\Phi$ is linearly growing.  If $G$ is also virtually tubular, then $G$ has a finite-index subgroup satisfying the hypotheses of Proposition \ref{prop:general-linear}, and hence $G$ acts freely on a CAT(0) cube complex.
\end{prop}

\begin{proof}
By the discussion in Section \ref{subsec:linear-finite-index}, we  assume that $G=\pi_1M_\phi$, where $M_\phi$ is the mapping torus of a multitwist $\phi:Y\to Y$ with $Y$ $\mathcal P$--clean and starred (notation is from Section \ref{subsec:linear-finite-index}). Since these properties persist under passing to further finite-index subgroups, we can also assume that $G$ is tubular.  Let $\Delta$ be the usual underlying bipartite graph of the graph of spaces decomposition of $M_\phi$, and let $T$ be the associated Bass-Serre tree. Let $S$ be the Bass-Serre tree of a tubular decomposition of $G$.  Fix $w\in\vertices_\circ(\Delta)$ and let $S_w\subset S$ be the minimal subtree for the action of $G_w=F_w\times \langle t_w\rangle$ on $S$.  For $e\in\mathbf E(w)$, let $g_e\in F_w$ represent the conjugacy class in $F_w$ corresponding to the embedded cycle $p_e$.  So, $\langle g_e\rangle$ is a free factor of $F_w$.

Suppose that $S_w\cong\reals$.  Let $G_w'$ be the subgroup of $G_w$ of index at most two that fixes the ends of $S_w$.  Since stabilisers of vertices in $S$ are abelian, there is a short exact sequence $K_w\hookrightarrow G_w'\to \integers$
where $K_w$ is abelian and contains all elliptic elements.  In particular, since $G$ is free-by-cyclic, we have $K_w\leq \integers$ and hence $G_w'$ is virtually abelian.  This contradicts Proposition \ref{prop:split-over-tori}, so $S_w$ is bounded or nonelementary.

Since $t_w$ is central in $G_w$, the subtree $Min(t_w)$ is invariant under $G_w$, so by the preceding argument and minimality of $S_w$, we cannot have $Min(t_w)\cong\reals$, and hence $Min(t_w)=S_w$ is either bounded or nonelementary, and it is fixed pointwise by $t_w$.  Since $G_w$ acts with cyclic edge stabilisers, it follows that $F_w$ acts on $S_w$ with trivial edge stabilisers.  Since the vertex stabilisers in $G_w$ are $\integers^2$, the $F_w$--action on $S_w$ has cyclic vertex stabilisers.

For each $e\in\mathbf E(w)$, let $S_e$ be the minimal subtree of $S_w$ for the $\langle g_e\rangle$--action, so that either $S_e$ is the axis for $g_e$ or $S_e$ is fixed pointwise by $g_e$.  Since $S_e$ is fixed pointwise by $t_w$ and each $t_v$ for which $v$ is incident to the black vertex of $T$ fixed by $\langle g_e\rangle$, and edge stabilisers for the $G$--action on $S$ are cyclic, $S_e$ is a single vertex, and $\stabilizer_{F_w}(S_e)=\langle g_e\rangle$.


Since $\{\langle g_e\rangle:e\in\mathbf E(w)\}$ is a malnormal collection in $F_w$, the $F_w$--tree $S_w$ contains vertices $S_e$ whose (cyclic) stabilisers are the conjugates of the various $\langle g_e\rangle$ in $F_w$.
This gives a free factor decomposition $F_w=F_w'*\left(\asterisk_{e\in \mathbf E(w)}A_e\right)$ with the following property.  For each $e\in\mathbf E(w)$, any element $g\in F_w$ representing the $F_w$--conjugacy class corresponding to the cycle $p_e$ is conjugate in $F_w$ to an element that generates $A_e$. This shows that the elements $p_e\in \homology_1(X_w,\integers)$ are linearly independent, so $\Delta$--edge measurements can be chosen to satisfy the hypotheses of Proposition \ref{prop:general-linear}, which then gives the desired cubulation.
\end{proof}

Here is an example that can be cubulated using Proposition \ref{prop:general-linear}.  

\begin{exmp}[Cubulated, not tubular]\label{exmp:good-measurements-non-tubular}
Let $\Delta$ have a single white vertex $w$, four black vertices $b_1,\ldots,b_4$, and eight edges $e_i,\bar e_i$, for $1\leq i\leq 4$, with $e_i,\bar e_i$ joining $b_i$ to $w$.  Let $X_w$ be the graph with two vertices $x,y$, edges $a_1,a_2,a_3,a_4$ all joining $x$ to $y$, and loops $r_i,\ 1\leq i\leq 4$ attached to $x$. We now define some cycles $p_1,\ldots,p_4$ in $X_w$.  Let $$p_1=a_1a_2^{-1},p_2=(a_2a_3^{-1})^{-1},p_3=(a_3a_4^{-1})^{-1},p_4=a_4a_1^{-1}.$$  The $w$--vertex space is $X_w\times S_w$ and the $b_i$--vertex space is $X_{b_i}\times S_{b_i}$.  The edge space for $e_i$ attaches $X_b\times S_b$ to the torus $p_i\times S_w$, and $\bar e_i$ to $r_i\times S_w$.  The edge maps are such that $r_{e_i}$ is mapped to $p_i$, and $t_{e_i}$ is sent to $p_i^{k_i}t_w$ for some $k_i\in\integers-\{0\}$ (the value of $k_i$ does not matter for the current purpose).  For $\bar e_i$, the maps are $r_{\bar e_i}\mapsto r_i$ and $t_{\bar e_i}\mapsto t_w$.

Let $\mathbf E_1(w)=\{p_1,p_2\}$ and let $\mathbf E_2(w)=\{p_3,p_4\}$.  Define $\pi^1_w:\pi_1(X_w)\to\integers$ by $p_1\mapsto 1,p_2\mapsto 1,p_3\mapsto 0$.  Since $p_1p_2^{-1}p_3^{-1}=p_4^{-1}$, we have $\pi^1_w(p_4)=0$.  Define $\pi^2_w$ by sending $p_1,p_2$ to $0$ and $p_3$ to $1$.  We declare all of these maps to send each $r_i$ to $0$.  Then $\pi^2_w(p_4)=1$.  So $\pi^1_w,\pi_w^2$ (together with the obvious maps sending $r_i$ to $1$ and each $p_i$ to $0$) are good edge-measurements, so Proposition \ref{prop:general-linear} implies $G$ is cubulated. On the other hand, $G$ is not tubular, because there is no free factor decomposition of $\pi_1X_w$ where the $p_i$ lie in distinct free factors (as in the proof of Proposition \ref{prop:tubular-cubulation}), because they are not linearly independent in $_1(X_w)$. 
\end{exmp}

\subsection{Variant of Proposition \ref{prop:general-linear} to enable exits in the unbranched case}\label{subsec:exits}
Now we give conditions on $G$ ensuring that it acts on a cube complex that can be used as the base of induction in the proof of Theorem \ref{thm:main}, for a given collection of future attaching elements.

\begin{prop}\label{prop:exit-arrangement}
Assume the hypotheses of Proposition \ref{prop:general-linear} and assume in addition that $G$ is unbranched.  Let $T$ be the Bass-Serre tree, and let $\mathcal Q\subset G$ be a set of future attaching elements (see Definition \ref{defn:linear-hyperbolics}).  Then $G$ acts freely on a CAT(0) cube complex $\widetilde X$, and there is a $G$--equivariant cubical map $\pi:\widetilde X\to T$ with the following properties:
\begin{enumerate}[(i)]
    \item \label{item:FAM-convex} for all $v\in\vertices(T)$, the preimage $\pi^{-1}(v)$ is a convex subcomplex of $\widetilde X$;
    \item \label{item:FAM-carrier} for all $e\in\edges(T)$, the preimage $\pi^{-1}(\interior e)$ is the open carrier of a hyperplane in $\widetilde X$;
    \item \label{item:FAM-bound} $\sup_H\diam(\pi(H))<\infty$, where the supremum is over all hyperplanes $H$ of $\widetilde X$.
\end{enumerate}
Hence $\widetilde X$ has walls exiting $\mathcal Q$, and for all $qt\in\mathcal Q$, the subcomplex $\widetilde Y_q=\pi^{-1}(\axis(qt))$ is a convex subcomplex of $\widetilde X$ with stabiliser exactly $\langle qt\rangle$.  Finally, there exists $L>0$ such that $\|t_w\|_{\widetilde X}=L$ for all $w\in\vertices_\circ(\Delta)$.
\end{prop}

We summarise the conclusion of Proposition \ref{prop:exit-arrangement} by saying that the $G$--action on $\widetilde X$ is \emph{$\mathcal Q$--exitable}.  Before proving the proposition, we extract a simplifying lemma.

\begin{lem}\label{lem:exit-simplify}
Suppose that $\widetilde X$ and $\pi$ satisfy conclusions \eqref{item:FAM-bound},\eqref{item:FAM-carrier},\eqref{item:FAM-convex} from Proposition \ref{prop:exit-arrangement}.  Then $\widetilde X$ has walls exiting $\mathcal Q$, and for all $qt\in\mathcal Q$, the subcomplex $\widetilde Y_q=\pi^{-1}(\axis(qt))$ is a convex subcomplex of $\widetilde X$ with stabiliser exactly $\langle qt\rangle$.
\end{lem}

\begin{proof}
Since $\pi$ is a cubical map whose vertices have convex preimages it is a \emph{restriction quotient} and hence convex subcomplexes of $T$ have convex preimages \cite{HuangKleiner}.  This implies that $\widetilde Y_q$ is convex.  Second, equivariance implies that $\stabilizer_G(\widetilde Y_q)=\stabilizer_G(\axis(qt))$, and this stabiliser is $\langle qt\rangle$ by Lemma \ref{lem:future-attaching-acyl}.  Finally, item \eqref{item:FAM-bound} and Definition \ref{defn:walls-exit} imply that walls exit $\mathcal Q$.
\end{proof}

\begin{proof}[Proof of Proposition \ref{prop:exit-arrangement}]
By Lemma \ref{lem:exit-simplify}, it suffices to produce a free $G$--action on a CAT(0) cube complex $\widetilde X$ satisfying \eqref{item:FAM-bound},\eqref{item:FAM-carrier},\eqref{item:FAM-convex}.  

Let $w$ be a white vertex.  For each black vertex $e^-$, where $e\in \mathbf E(w)$, let $i=i(e)$ be the unique $i$ with $e\in \mathbf E_i(w)$.  By the unbranched hypothesis and Lemma \ref{lem:unbranched-char}, there is a unique white vertex $\bar w$ and edge $\bar e$ with $\bar e^-=e^-$ and $\bar e^+=\bar w$.  Let $\bar k(e)=k(\bar e)$.  By Lemma \ref{lem:partition-edges}.\eqref{item:edge-partition-coeffs}, there are integers $m_i,\bar m_i$ such that $k(e)=m_i$ and $\bar k(e)=\bar m_i$ for all $e\in \mathbf E_i(w)$. Let $c_i=m_i-\bar m_i$.  Proposition \ref{prop:split-over-tori}.\eqref{item:linear-finite-index} ensures $c_i\neq 0$.  The good edge measurement assumption provides $K_i\neq 0$ such that $\pi_i^w(p_e)=K_i$ for $e\in\mathbf E_i(w)$ and $\pi_i^w(p_f)=0$ for $f\in\mathbf E(w)-\mathbf E_i(w)$.

As in the proof of Proposition \ref{prop:general-linear}, we first arrange temporarily that the $K_i$ are all equal (over all $i$ and $w$), so by further rescaling, we can assume that there is a fixed positive $K$ such that $K_i=Kc_i$ for all $w\in\vertices_\circ(\Delta)$ and $i\in J_w$.  Let $$O=\lcm\{|c_i|^2:w\in\vertices_\circ(\Delta),i\in J_w\}$$
and for each $i$, let $A_i=Oc_i/|c_i|^2$, so that $A_ic_i=O$.  In order for $A_i$ to be an integer satisfying the preceding, we just needed $O$ to be divisible by each $c_i$.  Our specific choice of $O$ additionally ensures that $c_i$ divides $A_i$ for each $i$.

Now fix a white vertex $w$.  Define a homomorphism $\gamma_w:G_w\to\integers$ by
$$\gamma_w(ft_w^x)=\sum_{j\in J_w}\frac{A_j\pi_j^w(f)}{c_j}-OKx$$
for $f\in F_w$ and $x\in\integers$.

\begin{claim}\label{claim:exit-kernel}
For each $i\in J_w$ and $e\in \mathbf E_i(w)$, we have $\kernel\gamma_w\cap \langle p_e,t_w\rangle=\langle p_e^{c_i}t_w\rangle$.
\end{claim}
\renewcommand{\qedsymbol}{$\blacksquare$}

\begin{proof}[Proof of Claim \ref{claim:exit-kernel}]
Observe that $\gamma_w(p_e)=A_iK$ and $\gamma_w(t_w)=-OK$.  Hence $\gamma_w(p_e^{c_i}t_w)=A_iKc_i-OK=0$ and the kernel intersects $\langle p_e,t_w\rangle$ in a subgroup of rank one, so since $p_e^{c_i}t_w$ generates a maximal cyclic subgroup of $\langle p_e,t_w\rangle$, the claim follows. 
\end{proof}

Let $\widetilde X_w^\dagger$ be the CAT(0) cube complex isomorphic to the standard tiling of $\reals$ by unit cubes, with $0$--skeleton $\integers$, equipped with the $G_w$--action given by $\gamma_w$.

\begin{claim}\label{claim:exit-action}
Let $e\in\edges(\Delta)$.  Then there exists a CAT(0) cube complex $\widetilde X_{\bar e}^\circ$ such that $G_{\bar e^+}=F_{\bar e^+}\times \langle t_{\bar e^+}\rangle$ acts on $\widetilde X_{\bar e}^\circ$ such that all hyperplane stabilisers are $\langle t_{\bar e^+}\rangle$, and, for all edges $f$ incident to $\bar e^+$,  
$p_{f}$ acts hyperbolically if $i(f)=i(\bar e)$, but $i(f)\neq i(\bar e)$, then $p_f$ acts elliptically.  Moreover, if $f,f'$ are distinct and $i(f)=i(f')=i(\bar e)$, then $p_f$ and $p_{f'}$ are not cut by any common walls.  
\end{claim}

\begin{proof}[Proof of Claim \ref{claim:exit-action}]
Let $i\in J_{e^+}$ be such that $\bar e\in \mathbf E_i(\bar e^+)$ and consider the homomorphism $\pi^i_{\bar e^+}$, which is induced by a cellular map $X_{\bar e^+}\to \mathbb S^1$ which is nullhomotopic on each  $p_f,\ f\not\in \mathbf E_i(\bar e^+)$, but not on $p_{\bar e}$ (or any other cycle in $\mathbf E_i(\bar e^+)$).  Moreover, we can assume that the preimage of the vertex in $\mathbb S^1$ is a subgraph of $X_{\bar e^+}$ and the preimage of the open edge is a union of open edges.  This splits $F_{\bar e^+}$ as a finite graph of groups with trivial edge groups, obtained by deleting the open edges of $X_{\bar e^+}$ that map homeomorphically to the edge in $\mathbb S^1$.  The vertex groups correspond to components of the preimage of the vertex.  Let $\widetilde X_{\bar e}^\circ$ be the Bass-Serre tree.  If $f,f'$ are distinct and act hyperbolically, then they belong to $\mathbf E_{i}(\bar e^+)$, so by Remark \ref{rem:disjoint-cycles-in-partition}, $p_f,p_{f'}$ do not share any edges, and hence are not cut by any common walls.
\end{proof}

For each edge $e\in\edges(\Delta)$ such that $\bar e^+=w$, let $i\in J_w$ be such that $i=i(\bar e)$.  Let $\widetilde X^\circ_i$ be a copy of $\widetilde X^\circ_{e}$, equipped with the $G_w$--action from Claim \ref{claim:exit-action}. Let $$\widetilde X_w=\widetilde X_w^\dagger\times \prod_{i\in J_w}\widetilde X^\circ_{i},$$ on which we let $G_w$ act diagonally.  For each edge $e$ incident to $w$, the subgroup $\langle p_e,t_w\rangle$ stabilises an isometric subcomplex $\widetilde Y_e$ isomorphic  to the cubulation of $\langle p_e,t_w\rangle$ coming from the two codimension--$1$ subgroups $\langle t_w\rangle$ and $\langle p_e^{c_i}t_w\rangle$.  By applying Lemma \ref{lem:small-cancellation-improper} and replacing $\widetilde X_w$ by its product with the cube complex output by that lemma, we can assume that the action on $\widetilde X_w$ is free, without changing the induced cubulations of the various $\langle p_e,t_w\rangle$.  For each black vertex $b$, with incident edges $e,\bar e$, let $\widetilde X_b$ be the cubulation of $G_b$ using the codimension--$1$ subgroups $\langle r_b^{-k(e)}t_b\rangle$ and $\langle r_b^{-k(\bar e)}t_b\rangle$ (each with multiplicity to be chosen shortly).  Let $w=e^+,\bar w=\bar e^+$.

\textbf{Shape comparison.}  Let $\omega_b$ be a wall in $\widetilde X_b$ stabilised by $\langle t_br_b^{-k(e)}\rangle$ and let $\bar\omega_b$ be a wall stabilised by $\langle t_br_b^{-k(\bar e)}\rangle$.  Then $G_b\cdot \omega_b=\langle r_b\rangle \cdot \omega_b$ and $G_b\cdot\bar\omega_b=\langle r_b\rangle\cdot \bar\omega_b$.  Meanwhile, the walls in $G_b\cdot \bar\omega_b$ belong to $|k(e)-k(\bar e)|=|c_i|$ orbits under the action of $\langle t_br_b^{-k(e)}\rangle$.

The analysis in $\widetilde Y_e$ is as follows.  Let $\upsilon_w$ be a wall stabilised by $\langle t_w\rangle$ and let $\bar\upsilon_w$ be a wall stabilised by $\langle t_wp_e^{c_i}\rangle$.  The orbit $\langle p_e,t_w\rangle\cdot \upsilon_w$ contains $|\pi_i^w(p_e)|=|Kc_i|$ orbits under the $\langle p_e\rangle$--action, and the orbit $\langle p_e,t_w\rangle\cdot \bar\upsilon_w$ contains $|\gamma_w(p_e)|=|A_iK|$ orbits under the $\langle p_e\rangle$--action.  The hyperplanes in the $\langle p_e,t_w\rangle$--orbit of $\bar\upsilon_w$ belong to $|\gamma_w(t_w)|=|OK|$ orbits under the $\langle t_w\rangle$--action.

Now subdivide $\widetilde X_b$ equivariantly so that $\bar\omega_w$ is replaced by $|A_iK|$ copies, and do the same for $\omega$.  Note that $|A_i|=|A_{i(\bar e)}|$ since $|c_i|=|c_{i(\bar e)}|$, so this subdivision agrees with the one performed when considering $\bar e$ instead of $e$.

The subdivisions we choose in $\widetilde X_w$ are not symmetric for the different types of walls.  First, we do not subdivide $\widetilde X^\dagger_w$.  Second, we subdivide the factor $\widetilde X^\circ_i$ so that each wall is replaced by $|A_i/c_i|$ parallel copies; $c_i$ divides $A_i$, by our choice of $O$.  In $\widetilde Y_e$, the cumulative effect of these subdividions is to replace $\upsilon_w$ with $|A_i|/|c_i|$ parallel copies, leaving $\bar\upsilon_w$ unchanged.

Since the edge map sends $r_b^{-k(e)}t_b$ to $t_{w}$ and $r_b^{-k(\bar e)}t_b$ to $p_{e}^{k(e)-k(\bar e)}t_{w}=p_e^{c_i}t_w$, after the subdivisions, the assignments $\upsilon_w\mapsto \omega_b$ and $\bar\upsilon_w\mapsto \bar\omega_b$ induce an equivariant bijection from the walls in $\widetilde Y_e$ to those in $\widetilde X_b$.  Hence the above cubulations of $G_b$ and $G_w$ give the same shape to $G_e$, and each $t_w$ has translation length $|OK|$.

\textbf{Conclusion.}  Proposition \ref{prop:joining-walls} now provides a free action of $G$ on a CAT(0) cube complex $\widetilde X$, and an equivariant map $\pi:\widetilde X\to T$.  There are finitely many orbits of hyperplanes by Lemma \ref{lem:finitely-many-orbits}, and \eqref{item:FAM-carrier} and \eqref{item:FAM-convex} follow from the proposition.  Since wandering hyperplanes in $\widetilde X$ intersect the vertex spaces in unique hyperplanes, the translation lengths of all $t_w$ on $\widetilde X$ are still $|OK|$.

It remains to bound the diameter of $\pi(H)$, where $H$ is a hyperplane of $\widetilde X$. Recall from the proof of Proposition \ref{prop:joining-walls} how $\widetilde X$ is constructed.  The \emph{vertical} hyperplanes are sent by $\pi$ to single points, by definition, so we only need to consider the case where $H$ is \emph{wandering}.  In our situation, if $H$ is wandering, then, letting $Z$ be the graph of spaces from the proof of Proposition \ref{prop:joining-walls}, $H$ is an elevation of an immersed wall $\bar H\to Z$ with the following properties.  There is a white vertex $w$ and $H$ intersects the vertex space $G_w\backslash\widetilde X_w$ in an immersed hyperplane corresponding to a normal subgroup intersecting each $\langle p_e,t_w\rangle,\ e\in\mathbf E(w)$, in the subgroup $\langle p^{k(e)-k(\bar e)}t_w\rangle$, and hence intersects $G_{\bar e^+}\backslash \widetilde X_{\bar e^+}$ in a hyperplane corresponding to a splitting of $G_{\bar e^+}$ over $\langle t_{\bar e^+}\rangle$ in which all $p_f,\ f\in\mathbf E(\bar e^+)-\{\bar e\}$ are not cut by $H$.  Hence $\pi(H)$ is contained in a ball of radius $2$ in $T$.
\renewcommand{\qedsymbol}{$\Box$}
\end{proof}

\subsection{Tubular termini and fast mapping tori}\label{subsec:tubular-termini}
Now we find applications for Proposition \ref{prop:exit-arrangement}.

\begin{defn}[Tubular termini]\label{defn:virtually-tubular-termini}
Let $\Phi\in\Out(F)$ have polynomial growth rate $d\geq 1$.  Let $G$ be the mapping torus of $\Phi$ and let $k>0$ be such that $\Phi^k$ is UPG.  Let $G_k$ be the mapping torus of $\Phi^k$, regarded as a finite-index subgroup of $G$.  Let $\{L_i\}_{i\in I}$ be the bottom-level vertex groups in the superlinear hierarchy of $G_k$.  Recall that each $L_i$ is the mapping torus of a linear-growth automorphism of a subgroup of $F$.  We say $G$ has \emph{tubular termini} if each $L_i$ is  a tubular group.
\end{defn}

\begin{lem}\label{lem:fast-tubular}
Let $\Phi\in\Out(F)$ have polynomial growth.  If $\Phi$ is fast, then $G=F\rtimes_\Phi\integers$ is unbranched and has tubular termini.
\end{lem}

\begin{proof}
Since $\Phi$ is fast, $G$ is unbranched by Lemma \ref{lem:fast-implies-unbranched}, and hence the finite-index subgroup $G_k$ from Definition \ref{defn:virtually-tubular-termini} is also unbranched. Since $G_k$ contains $F$, and $\Phi$ and $\Phi^k$ have the same growth rate, $\Phi^k$ is also fast.  So we can assume that $G=G_k$ and $\Phi^k=\Phi$, i.e. $\Phi$ is already UPG.  

We argue by induction on the growth rate $d=\rank(F)-1$ of $\Phi$ that $G$ has tubular termini.

\textbf{Linear case.}  Suppose $d=1$.  Then $\rank(F)=2$.  Let $f:\Gamma\to\Gamma$ be a relative train track for $\Phi$.  By \cite[Thm. 1.1]{BFH:kolchin},  $\Gamma$ has two edges $e_0,e_1$, which must therefore be loops at a common vertex.  By Proposition \ref{prop:basic-rttm}, we have $f(e_0)=e_0$ and $f(e_1)=e_1e_0^k$ for some $k$.  Thus $G$ is an HNN extension of $\langle e_0,t\rangle\cong\integers^2$ with stable letter $e_1^{-1}$ conjugating $t$ to $e_0^kt$.  Thus $G$ is tubular.

\textbf{Superlinear case.}  Suppose that $d\geq 2$.  Let $\Delta$ be the underlying graph of the top strata decomposition.  Exactly as in the proof of Lemma \ref{lem:fast-implies-unbranched}, each vertex group $G_v$ is the mapping torus of a fast automorphism $\phi_v$ of $F_v$ with growth rate at most $d-1$.  Hence $G_v$ has tubular termini, by induction, and Definition \ref{defn:virtually-tubular-termini} implies that $G$ has tubular termini.
\end{proof}

\begin{defn}[Immediate future attaching elements]\label{defn:terminal-Q_i}
Let $\Phi\in\Out(F)$ be a UPG element with polynomial growth rate $d\geq 1$, and let $G$ be its mapping torus.  Let $\{L_i\}_{i\in I}$ be the bottom-level vertex groups in the superlinear hierarchy, as in Definition \ref{defn:virtually-tubular-termini}.  For each $L_i$, let $T_i$ be the Bass-Serre tree of the splitting of $L_i$ from Proposition \ref{prop:split-over-tori}.  Let $\mathcal Q_i\subseteq L_i$ be the following set of future attaching elements.  Consider the graph of groups in the superlinear hierarchy that has $L_i$ as a vertex group.  For each edge $e$ incident to $L_i$, the edge group maps to a cyclic subgroup of $L_i$ generated by some element $\ell$.  The set $\mathcal Q_i$ consists of one such $\ell$ for each $e$, excluding those $e$ for which $\ell$ is elliptic on $T_i$.  We call $\mathcal Q_i$ the set of \emph{immediate future attaching elements} in $L_i$.
\end{defn}

\begin{cor}[Cubulating bottom-level vertex groups with exits]\label{cor:cubulated-with-exits}
Let $\Phi\in\Out(F)$ be a UPG element with polynomial growth rate $d\geq 1$, and let $G$ be its mapping torus.  Suppose that $\Phi$ is fast or, more generally, is unbranched and has tubular termini.  Let $\{L_i\}_{i\in I}$ be the bottom-level vertex groups in the superlinear hierarchy, as in Definition \ref{defn:virtually-tubular-termini}.  For each $i\in I$, let $\mathcal Q_i\subseteq L_i$ be the immediate future attaching elements.

Then for each $i\in I$, there is a free action of $L_i$ on a CAT(0) cube complex $\widetilde X_i$ that satisfies the conclusions from Proposition \ref{prop:exit-arrangement} with respect to the set $\mathcal Q_i$ of future attaching elements.
\end{cor}

\begin{proof}
By Lemma \ref{lem:fast-tubular}, each $L_i$ is tubular and $G$ is unbranched if $\Phi$ is fast; otherwise, these two statements hold by hypothesis.  Hence each $L_i$ is also unbranched by Lemma \ref{lem:unbranched-char}.  Proposition \ref{prop:tubular-cubulation} implies that the hypotheses of Proposition \ref{prop:general-linear} hold.  Together with the fact that $L_i$ is unbranched, this allows us to apply Proposition \ref{prop:exit-arrangement} to get the  desired cubulations $L_i\actson \widetilde X_i$.
\end{proof}

\begin{rem}[Cocompactness in the fast case]\label{rem:cocompactness-in-the-fast-case}
If $\Phi$ is fast, then the proof of Lemma \ref{lem:fast-tubular} implies that there is a single terminal vertex group $L$ and $L\cong \langle a,b,t\mid [a,t],b^{-1}tb=a^kt\rangle$ for some $k\neq 0$.  In this case, by either applying \cite[Cor. 5.10]{Wise:tubular} or tracing through the proof of Proposition \ref{prop:exit-arrangement} and noting that the use of Lemma \ref{lem:small-cancellation-improper} can be avoided in this case, the cubulation $L\actson \widetilde X$ from Corollary \ref{cor:cubulated-with-exits} can be chosen with the following additional properties:
\begin{itemize}
    \item $L$ acts on $\widetilde X$ cocompactly.
    \item Each axis of $t$ is cubically convex.
    \item Each future attaching element $qt$ is cubically convex-cocompact.
\end{itemize}
This observation will be used in Remark \ref{rem:cocompactness}.
\end{rem}

\section{Cubulating groups with acylindrical cyclic hierarchies}\label{subsec:acyl-cyclic}
Let $G$ be a torsion-free finitely generated group acting cocompactly and without edge inversions on a tree $T$.  We assume that $(G,T,\{\widetilde X_v\}_v,\{\widetilde X_e\},\{\widetilde X_e^\pm\}_e,\{\psi_e\}_e)$ is a graph of cubulated groups with compatibly cubulated edge groups (see Definition \ref{defn:graph-walls}).  Let $\mathcal Q\subset G$ be a finite set of elements satisfying Definition \ref{defn:graph-walls}.\eqref{item:stass-Q}.  Our first goal is the following theorem.

\begin{thm}[Turns from exits]\label{thm:cubulation-with-turns}
Let $\widetilde X$ be the $G$--CAT(0) cube complex, and $\pi:\widetilde X\to T$ the map, provided by Proposition \ref{prop:joining-walls}.  Suppose that  $\widetilde X$ has walls exiting $\mathcal Q$ (see Definition \ref{defn:walls-exit}).

For each $v\in\vertices(\bar T)$, let $\mathcal T_v\subset G_v$ be a finite set such that 
\begin{itemize}
 \item each $t\in\mathcal T_v$ generates a maximal cyclic subgroup of $G$ and
 \item $\mathcal T_v$ is wall-independent for the $G_v$--action on $\widetilde X_v$.
\end{itemize}

 Let $L\in\integers_{>0}$ and suppose that $\|t\|_{\widetilde X_v}=L$ for all $t\in\mathcal T_v$ and all $v\in\vertices(\bar T)$.  

Then there exists a CAT(0) cube complex $C_\sharp$ such that all of the following hold:
\begin{enumerate}[(a)]
 \item \label{item:turns-action} $G$ acts freely on $C_\sharp$ with finitely many orbits of hyperplanes.
 \item \label{item:turns-wall-independent} For each $v\in\vertices(\bar T)$, the set $\mathcal T_v\cup\mathcal Q$ is wall-independent for the $G$--action on $C_\sharp$.
 \item \label{item:turns-equal} For each $q\in\mathcal Q$ and $v\in\vertices(\bar T)$ and $t\in\mathcal T_v$, we have $\|q\|_{C_\sharp}=\|t\|_{C_\sharp}=2 L$.
\end{enumerate}
\end{thm}

Before proving the theorem, we need some preparatory discussion about the following cubical presentation (recall Definition \ref{defn:cubical_presentation}) for $G$.  Let $\widetilde X$ be as in the statement of Theorem \ref{thm:cubulation-with-turns} and let $X=G\backslash \widetilde X$.  Since the $G$--action on $\widetilde X$ is free and without inversions, $X$ is a nonpositively curved cube complex and the quotient map $\widetilde X\to X$ is the universal covering map and $\pi_1(X)=G$.  For each $q\in\mathcal Q$, composing the covering map with the inclusion $\widetilde Y_q\to\widetilde X$ gives a local isometry $\widetilde Y_q\to X$.  Each of these local isometries induces a map on fundamental group with trivial image, since $\widetilde Y_q$ is simply connected (by Proposition \ref{prop:joining-walls}, it is CAT(0)).  Hence we have a cubical presentation for $G$ in which the $\widetilde Y_q$ are relators, namely $G\cong \pi_1X^*$, where 
$$X^*=\langle X\mid \{\widetilde Y_q\to X:q\in\mathcal Q\}\rangle.$$

Recall that $\widetilde{X^*}$ is the cover of $X$ corresponding to the normal closure in $\pi_1X$ of the images of the fundamental groups of the relators, so in this case, where the relators are all simply connected, $\widetilde{X^*}=\widetilde X$. We will use Theorem \ref{thm:cubulating X star}, which in our situation involves declaring new walls in $\widetilde X$ by making certain hyperplanes equivalent.  (We are abusing notation slightly: $\widetilde{X^*}$ is really the universal cover of the cubical presentation complex, i.e. it contains a cubical part and some cones over relators, but we are always only concerned with the cubical part and therefore use $\widetilde{X^*}$ to refer to the cubical part with the cones removed.)

\begin{lem}\label{lem:metric-small-cancellation}
The cubical presentation $X^*$ satisfies the $C'(\frac{1}{14})$ condition from Definition \ref{defn:c12} and hence has short innerpaths in the sense of \cite[Defn. 3.69]{Wise:QCH}
.
\end{lem}

\begin{proof}
The metric condition holds vacuously since $\widetilde Y_q$ is simply connected, and this implies short innerpaths by \cite[Lem. 3.70]{Wise:QCH}.
\end{proof}

The next task is to define equivalence relations on the hyperplanes of each $\widetilde Y_q$ so that $X^*$ satisfies the simplified $B(6)$ condition from Definition \ref{defn:B6}.  Let $\chi_0$ be the constant from Definition \ref{defn:walls-exit} and let $\chi_1$ be the constant from Lemma \ref{lem:axis-subcomplex-overlap}.  

\begin{lem}\label{lem:piece-tree-bound}
There exists $\chi_2=\chi_2(\chi_0,\chi_1)$ such that the following holds.  Let $q\in\mathcal Q$ and let $P\to \widetilde Y_q$ be a geodesic piece-path.  Then $|\pi\circ P|\leq \chi_2$.
\end{lem}

\begin{proof}
Let $A$ be a convex subcomplex of $\widetilde X$ such that $\gate_{1,q}(A)$ is a piece in $\widetilde Y_q$ with $P$ a geodesic in $\gate_{1,q}(A)$.  So, either $A=g\widetilde Y_{q'}$ for some $g\in G$ and $q'\in\mathcal Q$ with either $q\neq q'$ or $g\not\in\langle q\rangle$ (cone-piece case), or $A=N(H)$ for a hyperplane $H$ of $\widetilde X$ not crossing $\widetilde Y_q$ (wall-piece case).

In both cases, $\pi\circ P$ is a geodesic of $T$ since $P$ is a geodesic and $\pi$ a restriction quotient, so $|\pi\circ P|\leq \diam_T(\pi(A)\cap \axis(q))$.  For cone-pieces, $\diam_T(\pi(A)\cap \axis(q))\leq \chi_1$ by Lemma \ref{lem:axis-subcomplex-overlap}.

In the wall-piece case, $H$ cannot intersect $\widetilde Y_q$.  Suppose $x,y\in\widetilde X$ have the property that $\pi(x),\pi(y)\in\pi(H)\cap \pi(\widetilde Y_q)$.  Then $x,y\in \widetilde Y_q$, since $\pi^{-1}(\pi(\widetilde Y_q))=\widetilde Y_q$.  If $H$ is a vertical hyperplane, then $\pi^{-1}(\pi(H))=H$, and so $x\in\widetilde Y_q\cap H$, contradicting that $H$ does not cross $\widetilde Y_q$.  Hence $H$ is wandering.  If $\pi(x)\neq \pi(y)$, then they are separated by some edge $e$ of the tree $\pi(\widetilde Y_q)$, and $e\subset \pi(H)$ since $\pi(H)$ is connected.  Now, if $z\in H$ satisfies $\pi(z)=m_e$, then $H$ crosses $\widetilde X_e$, which is contained in $\widetilde Y_q$,  contradicting that $H\cap \widetilde Y_q=\emptyset$.  Thus $\pi(x)=\pi(y)$, i.e. $\pi(P)$ is a point.
\end{proof}

\begin{cons}[New walls in the $\widetilde Y_q$]\label{cons:B6-walls}
Let $\mathcal H(\widetilde X)$ be the set of hyperplanes of $\widetilde X$ (recall that $\mathcal H(\widetilde X)=\mathcal W(\widetilde X)\sqcup\mathcal V(\widetilde X)$ is the $G$--invariant partition into wandering and vertical hyperplanes).

From Proposition \ref{prop:joining-walls}, $\|t\|_{\widetilde X}=\|t\|_{\widetilde X_v}$ for $t\in\mathcal T_v$, and the hyperplanes cutting $t$ in $\widetilde X$ are all wandering.  By replacing $\widetilde X$ with the CAT(0) cube complex obtained by cubically subdividing to replace each wandering hyperplane with two parallel copies (using that the wandering hyperplanes are $G$--invariant), we have $\|t\|_{\widetilde X}=2L$, and we now have an involution ${}^\dagger:\mathcal W(\widetilde X)\to\mathcal W(\widetilde X)$ taking each $H\in\mathcal W(\widetilde X)$ to the parallel copy $H^\dagger$ of $H$ with the property that $H,H^\dagger$ were created from the same hyperplane of the original complex by the subdivision.

For each $q\in\mathcal Q$, let $M_q>0$ be the translation length of $q$ on the tree $T$, which is the same as the number of $\langle q\rangle$--orbits of vertical hyperplanes in $\widetilde X$ that cut $q$, and hence the same as the number of $\langle q\rangle$--orbits of vertical hyperplanes in $\widetilde Y_q$ that cut $q$.

Now subdivide $T$ so each edge has length $2L$ and pull back the subdivision to $\widetilde X$ using $\pi$.  Then $q$ is now cut by $2LM_q$ $\langle q\rangle$--orbits of vertical hyperplanes.  (For both the vertical and  wandering subdivisions of $\widetilde X$, we also do the induced subdivision of $\widetilde Y_q$, to keep it a subcomplex of $\widetilde X$.)

For each of the original edges $e$ of $T$, let $m_1<m_2<\cdots < m_{2L}$ be the midpoints of the edges in the subdivision of $e$, and let $H_i=\pi^{-1}(m_i)$.  For $1\leq i\leq L$, let $H_i^\dagger=H_{2L-i+1}$ and for $L+1\leq i\leq 2L$, let $H_i^\dagger=H_{i-2L+1}$.  This extends the involution ${}^\dagger$ over the vertical hyperplanes.  Hence ${}^\dagger:\mathcal H(\widetilde X)\to\mathcal H(\widetilde X)$ is a $G$--equivariant involution preserving $\mathcal W(\widetilde X)$ and $\mathcal V(\widetilde X)$.  For each $H$, the hyperplanes $H,H^\dagger$ are parallel (they cross the same hyperplanes) and  at Hausdorff distance at most $2L$ in $\widetilde X$.  Moreover, $H,H^\dagger$ are in different $G$--orbits and have the same $G$--stabiliser.   For all $g\in G$, $H$ cuts $g$ if and only if $H^\dagger$ cuts $g$. If $H\cap \widetilde Y_q\neq\emptyset$, then $H^\dagger\cap\widetilde Y_q\neq\emptyset$.

Let $K\geq 1$ be an integer to be chosen later (any sufficiently large integer will do; we will point out the exact constraints as they arise). For each $q\in\mathcal Q$, let $\mathcal W_q=\{H\cap \widetilde Y_q:H\in\mathcal W(\widetilde X),\ H\cap\widetilde Y_q\neq \emptyset\}$ and $\mathcal V_q=\{H\cap \widetilde Y_q:H\in\mathcal V(\widetilde X),\ H\cap\widetilde Y_q\neq\emptyset\}$.  The properties of ${}^\dagger$ mentioned above imply that there is a decomposition $\mathcal W_q=\mathcal W_q^-\sqcup \mathcal W_q^+$ making ${}^\dagger:\mathcal W_q^\pm\to\mathcal W_q^\mp$ a bijection, and there is a decomposition $\mathcal V_q=\mathcal V_q^-\sqcup\mathcal V_q^+$ with the analogous property. Define a relation $\sim_q$ on $\mathcal W_q\cup\mathcal V_q$ as follows:
\begin{itemize}
 \item If $U\in\mathcal W_q$ does not cut $q$, then $U\sim_q U$, and $U$ is unrelated to other hyperplanes of $\widetilde Y_q$.
 
 \item If $U\in\mathcal W_q$ cuts $q$, and $U\in\mathcal W_q^\pm$, then let $U\sim_q q^{\mp K} U^\dagger$.
 
 \item If $U\in\mathcal V_q$, then $U$ necessarily cuts $q$, by the definition of $\widetilde Y_q$.  Let $\sigma$ be a fixed length--$2LM_q$ combinatorial geodesic in $\axis(q)\subset T$ (we make one such choice for each $q$ and extend equivariantly).  If $\pi(U)$ is the midpoint of one of the initial $2L$ edges of $\sigma$, then $U\sim_q U$ and $U$ is related to no other hyperplanes.  Otherwise, $U$ belongs to the ${}^\dagger$--invariant set of $2L(M_q-1)$ hyperplanes projecting to midpoints of the terminal $2L(M_q-1)$ edges of $\sigma$, and if $U\in\mathcal V^\pm_q$, we let $U\sim_q q^{\mp K}U^\dagger$.
\end{itemize}

Note that $\sim_q$ is an equivalence relation, and each hyperplane $U$ of $\widetilde Y_q$ is either unique in its $\sim_q$--class, or belongs to a class containing exactly two hyperplanes.
\end{cons}

The following lemma reveals how large $K$ must be.  To save notation, we use the following convention: after the subdivisions in Construction \ref{cons:B6-walls}, the bounds in Lemmas \ref{lem:piece-tree-bound} and \ref{lem:axis-subcomplex-overlap}, and the bound in Definition \ref{defn:walls-exit}, still hold, and we continue to use the notation $\chi_0,\chi_1,\chi_2$ for the bounds in the subdivided complex (i.e. we multiplied them by $2L$ and renamed them).

\begin{lem}[Equivalent hyperplanes are many pieces apart]\label{lem:K-large}
There exists $K_0$, depending on $L$ and the constants $\chi_0,\chi_1$, such that the following holds for all $q\in\mathcal Q$ provided $K\geq K_0$, where $K$ is as in Construction \ref{cons:B6-walls}.  Let $U,V\in\mathcal W_q\cup\mathcal V_q$ be distinct and suppose that $U\sim_qV$.  Then $\dist_T(\pi(U),\pi(V))>100(\chi_0+\chi_2+2L+1)$.  In particular, if $P\to\widetilde Y_q$ is a path starting on $N(U)$ and ending on $N(V)$, then $P$ is not the concatenation of fewer than $8$ piece-paths.
\end{lem}

\begin{proof}
The element $q$ has translation length at least $1$ along the combinatorial geodesic $\axis(q)$ of $T$.  Since $U$ is equivalent to a hyperplane distinct from $U$, the definition of $\sim_q$ implies that $U$ cuts $q$.  By Definition \ref{defn:walls-exit}, and our assumption on exiting, $\pi(U)$ has diameter at most $\chi_0$.

Without loss of generality, $V=q^KU^\dagger$, so $\dist_T(U,V)>K-2\chi_0-2L$, using that $q^KU$ and $q^KU^\dagger=(q^KU)^\dagger$ are at Hausdorff distance at most $2L$ and $\pi$ is $1$-lipschitz.  The existence of $K_0$ with the first property in the statement follows immediately.  

Let $P\to\widetilde Y_q$ be a path that is the concatenation of at most $7$ piece paths, starting on $N(U)$ and ending on $N(V)$.  Let $P=P_1\cdots P_k$, where $k\leq 7$ and each $P_i$ is a piece-path in $\widetilde Y_q$.  Let $\bar P_i$ be a geodesic joining the endpoints of $P_i$; since images of convex sets under $\gate_{1,q}$ are again convex, $\bar P_i$ is also a piece-path, so $|\pi\circ P_i|\leq \chi_2$ by Lemma \ref{lem:piece-tree-bound}.  Hence the distance in $T$ between the endpoints of $\pi\circ P$ is at most $7\chi_2$.  This contradicts that $\dist_T(U,V)>100(\chi_0+\chi_2+2L+1)$.
\end{proof}

\begin{lem}\label{lem:b6-check}
Let $K\geq K_0$, where $K_0$ is as in Lemma \ref{lem:K-large}.  Equipping the hyperplanes of $\widetilde Y_q,\ q\in\mathcal Q$ with the equivalence relation $\sim_q$ from Construction \ref{cons:B6-walls} makes $X^*$ a simplified $B(6)$ presentation in the sense of Definition \ref{defn:B6}.
\end{lem}

\begin{proof}
The $C'(1/14)$ condition holds by Lemma \ref{lem:metric-small-cancellation}.  Definition \ref{defn:B6}.\eqref{B6:equivalence} holds for each $\sim_q,\ q\in\mathcal Q$ by Construction \ref{cons:B6-walls}.  Given $q$, let $U,V$ be distinct hyperplanes of $\widetilde Y_q$ with $U\sim_qV$.  Then $U,V$ have disjoint carriers, by Lemma \ref{lem:K-large}, since $\pi$ is $1$--lipschitz.  In particular, $U,V$ cannot cross or osculate.  This verifies Definition \ref{defn:B6}.\eqref{B6:wallspace}.

Given a hyperplane $U=H\cap\widetilde Y_q$ of $\widetilde Y_q$, let $[U]$ be the $\sim_q$--class of $U$ and let $W=\cup_{V\in[U]}V$.  If $W=U$ then $\widetilde Y_q$ is partitioned into two halfspaces $\OL U,\OR U$ by $W$.  Otherwise, $W=U \sqcup q^KU^\dagger$ and, if $K\geq K_0$, we have (up to relabelling) that $\OL U \cap q^K\OR{U^\dagger}=\emptyset$.  Let $\OL W=\OL U\sqcup q^K\OR{U^\dagger}$ and let $\OR W=\widetilde Y_q-\OL W$.  Then the closures of $\OL W$ and $\OR W$ intersect in $W$, verifying Definition \ref{defn:B6}.\eqref{B6:wallspace_2}.

Next we check Definition \ref{defn:B6}.\eqref{B6:homotopy}.  Suppose that $P$ starts and ends on the carrier $N(W)$ of the wall $W$ (defined as above from a $\sim_q$--class) and is the concatenation of at most $7$ piece-paths.  By Lemma \ref{lem:K-large}, $P$ must have both endpoints on the same constituent hyperplane carrier $N(U)$ of $N(W)$, so since $\widetilde Y_q$ is CAT(0), $P$ is path-homotopic into $N(U)\subset N(W)$, giving item \eqref{B6:homotopy}.

Finally, consider the action of $\Aut(\widetilde Y_q\to G\backslash\widetilde X)=\stabilizer_G(\widetilde Y_q)=\langle q\rangle$ on $\widetilde Y_q$.  If $U\sim q^KU^\dagger$, then from Construction \ref{cons:B6-walls}, $qU\sim_q q^K(qU)^\dagger=q(q^KU^\dagger)$, verifying item \eqref{B6:aut}.  We have now verified all parts of the definition of the simplified $B(6)$ condition.
\end{proof}

\begin{lem}[Hypothesis \eqref{LS:strong separation} from Theorem \ref{thm:cubulating X star}]\label{lem:exits-strong-separation}
The following holds for each $q\in\mathcal Q$.  Let $\kappa\to\widetilde Y_q$ be a combinatorial geodesic joining $x,y\in\widetilde Y_q$ (in that order) and let $U_1,U_1'$ be hyperplanes in $\widetilde Y_q$ that are respectively intersections of $\widetilde Y_q$ with hyperplanes $H_1,H_1'$ of $\widetilde X$.  Suppose that $U_1\sim_q U_1'$.  Suppose that $\kappa= \kappa_0e_1\kappa_1$, where $e_1$ is dual to $U_1$.  Suppose that either $U_1'$ is proximate to $y$, or $\kappa_1$ crosses $U'_1$.  Then there is a hyperplane $V$ of $\widetilde Y_q$ such that $V$ separates $x,y$ but the $\sim_q$--class of $V$ does not have a representative that is proximate to $x$ or $y$. 
\end{lem}

\begin{proof}
Let $U$ be a hyperplane in $\widetilde Y_q$ and let $v\in\widetilde Y_q^{(0)}$.  Recall that $U$ is proximate to $v$ if there is a path $AB$ in $\widetilde Y_q$ whose initial $1$--cube is dual to $U$,  whose terminal $0$--cube is $v$, and where each of $AB$ is either a $1$--cube or a piece-path.  If $U$ is proximate to $v$, Lemma \ref{lem:piece-tree-bound} therefore implies that $\dist_T(\pi(U),\pi(v))\leq 2(\chi_2+1)$.

Now, if $U_1'$ crosses $\kappa_1$, let $e_1'$ be the $1$--cube dual to $U_1'$, so that $\kappa=\kappa_0e_1\kappa_1'e_1'\kappa_1''$.  Otherwise, $U_1'$ does not cross $\kappa_1$ but is proximate to $y$.  In this case, letting $\kappa'_1=\kappa_1$, we have $\kappa=\kappa_0e_1\kappa_1'$.  So in either case, $\pi\circ\kappa_1'$ is a geodesic of $T$ from $\pi(N(U_1))$ to a point at distance at most $2(\chi_2+1)$ from $\pi(U'_1)$.  Now, $\dist_T(\pi(U_1),\pi(U'_1))\geq 100(\chi_0+\chi_2+2L+1)$ by Lemma \ref{lem:K-large}, since we are assuming $U_1\sim_q U_1'$.  Thus $|\pi\circ\kappa_1'|>50\chi_2+50$.  Choose an edge $f$ of $T$ that lies on $\pi\circ\kappa_1'$ at a distance at least $20\chi_2+20$ from either endpoint.  Let $m_f$ be the midpoint of $f$ and let $V=\pi^{-1}(m_f)\cap\widetilde Y_q$, which is a vertical hyperplane in $\widetilde Y_q$.  The choice of $f$ ensures that $V$ separates $x$ and $y$ but is not proximate to either of them.  On the other hand, $f$ can also be chosen to lie at distance (measured in $T$) at most $60(\chi_0+\chi_2+2L+1)$ from $\pi(y)$, so another application of Lemma \ref{lem:K-large} shows that any $V'\neq V$ with $V\sim_qV'$ is too far from $y$ (measured in $T$) to be proximate to $y$.  Finally, from Construction \ref{cons:B6-walls}, we can choose $f$ so that any $V'$ such that $V'\sim_qV$ and $V\neq V'$ has the property that $\pi(x),\pi(V),\pi(y),\pi(V')$ occur along the axis $\pi(\widetilde Y_q)$ in that order, so the wall $[V]$ separates $x,y$ and $x$ is not proximate to $V'$ and hence not proximate to $[V]$.
\end{proof}

Now we can prove the theorem:

\begin{proof}[Proof of Theorem \ref{thm:cubulation-with-turns}]
From Lemma \ref{lem:b6-check}, $X^*=\langle X\mid \{\widetilde Y_q:q\in\mathcal Q\}\rangle$ is a simplified $B(6)$ presentation, where each $\widetilde Y_q$ is equipped with the wallspace structure from Construction \ref{cons:B6-walls}, and short innerpaths holds by Lemma \ref{lem:metric-small-cancellation}.  We will apply Theorem \ref{thm:cubulating X star}, taking $\mathcal R$ to be the full set of hyperplanes of $X$.  Hypothesis \eqref{LS:strong separation} is satisfied because of Lemma \ref{lem:exits-strong-separation}.  Hence applying Theorem \ref{thm:cubulating X star} and passing to the cube complex dual to the wallspace given in the theorem provides an action of $G$ on a CAT(0) cube complex $C_0$ such that one of the following holds for all $g\in G$:
\begin{itemize}
 \item $g$ is conjugate into $\langle q\rangle$ for some $q\in\mathcal Q$.
 \item $\|g\|_{C_0}>0$.
\end{itemize}
Indeed, conclusion \eqref{item:torsion} of Theorem \ref{thm:cubulating X star} cannot hold because $G$ is torsion-free.  Conclusion \eqref{item:not_preferred_cut} cannot hold since $g\neq 1$ implies $g$ is cut by a hyperplane in $\widetilde X$ (recall that the action on $\widetilde X$ was free) and all hyperplanes in $X$ are in $\mathcal R$.

Fix $q\in\mathcal Q$.  Then $q$ has axis in $\widetilde X$ lying in $\widetilde Y_q$, and every wall $W$ of $\widetilde{X^*}$ intersects $\widetilde Y_q$ in a $\sim_q$--class of hyperplanes, by Theorem \ref{thm:520}, so either $W$ intersects $\widetilde Y_q$ in a single hyperplane, or $W$ intersects $\widetilde Y_q$ in a pair of disjoint $\sim_q$--equivalent hyperplanes both cutting $q$; in the latter case, our choice of halfspaces in $\widetilde Y_q$ implies that $q$ is not cut by $W$.  In the former case, $W\cap\widetilde Y_q$ is either a wandering hyperplane that does not cut $q$, or it is one of $2L$ hyperplanes that cut $q$ and are unique in their $\sim_q$--classes (see Construction \ref{cons:B6-walls}).  Hence $\|q\|_{C_0}=2L$.  

Moreover, suppose that distinct hyperplanes $H,H'$ of $\widetilde X$ belong to the same wall of $\widetilde{X^*}$.  Then there is a sequence $\widetilde Z_1,\ldots,\widetilde Z_n$ of convex subcomplexes of $\widetilde X$, each equal to some translate of some $\widetilde Y_q$, and a sequence $H=H_0,H_1,\ldots,H_n=H'$ such that $H_i,H_{i+1}$ cross $\widetilde Z_i$ in equivalent hyperplanes of $\widetilde Z_i$.  Suppose that $H,H'$ both cross $\pi^{-1}(\tilde v)$ for some vertex $\tilde v\in T$.  Then $\pi^{-1}(\tilde v)$ is a vertex space of $\widetilde Z_1$, and thus $H,H'$ are hyperplanes belonging to the same wall and crossing a common cone $\widetilde Z_1$, so by Theorem \ref{thm:520}, letting $\widetilde Z_1=g\widetilde Y_q$, we have $g^{-1}H\sim_q g^{-1}H'$, while $\dist_T(\pi(g^{-1}H\cap \widetilde Y_q),\pi(g^{-1}H')\cap\widetilde Y_q)=0$, which is a contradiction with Lemma \ref{lem:K-large}.  This shows that each wall of $\widetilde{X^*}$ intersects each vertex space $\pi^{-1}(\tilde v)$ of $\widetilde X$ in at most one hyperplane.  In fact, as long as the constant $K$ was chosen large enough in terms of the acylindricity constant for the $G$--action on $T$, we have that, if $\gamma\in G$ is elliptic, then $\gamma$ is cut by exactly one wall in $\widetilde X^*$ for every hyperplane in $\widetilde X$ cutting $\gamma$, so $\|\gamma\|_{C_0}=\|\gamma\|_{\widetilde X}$.  In particular, 
$$\|t\|_{C_0}=\|t\|_{\widetilde X}=2L=\|q\|_{C_0},$$
for $t\in\mathcal T_v$, as required.  We have also shown that all nontrivial elements of $G$ have positive translation length in $C_0$, so $G$ acts on $C_0$ freely by Lemma \ref{lem:free}. This argument also shows that the $G$--action on $C_0$ is wall-independent for each $\mathcal T_v$, since that was true in $\widetilde X$.

\textbf{Wall-independence of $\mathcal Q$.} Let $q,q'\in\mathcal Q$ and let $g\in G$.  Suppose that $gqg^{-1}$ and $q'$ are cut by infinitely many common hyeperplanes of $C_0$, i.e. by infinitely many common walls of $\widetilde{X^*}$ coming from Theorem \ref{thm:cubulating X star}.  As noted above, any wall $W$ cutting $gqg^{-1}$ (respectively, $q'$) is an equivalence class of vertical walls in $\widetilde X$, and $W\cap g\widetilde Y_q$ is a single vertical hyperplane of $g\widetilde Y_q$ that cuts $gqg^{-1}$.  If $W$ also cuts $q'$, then $W$ intersects $\widetilde Y_{q'}$ in a single vertical hyperplane cutting $q'$.  Let $H,H'\in\mathcal V(\widetilde X)$ be the hyperplanes respectively cutting $q,q'$ and lying in $W$.  

Recall from \cite[Sec. 5]{Wise:QCH} that the \emph{thickened carrier} $T(W)$ is the union of the carriers of hyperplanes of in $W$, together with any $a\widetilde Y_p,\ a\in G,\ p\in\mathcal Q$ crossed by such a hyperplane.  By the $B(6)$ condition and \cite[Thm. 5.28]{Wise:QCH}, any two points of $T(W)$ are joined by a $\widetilde X$--geodesic lying in $T(W)$, and since $\pi:\widetilde X\to T$ is a restriction quotient, any such geodesic is sent by $\pi$ to an unparameterised geodesic in $T$.  So $\pi(T(W))$ is a subtree of $T$ consisting of all the edges dual to $\pi$--images of hyperplanes in $W$, together with exactly those axes (of conjugates of elements of $\mathcal Q$) that contain such edges.  Moreover, $\pi:T(W)\to \pi(T(W))$ sends geodesics to geodesics.  

Recall from Theorem \ref{thm:520} that there is exactly one hyperplane of $W$ intersecting $g\widetilde Y_q$, namely $H$.  Moreover, any subcomplex of the form $a\widetilde Y_p,\ a\in G,\ p\in\mathcal Q$ that crosses $H$ has the property that $a\axis(p)\cap g\axis(q)$ lies in the $\chi_1$--neighbourhood of $\pi(H)$, by Lemma \ref{lem:axis-subcomplex-overlap}.  From our choice of $K$, it follows that any path in $\pi(T(W))$ from $g\axis(q)$ to its complement must pass $\chi_1$--close to $\pi(H)$.  In particular, for all $n\in\integers-\{0\}$, we have $\pi(T(gq^{nK}g^{-1}W))\cap \pi(\axis(q')))=\emptyset$ and, more strongly, $\dist_T(\pi(T(gq^{nK}g^{-1}W)),\pi(\axis(q'))\geq \kappa n$, where $\kappa>0$ is independent of $W,g,q$.  

On the other hand, if infinitely many walls cut $gqg^{-1}$ and $q'$, then we may assume that infinitely many $g\langle q\rangle g^{-1}$--translates of $W$ cut both elements.  Therefore, for infinitely many values of $n$, we have that $\pi(T(gq^{nK}g^{-1}W))$ intersects both $g\axis(q)$ and $\axis(q')$, and hence intersects the projection of $\axis(q')$ to $g\axis(q)$, which has diameter bounded by $\chi_1$.  This is a contradiction, and proves that $\mathcal Q$ is wall-independent for the $G$--action on $C_0$.  

\textbf{Wall-independence of $\mathcal T_v\cup\mathcal Q$:}  Finally, for any $v\in\vertices(\bar T)$ and $t\in\mathcal T_v$, the hyperplanes of $\widetilde X$ that cut $t$ are all wandering, and therefore the walls in $\widetilde{X^*}$ that cut $t$ are formed from equivalence classes of wandering hyperplanes, and it was already noted that no such wall cuts any element of $\mathcal Q$.  Thus no wall cuts both $t$ and $q\in\mathcal Q$.  Meanwhile, any two elements of $\mathcal T_v$ (or $\mathcal Q$) were shown above to be cut by only finitely many common walls.  So $\mathcal T_v\cup\mathcal Q$ is wall-independent.
\end{proof}

We now apply Theorem \ref{thm:cubulation-with-turns} to acylindrical splittings over cyclic groups.

\begin{cor}[Cubulating acylindrical cyclic splittings]\label{cor:cyclic-splitting-turns}
Let $G$ act cofinitely and without inversions on the tree $T$.  Suppose that:
\begin{enumerate}
 \item \label{item:cyclic-acyl} The $G$--action on $T$ is $2$--acylindrical.
 
 \item \label{item:cyclic-cyclic} For each $e\in\edges(\bar T)$, the edge-group $G_e$ is a maximal infinite cyclic subgroup of $G$ generated by an element $t_e$.

 \item \label{item:cyclic-vertices} For each $v\in\vertices(\bar T)$, the vertex group $G_v$ acts freely and without inversions in hyperplanes on a CAT(0) cube complex $\widetilde X_v$, with finitely many orbits of hyperplanes.
 
 \item \label{item:cyclic-vertices} For each $v\in\vertices(\bar T)$, let $\mathcal T_v$ be the set of images of elements $t_e$ under edge maps $G_e\to G_v$ for all $e\in\edges(\bar T)$ incident to $v$.  Then the $G_v$--action on $\widetilde X_v$ is wall-independent for $\mathcal T_v$.

 \item There exists $L\in\integers_{>0}$ such that $\|t\|_{\widetilde X_v}=L$ for all $t\in\mathcal T_v,\ v\in\vertices(\bar T)$.  \label{item:cyclic-length}

 \item \label{item:separable-edge-groups}  The edge groups of $G$ are separable in $G$.
\end{enumerate}
Let $\mathcal Q\subset G$ be a finite set of $T$--hyperbolic elements such that each $q\in\mathcal Q$ generates a maximal cyclic subgroup of $G$. Then $G$ acts freely, with finitely many orbits of hyperplanes, on a CAT(0) cube complex $C_\sharp$, and:
\begin{enumerate}[(a)]
 \item \label{item:cyclic-conclusion-equalised} $\|q\|_{C_\sharp}=\|t\|_{C_\sharp}=2L$ for all $q\in\mathcal Q$ and $t\in\cup_{v\in\vertices(\bar T)}\mathcal T_v$.
 
 \item \label{item:cyclic-conclusion-independence} For each $v\in\vertices(\bar T)$, the set $\mathcal T_v\cup\mathcal Q$ is wall-independent for the action of $G$ on $C_\sharp$.
\end{enumerate}

\end{cor}

\begin{rem}
The proof would work for $k$--acylindrical actions, with $k>2$, at the expense of some additional complications.  Since, in our applications, we have $2$--acylindricity by Lemma \ref{lem:topmost-edges-acylindrical}, we have hypothesised it here.  In the proof, we work under the weaker hypothesis of $k$--acylindricity, except in one place that we will point out.
\end{rem}

\begin{proof}[Proof of Corollary \ref{cor:cyclic-splitting-turns}]
By hypothesis, the edge groups are separable, so since each $q\in\mathcal Q$ is hyperbolic on $T$ and $\mathcal Q$ is finite, we can pass to a finite-index normal subgroup $G'\leq G$ with the property that, for all $q\in\mathcal Q$, any generator of $\langle q\rangle\cap\mathcal Q$ has normal form projecting to an embedded cycle in the underlying graph $G'\backslash T$.

The hypotheses are still satisfied by the induced cyclic splitting of $G'$, once we replace $\mathcal Q\cup \bigcup_v\mathcal T_v$ by an algebraic elevation (see Definition \ref{defn:finite-set-elevation}).  Hence, if we exhibit a free $G'$--action on a CAT(0) cube complex $C_\sharp'$ satisfying conclusions \eqref{item:cyclic-conclusion-equalised},\eqref{item:cyclic-conclusion-independence} for the elevation, then we are done, by Lemma \ref{lem:star}.  We thus assume that the above embedding property already held for $\mathcal Q$ in $G$, so in particular, for all $q\in\mathcal Q$, consecutive edges in the axis of $q$ in $T$ belong to distinct $G$--orbits.

\textbf{Checking Definition \ref{defn:graph-walls}.}  We first verify that Definition \ref{defn:graph-walls} applies.  By hypothesis, we have a $G$--cocompact action on $T$ without inversions.  Since the action is acylindrical and each $\langle q\rangle,\ q\in\mathcal Q$ is a maximal cyclic subgroup of $G$, and the hypothesised free cubulations of the vertex stabilisers force $G$ to be torsion-free, each $\langle q\rangle=\stabilizer_G(\axis(q))$.  Finiteness of $\mathcal Q$ and acylindricity of the $G$--action on $T$ also provide a constant $\chi<\infty$ such that $\axis(q)\cap g\axis(q')$ has diameter at most $\chi$ for $g\in G,\ q,q'\in\mathcal Q$, unless $q=q'$ and $g\in\langle q\rangle$; this is clear from the definition of acylindricity.  This verifies Definition \ref{defn:graph-walls}.\eqref{item:stass-Q}.  Item \eqref{item:stass-cubulated} is hypothesised.

Now we verify item \eqref{item:stass-shape}.  Let $e\in\edges(\bar T)$, so that $G_e=\langle t_e\rangle$, and the images of $t_e$ in $G_{e^{\pm}}$ are elements $t^\pm\in \mathcal T_{e^\pm}$.  By definition, $G_{e}^\pm=\langle t^\pm\rangle\leq G_{e^\pm}$, and we have an isomorphism $\phi_e:G_e^-\to G_e^+$ given by $\phi_e(t^-)=t^+$.  Since $G_{e^\pm}$ acts freely on $\widetilde X_{e^\pm}$, without inversions in hyperplanes, \cite{Haglund:semisimple} supplies combinatorial geodesic axes $\widetilde X_e^\pm\subset\widetilde X_{e^\pm}$ for $t^\pm$.  Now, by hypothesis \eqref{item:cyclic-length}, the translation length of $t^\pm$ along $\widetilde X_e^\pm$ is $L$.  Since $\widetilde X_e^\pm$ is just a copy of the standard tiling of $\reals$ by $1$--cubes, the fact that $t^-$ and $t^+$ have the same translation length on their respective axes implies that there is a cubical isomorphism $\psi_e:\widetilde X_e^-\to\widetilde X_e^+$ such that $\psi_e(t^-x)=t^+\psi_e(x)$ for all $x\in\widetilde X_e^-$.

\textbf{Dehn-twisted walls.}  By hypothesis, $\widetilde X$ comes from applying Proposition \ref{prop:joining-walls} to an input graph of compatibly cubulated groups.  Recall from Remark \ref{rem:arranging-standing-assumption-with-isomorphisms} that there is a degree of freedom in choosing $\widetilde X$, namely we can modify the immersed walls in the space $Z$ from the proof of Proposition \ref{prop:joining-walls} by Dehn twisting in the edge-cylinders.  For any $\widetilde X$, item \eqref{item:join-vertex-space} says that each $\widetilde X_v$ is an isometric subcomplex of $\widetilde X$,  so translation lengths of all $t\in\bigcup_v\mathcal T_v$ are still $L$, and each $\mathcal T_v$ is still wall-independent for the $G_v$--action on its vertex space.

So, for each $\bar e\in\edges(\bar T$), let $n_e\in\integers_{\ge0}$ be a constant to be determined, and, as in Remark \ref{rem:arranging-standing-assumption-with-isomorphisms}, use Proposition \ref{prop:joining-walls} to produce a free action of $G$ on a CAT(0) cube complex $\widetilde X$ (satisfying all the conclusions of Proposition \ref{prop:joining-walls}), using $n_e$--fold Dehn-twisted walls in each edge space $X_e$.

Let $R$ be a constant to be chosen below, depending only on the input graph of groups, the cubulations of the vertex and edge groups, the edge-homomorphisms, and the finite set $\mathcal Q$.

Let $\bar e$ be an edge of $\bar T$.  Choose the $n_{\bar e}$ such that
$$\min\{n_{\bar e}, |n_{\bar e}\pm n_{\bar f}|:\bar e,\bar f\in\edges(\bar T),\ \bar e\neq \bar f\}\cdot L>R.$$

For instance, choose $n_{\bar e}=10^{\zeta_{\bar e}}(R+1)/L$, where the $\zeta_{\bar e}$ are positive integers that are distinct as we vary $\bar e$ among the finitely many edges $\bar e$ of $\bar T$.

\textbf{Choosing $R$ and verifying exits.}  We now choose $R$ and check that our choice of $n_e,\ e\in\bar T$, ensures that $\widetilde X$ has walls exiting $\mathcal Q$ (see Definition \ref{defn:walls-exit}).

Fix $q\in\mathcal Q$.  Then, by acylindricity, $\axis(q)$ contains a subgeodesic of the form $f_0e_0\delta f_1e_1$, where $\delta$ is a path whose length is at most $m$ (we allow the possibility that $\delta$ is trivial, but also that $\delta$ is empty and $e_0=f_1$), and $e_0,e_1$ are lifts of edges $\bar e_0,\bar e_1$ of $\bar T$, and the following holds for $i\in\{0,1\}$.  First, let $v_i=f_i\cap e_i$, let $t_i$ generate the stabiliser of $e_i$, let $s_i$ generate the stabiliser of $f_i$, let $t_i^\pm$ and $s_i^\pm$ be the images of $t_i,s_i$ in their terminal and initial vertex stabilisers.  Then $\langle t_i^-\rangle\cap \langle s_i^+\rangle=\{1\}$.  Call $f_ie_i$ a \emph{good length--$2$ subpath of $\axis(q)$}.  We can choose $f_0e_0,f_1e_1$ to be an innermost pair of good length--$2$ subpaths, in the sense that $e_0\delta f_1$ does not contain a good length--$2$ subpath.  Hence, if $f,e$ are consecutive edges of $e_0\delta f_1$, then $\stabilizer_G(e)=\stabilizer_G(f)$.  Let $v_i$ be the initial vertex of $e_i$.

Now, the $2$--acylindricity part of hypothesis \eqref{item:cyclic-acyl}  implies more: either $\delta=\emptyset$, and $f_1=e_0$, so our path is $f_0e_0e_1$, or $\delta$ is trivial, so our path is $f_0e_0f_1e_1$.  Up to replacing $q$ by $q^{-1}$, we can assume that $v_0$ projects to the initial vertex of $\bar e_0$ in $\bar T$.    

By the wall-independence hypothesis, the set $\mathcal C_i$ of hyperplanes in the vertex space $\widetilde X_{v_i}$ that cut both $s_i^+$ and $t_i^-$ is finite.  Let $A(s_i^+),A(t_i^-)$ be the axes of $s_i^+,t_i^-$ on $\widetilde X_{v_i}$.  Let $L_0$ be a subgeodesic of $A(t_i^-)$ that is as short as possible with the property that it crosses all elements of $|\mathcal C_0|$.  Define $L_1\subset A(s_1^+)$ analogously, for $|\mathcal C_1|$.  Note that $L_0$ and $L_1$ are determined by the length--$2$ paths $\bar f_0\bar e_0$ and $\bar f_1\bar e_1$ in $\bar T$ and the element $q$, so since there are finitely many such paths and $|\mathcal Q|<\infty$, there is a constant $R$, depending only on the graph of cubulated groups and the set $\mathcal Q$, such that $|L_0|,|L_1|\leq R$.  

Now, if $\delta$ is trivial, then $\psi_{e_0}(L_0)$ is a subgeodesic of $A(s_1^+)=A(t_0^+)$, and we can choose the constant $R$ (uniformly) with the property that $\dist_{\widetilde X_{v_1}}(\psi_{e_0}(L_0),L_1)\leq R$.  Indeed, there are finitely many possibilities for that distance because, up the the $G$--action, there are finitely many choices for $q$ and for the subpath $f_0e_0e_1$, and we take the maximum such distance.  Similarly, if $\delta$ is nontrivial, then $\psi_{f_1}(\psi_{e_0})(L_0)$ is a subsegment of $A(s_1^+)$ whose distance from $L_1$ we can assume is less than $R$, where $R$ depends only on the input graph of cubulated groups and the set $\mathcal Q$.

We now assume that the $n_{\bar e}$ were chosen as above, using this $R$.  Let $N=n_{\bar e_0}$ if $\delta=\emptyset$, and let $N=n_{\bar e_0}\pm n_{\bar f_1}$ if $\delta$ is trivial, where the sign is chosen according to whether the terminal vertex of $e_0$ projects to the initial or terminal vertex of $\bar f_1$.  In the latter case, since consecutive edges of $\axis(q)$ are in different $G$--orbits, $\bar e_0\neq \bar f_1$.

Now let $W$ be a wandering hyperplane in $\widetilde X$ that intersects $\widetilde Y_q$.  Up to translating by $\langle q\rangle$, either $\pi(W)\cap \axis(q)$ has diameter bounded in terms of the translation length of $q$ on $T$ (in which case the exit condition from Definition \ref{defn:walls-exit} is satisfied by the given $W,q$), or $W$ intersects $\widetilde X_{v_0}$ and $\widetilde X_{v_1}$; these intersections are respectively single hyperplanes $H_0,H_1$.  

If $H_0$ crosses $L_0$, then our choice of $n_e$ for our Dehn twists implies one of the following:
\begin{itemize}
    \item If $\delta=\emptyset$, the hyperplane $H_1$ crosses $(s_1^+)^N\psi_{e_0}(L_0)$, which is at distance at least $N\|s_1^+\|_{\widetilde X_{v_1}}-R>0$ along $A(s_1^+)$ from $L_1$.  Hence $H_0\in\mathcal C_0$ implies $H_1\not\in\mathcal C_1$.

    \item If $\delta$ is trivial, then $H_1$ crosses $(s_1^+)^N\psi_{f_1}(\psi_{e_0})(L_0)$, so again  $H_0\in\mathcal C_0$ implies $H_1\not\in\mathcal C_1$.
\end{itemize}

In either case, if $W$ intersects the vertex space corresponding to the initial vertex of $f_0$, then it cannot intersect the vertex space corresponding to the terminal vertex of $e_1$.  This implies that $\pi(W)\cap\axis(q)$ has diameter (in $T$) at most $4$.  Thus the exit condition (Definition \ref{defn:walls-exit}) holds for some $\chi_0$ depending only on the maximum of the $T$--translation lengths of $q\mathcal Q$.

\textbf{Recubulating with Theorem \ref{thm:cubulation-with-turns}.}  We have verified that all of the hypotheses of Theorem \ref{thm:cubulation-with-turns} are satisfied, so the theorem yields a free action of $G$ on a CAT(0) cube complex $C_\sharp$, with finitely many orbits of hyperplanes, such that each $q\in\mathcal Q$ and $t\in\cup_v\mathcal T_v$ has translation length $2L$, and each $\mathcal T_v\cup\mathcal Q$ is wall-independent.  This completes the proof.
\end{proof}

\begin{rem}[Cocompactness strategy]\label{rem:cocompactness}
It is likely one can choose the $G$--action on $C_\sharp$ from Corollary \ref{cor:cyclic-splitting-turns} to be cocompact if each $G_v$ acts on $\widetilde X_v$ cocompactly, and each $t\in \mathcal T_v$ generates a convex-cocompact cyclic subgroup of $G_v$.  First, one would have to arrange for the cube complex $\widetilde X$ from Proposition \ref{prop:joining-walls} to be cocompact.  In the simple scenario where the axes of $t_e^{\pm}$ in $\widetilde X_{e^\pm}$ and  have equivariantly isomorphic convex hulls was discussed in Remark \ref{rem:splice-cocompact}; the general case, where $t_e^+$ and $t_e^-$ are convex-cocompact in their respective vertex groups but do not have isomorphic convex hulls, does not seem likely to be much worse, under an additional assumption that would hold for the superlinear hierarchy of a free-by-$\integers$ group with cocompactly cubulated vertex groups, namely that at least one of $t_e^\pm$ is a rank-one isometry of $\widetilde X_{e^\pm}$.  

Once $G$ acts cocompactly on $\widetilde X$, arguing as in Remark \ref{rem:splice-cocompact} should ensure that the $t_e^\pm$ and $q\in\mathcal Q$ are convex-cocompact on $\widetilde X$.  The main issue would be to show that the walls constructed in Theorem \ref{thm:cubulation-with-turns} using the cubical presentation $\langle X\mid\{\widetilde Y_q:q\in\mathcal Q\}\rangle$ give a cocompact dual cube complex.  Now, each wall $W$ in $\widetilde X^*=\widetilde X$ from the proof of Theorem \ref{thm:cubulation-with-turns} is the union of hyperplanes of $\widetilde X$ belonging to the same $G$--orbit, and from the construction, $\stabilizer_G(W)$ acts cofinitely on the set of these hyperplanes, and if $H$ is a constituent hyperplane of $W$, then $\stabilizer_G(H)\leq \stabilizer_G(W)$.  So, since $G$--cocompactness of $\widetilde X$ implies $\stabilizer_G(H)$--cocompactness of $N(H)$, the $\stabilizer_G(W)$--action on $N(W)$ will be cocompact.  By choosing the constant $K$ in the proof of Theorem \ref{thm:cubulation-with-turns} sufficiently large, we imagine an argument similar to the thickened carrier discussion in the proof of Theorem \ref{thm:cubulation-with-turns} can be used to show that the convex hull $C(W)$ of $W$ in $\widetilde X$ is $\stabilizer_G(W)$--cocompact, (and probably that $\stabilizer_G(W)$ is the free product of the stabilisers of its constituent hyperplanes).  If $W_1,\ldots,W_n$ are pairwise crossing walls, then $C(W_1),\ldots,C(W_n)$ pairwise intersect in $\widetilde X$, so by the Helly property, they all intersect in some point $x$.  Since their stabilisers act uniformly cocompactly, it means that we can choose hyperplanes $H_i$ of $W_i$ so that each $H_i$ is uniformly close to $x$, and there are only finitely many possibilities for this collection since $G$ acts on $\widetilde X$ cocompactly.  This would verify cocompactness of the $G$--action on $C_\sharp$.  The fact that $q\in\mathcal Q$ is a Morse element of $G$ (as in Remark \ref{rem:splice-cocompact}) would then imply that $q$ remains convex-cocompact on $C_\sharp$.  The remaining thing would be to check this also holds for  $t\in\cup_v\mathcal T_v$.   
\end{rem}

\section{Applications to superlinear growth automorphisms}\label{sec:conclusion}
Let $F$ be a finite rank free group and let $\Phi\in\Out(F)$ be a polynomially growing outer automorphism with polynomial growth rate $d\geq 1$.  Let $G=F\rtimes_\Phi\integers$ be the mapping torus.

\begin{thm}\label{thm:main}
If $G$ is unbranched with tubular termini, then $G$ acts freely on a CAT(0) cube complex.
\end{thm}

\begin{proof}
By Lemma \ref{lem:star} and Remark \ref{rem:make-upg}, we can assume that $\Phi$ is UPG.  If $d\leq 1$, then the cubulation exists by Corollary \ref{cor:cubulated-with-exits}, so assume $d\geq 2$.

Using the top strata decomposition from Remark \ref{rem:topmost-edge-vertex-groups} and Lemma \ref{lem:topmost-edges-acylindrical}, and the assumption $d\geq 2$, we have a nontrivial $2$--acylindrical cyclic hierarchy for $G$ that terminates in a finite collection $G^i_0$ of groups, each of which is the mapping torus of a UPG automorphism of at most linear growth.  Since $G$ is unbranched, Lemma \ref{lem:unbranched-char} implies that each $G_0^i$ is unbranched.

We now induct on the length of the hierarchy, which is to say on $d$.  If $d=2$, $G$ is a finite graph $\Delta_1=G\backslash T_1$ of groups where the vertex groups are various $G_0^i=F^i\rtimes\langle t_i\rangle$.  The edge groups are infinite cyclic and the $G$--action on the Bass-Serre tree $T_1$ is acylindrical.  At each vertex group $G^i_0$, the set of incident edge groups is $\{\langle t_i\rangle\}\cup \{p^i_jt_i:j\in J_i\}$ for some finite set $J_i$.  Moreover, $G_0^i$ admits an acylindrical action on a tree $T_0^i$ coming from Propopsition \ref{prop:split-over-tori} (we allow the possibility that $T_0^i$ is a single vertex, in the case where $G_0^i=F_0^i\times \langle t_i\rangle$). 

If $T_0^i$ is nontrivial, then the set $ \{p^i_jt_i:j\in J_i\}\subset G_0^i$ is a set of  future attaching elements, in the sense of Definition \ref{defn:linear-hyperbolics}, because $T_0^i$--elliptic elements do not have cyclic centralisers.  We can now apply Corollary \ref{cor:cubulated-with-exits} to $G_0^i$ and $\mathcal Q_i= \{p^i_jt_i:j\in J_i\}$ to obtain an action of $G_0^i$ on a CAT(0) cube complex $\widetilde X_0^i$ whose walls exit the set $\mathcal Q_i$.

Each vertex group $G_0^i(w)$ in $G_0^i$ coming from the action on $T_0^i$ has the form $F_w\times \langle t_w\rangle$, so the singleton $\{t_w\}$ is obviously wall-independent for the action of the vertex group.  Moreover, Corollary \ref{cor:cubulated-with-exits} says that all of the $t_w$ have the same translation length in $\widetilde X_0^i$.  By subdividing each $\widetilde X_0^i$, we obtain some $L>0$ such that all of the $t_i$ have translation length $L$ on their $\widetilde X_0^i$.  

We then apply Theorem \ref{thm:cubulation-with-turns}  for each $i$ to modify $\widetilde X_0^i$ so that the free $G_0^i$--action has the property that $\mathcal Q_i\cup\{t_i\}$ is wall-independent and $\|t\|_{\widetilde X_0}=\|p^i_jt\|_{\widetilde X_0}=2L$ for all $p^i_jt\in\mathcal Q_i$.  (Theorem \ref{thm:cubulation-with-turns} does not require $2$--acylindricity, just acylindricity, and allows non-cyclic edge groups.)

If $T_0^i$ is trivial, i.e. $G_0^i=F_0^i\times\langle t_i\rangle$, then the incident edge groups correspond to quadratically growing edges $e$ in the relative train track representative, with $e$ mapping to $eP$, where $P$ is not a Nielsen path.  Hence $P$ is not in the vertex space corresponding to $F_0^i$, so the edge group corresponding to $e$ has image in $T_0^i$ generated by $\langle t_i\rangle$.  In other words, $\mathcal Q_i=\emptyset$ in this case.  Hence any free action of $T_0^i$ on a CAT(0) cube complex $\widetilde X_0^i$ vacuously has $\mathcal Q_i\cup\{t_i\}$ wall-independent and can be subdivided so that $\|t\|_{\widetilde X_0}=\|p^i_jt\|_{\widetilde X_0}=2L$ for all $p^i_jt\in\mathcal Q_i$.

Now Corollary \ref{cor:cyclic-splitting-turns}, applied to $G$ and the action on $T_1$, with the vertex group cubulations $\widetilde X_0^i$ as above, yields a free action of $G$ on a CAT(0) cube complex $C_\sharp$.  This completes the case $d=2$, but for the purposes of the induction, recall from Corollary \ref{cor:cyclic-splitting-turns} that, given a finite set $\mathcal Q$ of $T_1$--hyperbolic elements generating maximal cyclic subgroups of $G$, the action can be chosen so that $\|q\|_{C_\sharp}=\|t\|_{C_\sharp}$ whenever $t\in\cup_i\{t_i\}$ and each $\{t_i\}\cup\mathcal Q$ is wall-independent.  By subdividing, we can arrange for all of the preceding translation lengths (as $i$ varies) to be equal.

For $d>2$, we can now conclude by applying Corollary \ref{cor:cyclic-splitting-turns} inductively using the superlinear hierarchy of $G$.  (At each stage, the set $\mathcal Q$ is the set of non-elliptic generators of edge groups in the splitting at the next stage.)

Each application of Corollary \ref{cor:cyclic-splitting-turns} is justified because: the splittings in the superlinear hierarchy are $2$--acylindrical by Lemma \ref{lem:topmost-edges-acylindrical}; the hypotheses of the corollary that concern the cubulations of the vertex groups hold by induction, and hypothesis \eqref{item:separable-edge-groups} holds because abelian subgroups of free-by-cyclic groups are separable \cite[Cor. 2.10]{HughesKudlinska}.
\end{proof}

\begin{cor}[Fast cubulation]\label{cor:fast-cubulated}
If $\Phi\in\Out(F)$ has polynomial growth and is fast, then $G=F\rtimes_\Phi\integers$ acts freely on a CAT(0) cube complex.
\end{cor}

\begin{proof}
Lemma \ref{lem:fast-tubular} says that $G$ is unbranched with tubular termini, so Theorem \ref{thm:main} applies.
\end{proof}

\subsection{Genericity of fast mapping tori}\label{subsection:genericity}
One can ask if there is a reasonable sense in which being fast is a generic property of polynomial-growth automorphisms of $F_r$.  This is complicated by the fact that the usual ways of choosing elements of $\Out(F_r)$ randomly generically give irreducible elements when $r\geq 3$ (see \cite{KMPT}).  So having polynomial growth is non-generic in those senses.  We  use the following notion of genericity \emph{among UPG automorphisms} instead.  The observation that choosing a random UPG element amounts to choosing relative train track suffixes at random on a fixed graph was suggested by Naomi Andrew.

Let $N(r)=3r/2-1$.  Let $\Gamma$ be a connected graph of rank $r$ with at most $N(r)$ edges.  Let $f:\Gamma\to\Gamma$ be an improved relative train track map, i.e. a homotopy equivalence fixing all vertices of $\Gamma$, such that there is an ordering $E_1,\ldots,E_s$ of the edges and such that $f(E_i)=E_iP_i$ for all $i$, where $P_i$ is a (possibly trivial) immersed closed path using the edges $E_1,\ldots,E_{i-1}$ only.  Let $\ell(f)=\max_i|P_i|$. 

The map $f$ represents a UPG element of $\Out(F_r)$ and that all UPG elements have such representatives (by \cite[Thm. 5.1.8]{BFH:tits-I}; see \cite[Thm. 1.1]{BFH:kolchin} for the bound on $N(r)$).  We will not need to work with all the properties of $f$ given by aforementioned theorems, which is why we only stated some of them explicitly.

\begin{defn}\label{defn:RTT-sets}
For a given graph $\Gamma$ of rank $r$ and at most $N(r)$ edges, let $\ttm_\Gamma^r$ be the set of all relative train tracks $f:\Gamma\to\Gamma$ satisfying the conclusions of \cite[Thm. 5.1.8]{BFH:tits-I}.  Let $\ttm_\Gamma^r(L)$ be the subset consisting of those $f$ with $\ell(f)\leq L$.
\end{defn}

\begin{defn}\label{defn:ttm-graphs}
Let $\mathcal G_r$ be the set of connected, rank-$r$ graphs $\Gamma$ such that $\ttm^r_\Gamma$ is infinite and $|\edges(\Gamma)|\leq N(r)$.  For $\Gamma\in\mathcal G_r$, let $\edges_0(\Gamma)$ be the set of non-separating edges of $\Gamma$ and let $\edges_1(\Gamma)$ be the set of separating edges.  Let $\edges'_1(\Gamma)$ be the set of separating edges where one complementary component has rank at most $1$.
\end{defn}

\begin{defn}\label{defn:train-track-counting-sets}
Let $\ttm^r(L)=\bigcup_{\Gamma\in\mathcal G_r}\ttm_\Gamma^r(L)$ and let $\ttm^r=\bigcup_L\ttm^r(L)$.  For $f\in \ttm^r$, let $d(f)$ be the polynomial growth rate of the outer automorphism $O(f)$ represented by $f$.  Let $\Fast_{r}(L)$ be the set of such $f$ with $\ell(f)\leq L$ and $d(f)=r-1$.
\end{defn}

\begin{lem}[Exponential growth rates]\label{lem:graph-rank-growth}
For each $r\geq 2$, there exists $k_r>1$ such that the following holds.  Let $\Gamma$ be a finite connected graph of rank $r$ with at most $N(r)$ edges.  Let $v\in\vertices(\Gamma)$ and let $\gate_{\Gamma,v}(L)$ be the set of immersed closed paths in $\Gamma$ that have length at most $L$ and initial/terminal vertex $v$.  Let $k_{\Gamma,v}=\limsup_{L\to\infty}|\gate_{\Gamma,v}(L)^{1/L}|$.  Then $k_{\Gamma,v}\geq k_r$ for all $\Gamma,v$, and if $\Gamma'$ is a connected subgraph of $\Gamma$ of rank less than $r$, and $w\in\vertices(\Gamma')$, then $k_{\Gamma',w}<k_{\Gamma,v}$.
\end{lem}

\begin{proof}
Let $\widetilde\Gamma\to \Gamma$ be the universal cover and fix a lift $\tilde v\in\widetilde \Gamma$ of $v$.  The quantity $\gate_{\Gamma,v}(L)$ is just the number of $g\in\pi_1(\Gamma,v)\cong F_r$ such that $\dist_{\widetilde \Gamma}(g\tilde v,\tilde v)\leq L$, so $k_{\Gamma,v}$ is the exponential growth rate for the $F_r$--action on $\widetilde \Gamma$ as defined in \cite[Eqn. 1.2]{DahmaniFuterWise}.  

Let $\alpha$ be a path in $\Gamma$ from $v$ to $w$ such that $|\alpha|\leq N(r)$.  Let $\tilde\alpha$ be the lift of $\alpha$ at $\tilde v$, and let $\widetilde \Gamma'\leq \widetilde \Gamma$ be the component of the preimage of $\Gamma'$ that contains $\tilde w$.  Then $\gate_{\Gamma',w}(L)$ is the number of $g\in \stabilizer_{F_r}(\widetilde \Gamma')$ such that $\dist_{\widetilde \Gamma}(\tilde v,g\tilde v)\leq L+2N(r)$.  So $k_{\Gamma',w}$ is the exponential growth rate for the action of $\stabilizer_{F_r}(\widetilde \Gamma')$ on $\widetilde \Gamma$.  Hence \cite[Thm. 1.1]{DahmaniFuterWise} implies $k_{\Gamma',w}<k_{\Gamma,v}$, by the rank assumption on the subgraph $\Gamma'$.

Now we produce $k_r$.  By considering a spanning tree in $\Gamma$, we get $\widetilde\Gamma=F_r\cdot B$, where $B$ is the ball of radius $N(r)+1$ centred at $\tilde v$.  The Milnor-\u{S}varc lemma says that $\gate_{\Gamma,v}(N(r)+1)$ generates $F_r$, and that the orbit map $F_r\to \widetilde\Gamma$ given by $g\mapsto g\tilde v$ is a $(C_r,1)$--quasi-isometry for some $C_r$ depending only on $N(r)$, whee $F_r$ is given the word-metric associated to the generating set $\gate_{\Gamma,v}(N(r)+1)$.  In particular, for sufficiently large $L$,
$$\frac{1}{2}\lambda^{\frac{L}{2C_r}}\leq\gate_{F_r,1}(L/2C_r)\leq\gate_{\Gamma,v}(L),$$
where $\lambda>1$ is the uniform exponential growth rate (for word metrics) of $F_r$ (see \cite{delaHarpe}).  Since $\lambda$ is independent of the generating set and $C_r$ depends only on $r$, we get the required $k_r>1$.  
\end{proof}

The following is established in \cite[Lem. 2.16]{Macura:detour}:

\begin{lem}\label{lem:fast-edge}
Let $r\geq 2$ and let $f\in\ttm^r$.  Then $d(f)=r-1$ if and only if there exists an edge $E$ in $\Gamma$ (here $\Gamma$ is the domain of $f$) such that $E$ has polynomial growth rate $r-1$.  
\end{lem}

Now we can say in what sense having maximum possible growth rate is generic:

\begin{prop}[Being fast is generic]\label{prop:fast-generic}
Let $r\geq 2$.  Then 
$$\lim_{L\to\infty} \frac{|\Fast_r(L)|}{|\ttm^r(L)|}=1.$$
\end{prop}

Before the proof, we define some notation.

\begin{defn}\label{defn:fast-count}
For $\Gamma\in\mathcal G_r$, let $k_\Gamma(L)=|\ttm_\Gamma^r(L)|$ and let $k(L)=|\ttm^r(L)|$.  For each edge $E$ of $\Gamma$, let $k^E_\Gamma(L)$ be the number of $f\in\ttm_\Gamma^r(L)$ where the topmost edge in the filtration of $\Gamma$ associated to $f$ is $E$.  Let $m^E_\Gamma(L)$ be the number of such $f$ that have growth rate $r-1$ and topmost edge $E$, and $m_\Gamma(L)$ the number of such $f$ on $\Gamma$ with growth rate $r-1$ but no restriction on topmost edge.
\end{defn}

\begin{proof}[Proof of Proposition \ref{prop:fast-generic}]
Suppose $r\geq 2$.  Fix $\Gamma\in\mathcal G^r$.  

\textbf{Non-separating topmost edges.}  Let $E\in\edges_0(\Gamma)$ be a non-separating edge, with (possibly equal) vertices $v$ and $w$, and let $\Gamma'=\Gamma-\interior{E}$, which has rank $r-1$.  We first estimate $k^E_\Gamma(L)$.  Each relative train track map $f:\Gamma\to\Gamma$ with topmost edge $E$ restricts to a relative train track map $f'$ on $\Gamma'$, and $\ell(f')\leq L$.  There are $k_{\Gamma'}(L)$ choices for $f'$.  

Independently of the choice of $f'$, we extend $f'$ to some $f$ with the above properties by orienting $E$ and choosing a closed immersed path $P$ in $\Gamma'$ starting and ending at the terminal point of $E$, and declaring $f(E)=EP$.  Hence
$$k^E_\Gamma(L)\leq k_{\Gamma'}(L)[\gate_{\Gamma',v}(L)+\gate_{\Gamma',w}(L)]$$
if $v\neq w$, and 
$$k^E_\Gamma(L)\leq k_{\Gamma'}(L)\gate_{\Gamma',v}(L)$$
if $v=w$.  (The inequalities reflect that in principle, not all extensions of $f'$ to $f$ satisfy the extra conclusions of \cite[Thm. 5.1.8]{BFH:tits-I}.)

Now, $\rank(\Gamma')=r-1$, so $f'$  has growth rate $r-2$ if and only if $\Gamma'$ has an edge $E'$ which grows polynomially with rate $r-2$.  If $r>2$, then there are $m_{\Gamma'}(L)$ such choices for $f'$.  If $r=2$, then $N(r)=2$, and so $\Gamma'=E'$ is a loop fixed by $f'$, and we let $m_{\Gamma'}(L)=1$.  In either case, independently of $f'$, we extend to $f$ by choosing $P$ as above.

For a given $f'$ of growth rate $r-2$, let $E''$ be the topmost edge in the filtration of $\Gamma'$ associated to $f'$.  As in Lemma \ref{lem:stratify-linear-superlinear}, we can assume that $E''$ has the maximum possible growth rate among edges.  If $r>2$, then the map $f$ given by extending $f'$ via $f(E)=EP$ will have growth rate $r-1$ provided $P$ traverses $E''$, by \cite[Lem. 2.16]{Macura:detour} (and in this case \cite[Thm. 1.1]{BFH:kolchin} holds for $f$).  Hence, in this case, if $v\neq w$, we have
$$m_\Gamma^E(L)\geq m_{\Gamma'}(L)[\gate_{\Gamma',v}(L)+\gate_{\Gamma',w}(L)-\gate_{\Gamma''_v,v}(L)-\gate_{\Gamma''_w,w}(L)],$$
where $\Gamma''_v$ is the component of $\Gamma'-\interior{E''}$ containing $v$.  If $v=w$, we similarly get:
$$m_\Gamma^E(L)\geq m_{\Gamma'}(L)[\gate_{\Gamma',v}(L)-\gate_{\Gamma''_v,v}(L)].$$
If $r=2$, then $E''=E'$ and $\Gamma'=E'$ is a loop, and $m_\Gamma^E(L)=2L$ and $k^E_\Gamma(L)=2L+1$.

Now let $\epsilon>0$ be given.  Then there exists $L_0$ such that $\gate_{\Gamma',v}(L)-\gate_{\Gamma''_v,v}(L)>(1-\epsilon)\gate_{\Gamma',v}(L)$ for all $L\geq L_0$, and so $\gate_{\Gamma',v}(L)+\gate_{\Gamma',w}(L)-\gate_{\Gamma''_v,v}(L)-\gate_{\Gamma''_w,w}(L)>(1-\epsilon)[\gate_{\Gamma',v}(L)+\gate_{\Gamma',w}(L)]$, by Lemma \ref{lem:graph-rank-growth}.
Now, by induction on $r$, we can choose $L_0$ such that for all $L\geq L_0$, we have $m_{\Gamma'}(L)\geq (1-\epsilon)k_{\Gamma'}(L)$ (for any choice of $\Gamma$ and nonseparating edge $E$).  In the base case, we have $r=2=N(r)$ and this follows from  $m_\Gamma^E(L)=2L$ and $k^E_\Gamma(L)=2L+1$.

Hence for all $L\geq L_0$, we have for all $\Gamma\in\mathcal G_r$ and all nonseparating $E$ in $\Gamma$ that 
$$m_\Gamma^E(L)\geq (1-\epsilon)^2k_{\Gamma}^E(L).$$

\textbf{Edges with a small complementary component.} The preceding estimate holds by a virtually identical argument for edges $E$ of $\Gamma$ that are separating, but where one of the complementary components has rank $1$ (the statement holds up to uniformly enlarging $L_0$ --- the important thing is that there are only finitely many $\Gamma$ and $E$ involved, because of the bound $N(r)$ on the number of edges, so one can take maxima where needed).  

\textbf{Other separating edges.}  In the remaining case, $E$ is an edge with the property that removing $E$ from $\Gamma$ leaves two graphs $\Gamma_1,\Gamma_2$, respectively of ranks $r_1,r_2$ such that $1<r_1,r_2<r-1$ and $r_1+r_2=r$.  In this case, $m_\Gamma^E(L)=0$ for all $L$.  Let $E_1$ be a non-separating edge of $\Gamma_1$ and let $E_2$ be a non-separating edge of $\Gamma_2$.  Then for all sufficiently large $L$, we have 
$$k^E_{\Gamma_1}(L)k^E_{\Gamma_2}(L)\gate_{\Gamma_i}(L)\leq \epsilon  k_\Gamma^{E_{2-i+1}}(L).$$
Indeed, the right hand side counts relative train tracks $f$ on $\Gamma$ where $E_1,E_2$ are topmost for the restrictions to $\Gamma_1,\Gamma_2$, and $f(E)=E\cdot P$, where $P$ is confined to $\Gamma_i$.  For any such $f$, we can view $E_{2-i+1}$ as the topmost edge in $\Gamma$, since $E$ does not map over it.  Now, the proportion of $f$ with topmost edge $E_{2-i+1}$ such that $f(E_{2-i+1})$ stays in $\Gamma_{2-i+1}$ tends to $0$ as $L\to\infty$, by Lemma \ref{lem:graph-rank-growth}.

We also have $k_\Gamma^E(L)\leq k_{\Gamma_1}(L)k_{\Gamma_2}(L)[\gate_{\Gamma_1}(L)+\gate_{\Gamma_2}(L)]$ in this case.

\textbf{Combining the estimates.}  Let $\epsilon>0$.  For all  $L$ sufficiently large and all $\Gamma\in\mathcal G_r$, the final estimate in the separating case shows that
$$k_\Gamma(L)\leq (1+\epsilon)\sum_{E\in\edges_0(\Gamma)\cup \edges_1'(\Gamma)}k^E_\Gamma(L).$$

Note that for any $f\in\Fast_r(L)$ constructed as above, the edge $E$ is the unique edge of $\Gamma$ realising the growth rate, so it is the unique topmost edge for $f$.   Hence
$$m_\Gamma(L)=\sum_{E\in\edges_0(\Gamma)}m^E_\Gamma(L)+\sum_{E\in\edges_1'(\Gamma)}m^E_\Gamma(L)\geq (1-\epsilon)^2\sum_{E\in\edges_0(\Gamma)\cup\edges_1'(\Gamma)}k_\Gamma^E(L).$$
So for all sufficiently large $L$, 
$$\frac{|\Fast_r(L)|}{|\ttm^r(L)|}=\frac{\sum_{\Gamma\in\mathcal G_r}m_\Gamma(L)}{\sum_{\Gamma\in\mathcal G_r}k_\Gamma(L)}\geq \frac{(1-\epsilon)^2\sum_\Gamma \sum_E k_\Gamma^E(L)}{(1+\epsilon)\sum_\Gamma\sum_E k_\Gamma^E(L)},$$
as required.
\end{proof}

\section{No exit}\label{subsec:intro-counterexample-question}
Let $G_0=\langle a,e,f,t\mid[a,t], e^{-1}te=at,f^{-1}tf=a^{-1}t\rangle.$
This is the mapping torus of the linearly growing automorphism of $F_0=\langle a,e,f\rangle$ given by $a\mapsto a,e\mapsto ea,f\mapsto fa^{-1}$.  To compute the torus splitting, rewrite the presentation as 
$$G_0=\langle a,r,s,t,e,f\mid[a,t],[r,t],[s,t], eae^{-1}=r,faf^{-1}=s,e^{-1}te=at,f^{-1}tf=a^{-1}t\rangle,$$
which we view as a graph of groups with vertex group $G_v=\langle a,r,s,t\rangle=\langle a,r,s\rangle\times\langle t\rangle$ and two stable letters $e^{-1}$ and $f^{-1}$ respectively conjugating $\langle r,t\rangle$ and $\langle s,t\rangle$ to $\langle a,t\rangle$ using the isomorphisms $r\mapsto a,t\mapsto at$ and $s\mapsto a,t\mapsto a^{-1}t$.  The $\integers^2$ subgroup $\langle a,t\rangle$ therefore belongs to three distinct maximal $F_3\times\integers$ subgroups, all conjugate to $G_v$, so $G_0$ fails to be unbranched.

The results of \cite{Wise:tubular} give free $G_0$--actions on CAT(0) cube complexes containing walls corresponding to codimension--$1$ subgroups acting minimally on the Bass-Serre tree $T_0$ of the above double HNN extension: these walls intersect every vertex stabiliser in a horizontal subgroup, so the strong form of exiting from Definition \ref{defn:walls-exit} cannot be arranged for any $\mathcal Q$.  We first argue that this holds for any cubulation.

Lemma \ref{lem:generalised-axis} shows that if $G_0$ acts freely on a cube complex $\widetilde X$, then there is a $\langle a,t\rangle$--invariant plane $\widetilde P$ in $\widetilde X$, and each hyperplane cutting an element of $A:=\langle a,t\rangle$ intersects $\widetilde P$ in a line.  Hence there are lines $\{L_i^j:1\leq i\leq n,\ 1\leq j\leq m_i\}$ in $\widetilde P$ such that $L_i^j$ intersects $L_{i'}^{j'}$ if and only if $i\neq i'$, and the $A$--translates of this family of lines account for all hyperplanes crossing $\widetilde P$.  The subgroups $A,e^{-1}Ae,f^{-1}Af$ must be cut by exactly the same hyperplanes, so the hyperplanes intersect $e\widetilde P$ (resp. $f\widetilde P$) in $eAe^{-1}$--translates (resp. $fAf^{-1}$--translates) of the $eL_i^j$ and $fL_i^j$.  

Let $L_i^j$ be parallel to the vector $(x_i,y_i)$, whose components can be chosen to be integers since $A$ acts on $\widetilde P$ cocompactly; assume $y_i\geq 0$ for all $i$.  Here we view $a$ as the translation $(1,0)$ and $t$ as $(0,1)$.  Let $L_i$ be the line through the origin parallel to $L_i^j$.  The relations in the presentation correspond to the linear maps 
$$\phi_+=\begin{pmatrix}1&1\\0&1\end{pmatrix}\ \   \text{and}\ \ \phi_-=\begin{pmatrix}1&-1\\0&1\end{pmatrix}.
$$
Let $(x_i,y_i)$ be an integer vector parallel to $L_i$ with $y_i\ge 0$.  Then for all integers $w,z$, the intersection numbers of $\phi_{\pm}(w,z)$ with the collection $\{(x_i,y_i)\}_i$ must equal those for $(w,z)$, so
$$\sum_i|y_i-x_i|=\sum_i|y_i+x_i|=\sum_i|x_i|.$$
If any $x_i=0$, then $\sum_{x_i=0}|y_i|>0$, in which case 
$$\sum_{x_i,y_i\neq 0}|y_i\pm x_i|<\sum_{x_i,y_i\neq0}|x_i|,$$
which implies $x_i\neq 0$ for all $i$. This shows that if $H$ is a hyperplane cutting $\langle a,t\rangle$, then $\stabilizer_{G_0}(H)$ acts minimally on $T_0$,  as in the cubulation from \cite{Wise:tubular}.  In other words:

\begin{prop}\label{prop:no-exit}
For any finite set $\mathcal Q\subseteq G_0$ of future attaching elements, no free $G_0$--action on a CAT(0) cube complex has all walls exiting $\mathcal Q$.
\end{prop}

The raises the question of whether one can cubulate quadratic mapping tori where $G_0$ is the terminal vertex group in the superlinear hierarchy.

\begin{prob}\label{prob:super-dehn}
Let $\mathcal H$ be the set of hyperplanes that cut $\langle a,t\rangle$ and let $\mathcal V$ be the hyperplanes that fail to cut $\langle a,t\rangle$ and hence fail to cut any vertex stabiliser.  By passing to $\widetilde X\times T_0$, assume that $\mathcal V\neq\emptyset$. Let $\mathcal Q\subset G_0$ be a finite set of future attaching elements.  Can $\widetilde X$ be chosen so that no hyperplane in $\mathcal H$ cuts any $qt\in\mathcal Q$?  Perhaps Dehn twists on the graph of groups are helpful.
\end{prob}

If the answer is yes, then the $G_0$--action on the resulting $\widetilde X$ does not have the exit property, but does satisfy a strong form of wall-independence: first, $\mathcal Q$ is wall-independent just because it is wall-independent for the action on $T_0$; second, for each $qt\in\mathcal Q$, there is \emph{no wall} that cuts both $t$ and $qt$.  By turning $\mathcal V$--hyperplanes (essentially a simplified form of Theorem \ref{thm:cubulation-with-turns} applied to the action on $T_0$), and subdividing, we can equalise the translation lengths of $t$ and all $qt$.  Corollary \ref{cor:cyclic-splitting-turns} then applies to the group $G_1$ arising as a multiple HNN extension of $G_0$ with stable letters conjugating $t$ to the various $qt$.  So a positive solution to Problem \ref{prob:super-dehn} would suggest that the unbranching hypothesis in Theorem \ref{thm:main} is unnecessary.  A negative solution would suggest that $\mathcal Q$ can be chosen so that the quadratic mapping torus $G_1$ is not cubulated.

\bibliographystyle{alpha}
\bibliography{poly-cube}
\end{document}